\documentclass[a4paper,reqno,11pt]{amsart}

\usepackage{amsmath, amsfonts, amssymb, amsthm, amscd}
\usepackage{graphicx}
\usepackage{psfrag}
\usepackage{perpage}
\usepackage{url}
\usepackage{color}
\usepackage{mathrsfs}
\usepackage{mathabx}
\usepackage{scalerel}
\usepackage{tikz}\usepackage{caption,subcaption}
\usetikzlibrary{arrows,decorations.pathmorphing,backgrounds,positioning,fit,petri}
\usepackage{subdepth}

\usepackage{dsfont} 

\usepackage[utf8]{inputenc}
\usepackage[T1]{fontenc}
\usepackage{microtype}

\usepackage[a4paper,scale={0.72,0.74},marginratio={1:1},footskip=7mm,headsep=10mm]{geometry}

\usepackage{hyperref}

\makeatletter
\def\@secnumfont{\bfseries\scshape}

\def\section{\@startsection{section}{1}
  \z@{.9\linespacing\@plus\linespacing}{.5\linespacing}%
  {\normalfont\large\bfseries\scshape\centering}}

\def\subsection{\@startsection{subsection}{2}%
  \z@{.5\linespacing\@plus.7\linespacing}{-.5em}%
  {\normalfont\bfseries\scshape}}

\def\subsubsection{\@startsection{subsubsection}{3}%
  \z@{.5\linespacing\@plus.7\linespacing}{-.5em}%
  {\normalfont\scshape}}

\def\specialsection{\@startsection{section}{1}%
  \z@{\linespacing\@plus\linespacing}{.5\linespacing}%
  {\normalfont\centering\large\bfseries\scshape}}
\makeatother

\makeatletter

\renewenvironment{proof}[1][\proofname]{\par
\pushQED{\qed}%
\normalfont \topsep4\p@\@plus4\p@\relax
\trivlist
\item[\hskip\labelsep
\bfseries
#1\@addpunct{.}]\ignorespaces
}{%
\popQED\endtrivlist\@endpefalse
}
\makeatother

\makeatletter
\newcommand \Dotfill {\leavevmode \leaders \hb@xt@ 6pt{\hss .\hss }\hfill \kern \z@}
\makeatother

\makeatletter
\def\@tocline#1#2#3#4#5#6#7{\relax
  \ifnum #1>\c@tocdepth 
  \else
    \par \addpenalty\@secpenalty\addvspace{#2}%
    \begingroup \hyphenpenalty\@M
    \@ifempty{#4}{%
      \@tempdima\csname r@tocindent\number#1\endcsname\relax
    }{%
      \@tempdima#4\relax
    }%
    \parindent\z@ \leftskip#3\relax \advance\leftskip\@tempdima\relax
    \rightskip\@pnumwidth plus4em \parfillskip-\@pnumwidth
    #5\leavevmode\hskip-\@tempdima
      \ifcase #1
       \or\or \hskip 1.65em \or \hskip 3.3em \else \hskip 4.95em \fi%
      #6\nobreak\relax
    \Dotfill
    \hbox to\@pnumwidth{\@tocpagenum{#7}}\par
    \nobreak
    \endgroup
  \fi}
\makeatother

\makeatletter
\def\l@section{\@tocline{1}{0pt}{1pc}{}{\scshape}}
\renewcommand{\tocsection}[3]{%
\indentlabel{\@ifnotempty{#2}{\ignorespaces#1 #2.\hskip 0.7em}}#3}
\def\l@subsection{\@tocline{2}{0pt}{1pc}{5pc}{}}

\def\l@subsubsection{\@tocline{3}{0pt}{1pc}{7pc}{}}

\makeatother

\numberwithin{equation}{section}

\newtheoremstyle{mytheorem}{.7\linespacing\@plus.3\linespacing}{.7\linespacing\@plus.3\linespacing}%
     {\itshape}
     {}
     {\bfseries}
     {. }
     {0.3ex}
     {\thmname{{\bfseries #1}}\thmnumber{ {\bfseries #2}}\thmnote{ (#3)}}  

\theoremstyle{mytheorem}

\newtheorem{theorem}{Theorem}[section]
\newtheorem{lemma}[theorem]{Lemma}
\newtheorem{proposition}[theorem]{Proposition}

\newtheorem{remark}[theorem]{Remark}

\newcommand{\bbE}{{\ensuremath{\mathbb E}} }

\newcommand{\bbP}{{\ensuremath{\mathbb P}} }

\newcommand{\bbT}{{\ensuremath{\mathbb T}} }

\newcommand{\cI}{{\ensuremath{\mathcal I}} }

\newcommand{\cM}{{\ensuremath{\mathcal M}} }

\newcommand{\cR}{{\ensuremath{\mathcal R}} }

\newcommand{\cW}{{\ensuremath{\mathcal W}} }

\newcommand{\cZ}{{\ensuremath{\mathcal Z}} }

\DeclareMathSymbol{\leqslant}{\mathalpha}{AMSa}{"36} 
\DeclareMathSymbol{\geqslant}{\mathalpha}{AMSa}{"3E} 
\DeclareMathSymbol{\eset}{\mathalpha}{AMSb}{"3F}     

\newcommand{\sumtwo}[2]{\sum_{\substack{#1 \\ #2}}} 

\newcommand{\be}{\begin{equation}}
\newcommand{\ee}{\end{equation}}

\newcommand{\R}{\mathbb{R}}

\newcommand{\Z}{\mathbb{Z}}
\newcommand{\N}{\mathbb{N}}

\def\bs{\boldsymbol}

\newcommand{\PEfont}{\mathrm}

\newcommand{\p}{\ensuremath{\PEfont P}}

\newcommand{\E}{\ensuremath{\PEfont  E}}
\renewcommand{\P}{\p}

\newcommand\bP{\ensuremath{\bs{\mathrm{P}}}}
\newcommand\bE{\ensuremath{\bs{\mathrm{E}}}}

\DeclareMathOperator{\bbvar}{\ensuremath{\mathbb{V}ar}}
\DeclareMathOperator{\bbcov}{\ensuremath{\mathbb{C}ov}}

\newcommand{\ind}{\mathds{1}}

\newcommand{\eps}{\varepsilon}
\renewcommand{\epsilon}{\varepsilon}
\renewcommand{\theta}{\vartheta}
\renewcommand{\rho}{\varrho}

\newenvironment{myenumerate}{%
\renewcommand{\theenumi}{\arabic{enumi}}%
\renewcommand{\labelenumi}{{\rm(\theenumi)}}%
\begin{list}{\labelenumi}
	{%
	\setlength{\itemsep}{0.4em}%
	\setlength{\topsep}{0.5em}%
	\setlength\leftmargin{2.45em}%
	\setlength\labelwidth{2.05em}%
	\setlength{\labelsep}{0.4em}%
	\usecounter{enumi}%
	}%
	}%
{\end{list}
}

{\end{list}
}

{\end{list}
}

\renewenvironment{enumerate}{
\begin{myenumerate}}%
{\end{myenumerate}}

\newenvironment{myitemize}{%
\begin{list}{$\bullet$}%
 	{%
	\setlength{\itemsep}{0.4em}%
	\setlength{\topsep}{0.5em}%
	\setlength\leftmargin{2.65em}%
	\setlength\labelwidth{2.65em}%
	\setlength{\labelsep}{0.4em}%
	}%
	}%
{\end{list}}

\renewenvironment{itemize}{
\begin{myitemize}}%
{\end{myitemize}}

\MakePerPage[2]{footnote} 

\date{\today}

\newcommand\dd{\mathrm{d}}

\newcommand\sfG{\mathsf G}

\newcommand\sfQ{\mathsf Q}

\newcommand\sfT{\mathsf T}

\newcommand\sfX{\mathsf X}

\newcommand\sfs{\mathsf s}

\newcommand\bx{\boldsymbol{x}}
\newcommand\by{\boldsymbol{y}}
\newcommand\bz{\boldsymbol{z}}

\newcommand\even{\mathrm{even}}
\newcommand\odd{\mathrm{odd}}

\newcommand{\ev}[1]{[\![ #1 ]\!]}
\newcommand\scrC{\mathscr{C}}
\newcommand\scrG{\mathscr{G}}

\newcommand\scrZ{\mathscr{Z}}
\newcommand\scrX{\mathscr{X}}
\newcommand\scrM{\mathscr{M}}
\newcommand\scrK{\mathscr{K}}

\newcommand\rme{\mathrm{e}}

\newcommand\GMC{\mathrm{GMC}}

\title[2\MakeLowercase{d}SHF ain't GMC]{The critical 2d Stochastic Heat Flow is not\\
a Gaussian Multiplicative Chaos}

\begin{document}

\author[F. Caravenna]{Francesco Caravenna}
\address{Dipartimento di Matematica e Applicazioni\\
 Universit\`a degli Studi di Milano-Bicocca\\
 via Cozzi 55, 20125 Milano, Italy}
\email{francesco.caravenna@unimib.it}

\author[R. Sun]{Rongfeng Sun}
\address{Department of Mathematics\\
National University of Singapore\\
10 Lower Kent Ridge Road, 119076 Singapore
}
\email{matsr@nus.edu.sg}

\author[N. Zygouras]{Nikos Zygouras}
\address{Department of Mathematics\\
University of Warwick\\
Coventry CV4 7AL, UK}
\email{N.Zygouras@warwick.ac.uk}


\maketitle

\begin{abstract}
The critical $2d$ Stochastic Heat Flow (SHF) is a stochastic process of random measures on
$\R^2$, recently constructed in \cite{CSZ23}.
We show that this process falls outside the class of Gaussian Multiplicative Chaos (GMC), in the
sense that it cannot be realised as the exponential of a (generalised) Gaussian field. We achieve this
by deriving strict lower bounds on the moments of the SHF that are of independent interest.
\end{abstract}

\section{Introduction}
The critical $2d$ Stochastic Heat Flow (SHF) is a stochastic process
of random measures on $\R^2$, constructed in \cite{CSZ23}
as a universal limit of random polymer models. It is the natural candidate solution
of the (ill-defined) critical $2d$ Stochastic Heat Equation:
\begin{align}\label{eq:SHE}
	\partial_t u(t,x) =\frac{1}{2} \Delta u(t,x) +\beta \, \xi(t,x) \, u(t,x) \,,\qquad t>0, \ x\in\R^2,
\end{align}
where $\xi(t,x)$ represents space-time
white noise, that is a Gaussian field, delta-correlated in space and time.
The term {\it critical} refers both to the fact that \emph{dimension~$2$ is a critical dimension},
in the sense of
singular stochastic PDEs \cite{H14, GIP15} and renormalisation theory \cite{Kup14}, and that
\emph{a critical scaling in the noise strength $\beta$ is needed}, see \eqref{beta-eps} below.

The criticality of
dimension~$d=2$ for the Stochastic Heat Equation \eqref{eq:SHE}
can be seen through a scaling argument,
in the spirit of renormalisation. Indeed, in general dimension $d\geq 1$, one can note that the rescaled function
$\tilde u(t, x) := u(\epsilon^2 t, \epsilon x)$ solves
\begin{align*}
\partial_{ t} \tilde u =\frac{1}{2} \Delta \tilde u +
\beta\, \epsilon^{1-\tfrac{d}{2}} \,\tilde \xi \, \tilde u \,,
\qquad \tilde t>0, \ \tilde x\in\R^d,
\end{align*}
where $\tilde \xi(t, x) := \epsilon^{1+\frac{d}{2}} \xi(\epsilon^2 t,
\epsilon x)$ is a new space-time white noise.
One now sees that, as $\epsilon\to 0$, when
$d<2$ the multiplicative factor
$\epsilon^{1-d/2}$
attenuates the small scale effects of the noise, while these effects are amplified when $d>3$.
On the other hand, when $d=2$,
the exponent $1-\tfrac{d}{2}$ vanishes and the extent to which
the noise influences the solution is not apparent.

\smallskip

In this paper we obtain explicit lower bounds on the moments of the SHF.
Besides their own interest, these bounds
imply that \emph{the SHF is not the ``exponential of a Gaussian field''}
in the sense of Gaussian Multiplicative Chaos (GMC).
This result provides insight on the
critical $2d$ Kardar-Parisi-Zhang (KPZ) equation:
\begin{align}\label{eq:KPZ}
	\partial_t h(t,x) =\frac{1}{2} \Delta h(t,x) +
	\frac{1}{2} |\nabla h(t,x)|^2 + \beta \, \xi(t,x) \,,\qquad t>0, \ x\in\R^2\,.
\end{align}
Indeed, when the solution $u(t,x)$ of the
Stochastic Heat Equation \eqref{eq:SHE} is function valued,
its logarithm $h(t,x) := \log u(t,x)$ is
a solution of the KPZ equation \eqref{eq:KPZ}. Since the critical $2d$ SHF
is the candidate solution of \eqref{eq:SHE}, the fact that it is not a GMC suggests
that \emph{the critical $2d$ KPZ solution
(yet to be constructed) is likely not a Gaussian field}.

\smallskip

In the rest of this introduction, we first recall the construction of the SHF
from \cite{CSZ23}; then we state our main results on the moments of the SHF
and the comparison with GMC; finally, we discuss related results from
the literature and outline future directions of research.

\subsection{Reminder: the critical $2d$ SHF}\label{S:SHF}
A key difficulty in making sense of equation \eqref{eq:SHE} is that its solution $u(t,x)$ is expected
to be a genuine distribution on $\R^2$, not a function,
so the product $\xi(t,x) \, u(t,x)$ is not well-defined.
A natural strategy to make sense of it is to
\begin{enumerate}
\item \emph{regularise the equation},
so that a well-defined approximating solution exists;
\item prove that the approximating solution \emph{has a non-trivial limit} as
the regularisation is removed (and the noise strength $\beta$ is suitably rescaled).
\end{enumerate}

This approach was recently carried out in \cite{CSZ23},
where equation \eqref{eq:SHE} is regularised via discretisation of space and time, i.e.\
white noise $\xi$ is replaced by a family
of i.i.d. random variables $\omega=(\omega(n,x))_{n\in \N, x\in \Z^2}$ with law $\bbP$,
called \emph{disorder}, which satisfy
\begin{equation}\label{eq:lambda}
\begin{gathered}
	\bbE[\omega]=0 \,, \qquad \bbE[\omega^2]=1 \,, \qquad
	\exists\, \beta_0>0: \quad 	\lambda(\beta)
	:= \log \bbE[\rme^{\beta\omega}] < \infty \quad \forall \beta \in [0, \beta_0] \,.
\end{gathered}
\end{equation}
Replacing derivatives in the Stochastic Heat Equation \eqref{eq:SHE} by suitable difference
operators,
the solution is the partition function of
\emph{directed polymers in random environment}:
\begin{align} \label{eq:paf}
	Z_{M,N}^{\beta}(x,y) &=
	\E\bigg[ \rme^{\sum_{n=M+1}^{N-1}
	\{\beta \omega(n,S_n) - \lambda(\beta)\}} \,\, \ind_{\{S_N=y\}}
	\,\bigg|\, S_M = x \bigg] \,,
\end{align}
where $\E$ is the expectation with respect to $S=(S_n)$,
the simple random walk on $\Z^2$.
Note that \eqref{eq:paf} is a discretised Feynman-Kac formula
for \eqref{eq:SHE} on the time interval $[M,N]$, up to time-reversal and
with a delta initial condition at time~$M$.
An alternative regularisation of \eqref{eq:SHE}, via mollification, is discussed
in Subsection~\ref{sec:background} below.

The main result of \cite{CSZ23} is that the random field of partition
functions $Z_{M,N}^{\beta}(x,y)$, under diffusive
rescaling of space and time and for a suitable critical scaling of $\beta = \beta_N$,
converges in law as $N\to\infty$ to a unique \emph{measure valued
random field} $\mathscr{Z}_{s,t}^\theta(\dd x , \dd y)$.
More precisely, we define the diffusively rescaled random field of partition
functions:\footnote{The factor  $\frac{1}{4}$ in \eqref{eq:rescZmeas} is due to the periodicity
of the simple random walk,
while the multiplication by $N$ is due to the local limit theorem: $\bbE[Z_{M,N}^{\beta_N}(w,z)]
= \P(S_{N} = z \,|\, S_{M} = w) = O(\frac{1}{N-M}) = O(\frac{1}{N})$
for $M/N \le c < 1$.}
\begin{equation} \label{eq:rescZmeas}
	\cZ^{\beta}_N = \bigg(\cZ^{\beta}_{N;\, s,t}(\dd x, \dd y) :=
	\frac{N}{4} \, Z_{\ev{Ns}, \ev{Nt}}^{\beta,\,\omega}(\ev{\sqrt{N} x},
	\ev{\sqrt{N} y}) \, \dd x \, \dd y
	\bigg)_{0 \le s \le t < \infty}
\end{equation}
where $\dd x \, \dd y$ is the Lebesgue measure on $\R^2 \times \R^2$ and
$\ev{Ns}$ is the even integer closest to $Ns$, while $\ev{\sqrt{N}x}$
is the point
closest to $\sqrt{N}x\in\R^2$ in the lattice $\Z^2_{\rm even}$, where
we set
\begin{equation} \label{eq:even}
	\Z^d_{\rm even}:=\{(z_1, \ldots, z_d)\in \Z^d \colon
	z_1+\ldots+ z_d \text{ is even}\} \,.
\end{equation}

We next rescale $\beta = (\beta_N)_{N\in\N}$ in  a
\emph{critical window},
defined by \eqref{intro:sigma}-\eqref{eq:RN0} in Appendix~\ref{sec:crit-wind},
which separates the weak and strong disorder phases of $2d$ directed polymers \cite{CSZ17b}.
When the disorder $\omega$ has a symmetric distribution (for simplicity), this reads as follows:
\begin{equation}\label{eq:betasimple}
	\beta_N^2 = \frac{\pi}{\log N} \bigg( 1 + \frac{\rho + o(1)}{\log N}\bigg)
	\qquad \text{for some } \rho\in\R \,.
\end{equation}
To have a universal parametrisation,
our results will be formulated using a slightly different parameter $\theta$,
see \eqref{intro:sigma},
which differs from $\rho$ in \eqref{eq:betasimple}
by a constant, see \cite[eq.~(1.17)]{CSZ19b}.

\smallskip

We can now state the main result of \cite{CSZ23}.

\begin{theorem}[The critical $2d$ SHF \cite{CSZ23}]\label{thm:2dSHF}
Fix $\beta_N$ in the critical window \eqref{intro:sigma}-\eqref{eq:RN0}
for a given $\theta\in \R$. The process of random measures
$\cZ^{\beta_N}_N =
(\cZ^{\beta_N}_{N;\, s,t}(\dd x, \dd y))_{0\le s \le t < \infty}$ defined in \eqref{eq:rescZmeas}
converges in finite dimensional distributions $($as $N\to\infty)$ to a {\it unique} limit
\begin{equation*}
	\mathscr{Z}^\theta = (\mathscr{Z}_{s,t}^\theta(\dd x , \dd y))_{0 \le s \le t <\infty} \,,
\end{equation*}
called the \emph{critical $2d$ Stochastic Heat Flow}.
\end{theorem}

The convergence in distribution in Theorem~\ref{thm:2dSHF}
takes place in the space of
locally finite measures on $\R^2 \times \R^2$,
equipped with the topology of vague convergence:
\begin{equation*}
	\mu_N \rightarrow \mu \quad \ \iff \quad \
	\int \phi(x,y) \, \mu_N(\dd x, \dd y) \to \int \phi(x,y) \, \mu(\dd x, \dd y)
	\quad \forall \phi\in C_c(\R^2\times \R^2) \,.
\end{equation*}

\subsection{Main result I: SHF vs. GMC}

We focus on the SHF's one-time marginal:
\begin{equation} \label{eq:SHF1}
	\mathscr{Z}_{t}^\theta(\dd x) :=
	\int_{y \in \R^2} \mathscr{Z}_{0,t}^\theta(\dd x , \dd y)  \,.
\end{equation}
This is a stochastic process of \emph{log-correlated random measures on $\R^2$},
see \eqref{eq:second}-\eqref{eq:rhot} below.
Higher moments of the SHF admit explicit series expansions,
see \eqref{eq:Zmomh}-\eqref{eq:Zmomh-kernel} below,
which stem from the works \cite{GQT21, Che21, CSZ19b, BC98}.
However, as we will show below,
the SHF moments grow too fast to uniquely determine the field.

In the subcritical regime $\beta_N^2 \sim \hat\beta \, \pi / \log N$ with $\hat\beta < 1$
--- that is, strictly below the critical window \eqref{eq:betasimple} that we consider here ---
the logarithm of the directed polymer partition function
displays Gaussian fluctuations  \cite{CSZ17b,Gu20,CSZ20}.
This suggests that, \emph{in the subcritical regime}, partition functions
should be close to the exponential of a Gaussian field.

It is natural to wonder whether a similar picture still holds true \emph{at criticality}:
is the critical $2d$ SHF the exponential of a Gaussian field in the sense of
{\it Gaussian Multiplicative Chaos} (GMC)?
Our first main result shows that this is not the case.

\begin{theorem}\label{thm:main}
The critical $2d$ Stochastic Heat Flow is not a Gaussian Multiplicative Chaos.
\end{theorem}

We will recall the definition of GMC in Section~\ref{GMC-recap}.
We point out that GMC has been studied extensively and has applications in many contexts,
including Liouville quantum gravity,
turbulence, zeroes of characteristic polynomials etc.
A comprehensive review of its connections to various
fields in probability and mathematical physics, as well as a nice introduction to its properties,
is given in \cite{RV14}.

\smallskip

Theorem~\ref{thm:main} suggests
that, in the critical window \eqref{eq:betasimple},
the logarithm of the partition functions has a \emph{non-Gaussian limit}.
Such a limit would then be the natural candidate solution of the critical $2d$ KPZ equation \eqref{eq:KPZ}.
Of course, putting this conjecture on firm ground
will require further work --- we cannot just take
the logarithm of the SHF, which is a random measure --- but our results provide
an indication for the emergence of non-Gaussianity in the $2d$ KPZ equation.

It is also an interesting question whether the critical $2d$ Stochastic Heat Flow is absolutely continuous w.r.t.\ {\em some} GMC. Our current
techniques (based on comparison of moments) seem insufficient to resolve this question.

\subsection{Main results II: lower bounds for the SHF moments}

Our next main results are explicit lower bounds
on the moments of the critical $2d$ SHF.
These bounds are the key to proving Theorem~\ref{thm:main},
because they show that
\emph{the moments of the SHF are strictly larger than those of a corresponding GMC},
in a sense that we now make precise.

The one-time marginal  $\mathscr{Z}_{t}^\theta(\dd x)$ of the SHF,
see \eqref{eq:SHF1}, is a random measure on $\R^2$.
Let us denote by $\mathscr{M}_t^\theta(\dd x)$ the GMC on $\R^2$
with the same first and second moments of the SHF:
\begin{align}
	\label{eq:first}
	\bbE\big[\mathscr{M}_t^\theta(\dd x)\big] &= \bbE\big[\mathscr{Z}_{t}^\theta(\dd x)\big]
	= \tfrac{1}{2} \, \dd x \,, \\
	\label{eq:second}
	\bbE\big[\mathscr{M}_t^\theta(\dd x) \, \mathscr{M}_t^\theta(\dd y)\big]
	& = \bbE\big[ \mathscr{Z}_{t}^\theta(\dd x) \,
	\mathscr{Z}_{t}^\theta(\dd y) \big]
	= \tfrac{1}{4}\, \scrK_{t,\theta}^{(2)}(x,y) \, \dd x \, \dd y \,,
\end{align}
where $\scrK_{t,\theta}^{(2)}(x,y)$ is known, see\ \eqref{eq:Zmom2},
and it is \emph{log-divergent} along the diagonal (see \eqref{eq:ktasy}):
\begin{equation}\label{eq:rhot}
	\scrK_{t,\theta}^{(2)}(x,y)
	\sim C_{t,\theta} \, \log\frac{1}{|y-x|} \qquad \text{as } \, |y-x| \to 0 \,.
\end{equation}
As will be noted after \eqref{eq:wefix},
the Gaussian field underlying such a GMC is \emph{log-log-correlated},
i.e.\ its covariance kernel satisfies $k_t(x, y)\sim \log \log \frac{1}{|y-x|}$
as $|y-x|\to 0$.\footnote{For a comparison,
the much studied Gaussian Free Field on $\R^2$ is  log-correlated, hence the corresponding
GMC is polynomially correlated.}

\smallskip

We first compare the third moment of the SHF
$\mathscr{Z}_{t}^\theta(\varphi) :=
\int_{\R^2} \varphi(x) \, \mathscr{Z}_{t}^\theta(\dd x)$
with that of the GMC $\mathscr{M}_{t}^\theta(\varphi) :=
\int_{\R^2} \varphi(x) \, \mathscr{M}_{t}^\theta(\dd x)$
averaged over integrable functions $\varphi : \R^2 \to \R$.

\begin{theorem}[Third moment lower bound]\label{th:3mom}
For $t > 0$ and $\theta \in \R$,
let $\mathscr{M}_t^\theta(\dd x)$ be the GMC with the same first and second moments
as the SHF $\mathscr{Z}_t^\theta(\dd x)$, see \eqref{eq:first}-\eqref{eq:second}.
If $\varphi$ is the indicator function
of a ball, or the heat kernel on $\R^2$, see \eqref{eq:gt}, we have
\begin{equation} \label{eq:lb3-SHF-GMC}
	\bbE\big[ \mathscr{Z}_{t}^\theta(\varphi)^3 \big] >
	\bbE\big[ \mathscr{M}_{t}^\theta(\varphi)^3 \big] \,,
\end{equation}
hence $\mathscr{Z}_{t}^\theta(\dd x) \ne \mathscr{M}_{t}^\theta(\dd x)$.
\end{theorem}

\begin{remark}
The bound \eqref{eq:lb3-SHF-GMC} actually holds for all radially symmetric and decreasing
functions $\varphi$ that satisfy a basic inequality, see \eqref{eq:goalGG} below.
These include, in particular, the indicator function of a ball and the heat kernel,
that we single out in Theorem~\ref{th:3mom}.
\end{remark}

We next turn to moments of any order $m \ge 3$.
Since $\scrM_t^\theta(\dd x)$
is a GMC with a log-divergent second moment kernel, see \eqref{eq:rhot}, one  can shown that
(see Proposition~\ref{P:momasy} below)
\begin{equation}\label{eq:factorization}
	\bbE\big[ \big(2\, \scrM_t^\theta(g_\delta) \big)^m \big] \sim
	\bbE\big[ \big(2\,\scrM_t^\theta(g_\delta)\big)^2 \big]^{\binom{m}{2}} \quad
	\text{as } \delta \downarrow 0 \,,
\end{equation}
where $g_\delta$ is the heat kernel on $\R^2$ at time $\delta$, the multiplicative factor~$2$ arises from \eqref{eq:first}, and the notation $\phi(\delta)\sim \psi(\delta)$ as $\delta\downarrow 0$ means $\lim_{\delta\downarrow 0} \phi(\delta)/\psi(\delta) =1$. We show that for the SHF $\scrZ_t^\theta$ this asymptotic factorisation \emph{does not hold}.

\begin{theorem}[Higher moments lower bound]\label{th:mmom}
Given any $t > 0$ and $\theta \in \R$, there exists $\eta = \eta_{t,\theta} > 0$
such that for any $h\in\N$ with $h \ge 3$ we have
\begin{equation}\label{eq:mombd0}
	\bbE\big[ \big(2\,\scrZ_t^\theta(g_\delta)\big)^h \big]
	\geq (1+\eta) \, \bbE\big[ \big(2\,\scrZ_t^\theta(g_\delta)\big)^2
	\big]^{h\choose 2}
	\qquad \forall \delta \in (0,1) \,.
\end{equation}
As a consequence, by \eqref{eq:factorization},
for any $h\in\N$ with $h \ge 3$ we have
\begin{equation} \label{eq:asystrict}
	\liminf_{\delta\downarrow 0}  \;
	\frac{\bbE\big[ \scrZ_t^\theta(g_\delta)^h \big]}{\bbE\big[ \scrM_t^\theta(g_\delta)^h \big]}
	\ge 1+\eta > 1 \,,
\end{equation}
hence $\mathscr{Z}_{t}^\theta(\dd x) \ne \mathscr{M}_{t}^\theta(\dd x)$.
\end{theorem}

\begin{remark}
For the directed polymer partition functions in the whole subcritical regime, a lower bound qualitatively
similar to \eqref{eq:mombd0}, \emph{but with $\eta = 0$},
is also valid and matches
the asymptotic behaviour of the upper bounds obtained
in \cite{LZ21, CZ21}.
\end{remark}

Theorem~\ref{th:3mom} will be proved in Section~\ref{via-integral}
by exploiting a series expansion for the moments \eqref{eq:Zmomh}-\eqref{eq:Zmomh-kernel},
which in the case of the third moment admits a renewal-type form \cite{CSZ19b},
see \eqref{eq:Zmom3}-\eqref{eq:gm}. This is quite involved and can be represented
as a series of complicated diagrams.
Through an explicit computation,
we are able to integrate out the spatial variables in these diagrams.
What remains is a multiple integral of time variables that have
monotonicity properties, which we exploit in order to obtain the lower bound~\eqref{eq:lb3-SHF-GMC}.

Theorem~\ref{th:mmom} will be proved in Section~\ref{via-gci}
via a very different approach, inspired by the work of Feng \cite{Feng}. A key role
here is played by
the Gaussian Correlation Inequality \cite{R14, LM17},
which saves us from analysing the complicated diagrammatic representation of the moments.
By means of probabilistic arguments, such as bounding the variance of suitable random variables,
we obtain the lower bound \eqref{eq:mombd0}, which then yields \eqref{eq:asystrict}.

\subsection{Background}
\label{sec:background}
We recall here some results that  led to the critical $2d$ SHF.

To regularise the $2d$ Stochastic Heat Equation \eqref{eq:SHE}, we used in Section \ref{S:SHF}
a discretisation of space and time, which led to the directed polymer partition functions.
Alternatively, one can \emph{mollify the white noise $\xi$ in space on scale $\epsilon > 0$} by defining
$\xi^\eps(t,x) := (\xi(t,\cdot) * j_\eps)(x)$,
where $j_\eps(x):=\epsilon^{-2} j(x/\epsilon)$
and $j(\cdot)$ is a smooth probability kernel,
say compactly supported.
This leads to the mollified Stochastic Heat Equation:
\begin{equation}\label{eq:SHE0}
	\partial_t u^\eps(t,x)
	= \frac{1}{2} \Delta u^\eps(t,x) + \beta \, u^\eps(t,x) \, \xi^\eps(t,x) \,.
\end{equation}
The solution admits a Feynman-Kac representation \cite{BC95,BC98}:
\begin{equation}\label{eq:FK1}
	u^\eps(t,x)  = \bE_x\Big[\rme^{\beta \int_0^t \xi^\eps(t-s, B_s) \, \dd s
	\,-\, \frac{1}{2}\beta^2 \Vert j_\eps\Vert_2^2 t}  \Big]
	\stackrel{\rm dist}{=} \bE_x\Big[\rme^{\beta \int_0^t \xi^\eps(s, B_s) \, \dd s
	\,-\, \frac{1}{2}\beta^2 \Vert j_\eps\Vert_2^2 t}  \Big] \,,
\end{equation}
where $\bE_x$ denotes expectation for a standard Brownian motion $B$ starting at $x$
(for simplicity, we consider a flat initial condition $u^\epsilon(0,x) \equiv 1$).
The goal is then to make sense of the limit of $u^\eps(\cdot, \cdot)$ as $\epsilon \to 0$,
for suitable rescaling of $\beta = \beta_\epsilon$.

\begin{remark}\label{rem:dp-she}
Comparing \eqref{eq:FK1} with \eqref{eq:paf},
we can see $u^\eps(t,x)$
as the partition function of a Brownian directed polymer in the random environment $\xi^\eps$.
Thus the two schemes of regularisation, discretisation and mollification, are conceptually
(if not technically) analogous, with the correspondence
$\epsilon \leftrightsquigarrow 1 / \sqrt{N}$
(see Appendix~\ref{sec:comparison} for more details).
Most existing results apply to both schemes
\cite{CSZ17b, CSZ19b, CSZ20}, so we will focus on the mollified Stochastic Heat Equation
in what follows.
\end{remark}

Denote by $u_\epsilon^{(\hat \beta)}(t,x)$ the solution \eqref{eq:FK1} with
$\beta =\hat \beta \, \sqrt{4\pi} / \sqrt{\log \epsilon^{-2}}$ for $\hat\beta > 0$.
A phase transition on this scale with critical point $\hat\beta_c = 1$
was first identified in \cite{CSZ17b}, where it was shown that
for any fixed $(t,x)$, the following limit in distribution holds:
\begin{align}\label{onepoint}
u_\epsilon^{(\hat\beta)}(t,x)\xrightarrow[\epsilon \to 0]{d}
\begin{cases}
	\rme^{\sigma(\hat\beta) \, \sfX \,-\, \tfrac{1}{2}\sigma(\hat\beta)^2}
	& \text{ if} \ \hat\beta < 1 \,, \\
	\rule{0pt}{1.3em}0 & \text{ if} \ \hat\beta \geq 1 \,,
\end{cases}
\end{align}
where $\sfX$ is a standard normal random variable and $\sigma(\hat\beta)^2
:= \log(1 / (1-\hat\beta^2))$.

For $\hat \beta<1$, known as the \emph{subcritical regime},
the solution $u_\epsilon^{(\hat \beta)}$
viewed as a random field,
suitably centred and normalised,
was shown in \cite{CSZ17b} to converge in distribution
to a Gaussian free field,
given by the solution $v^{(\hat\beta)}$ of the \emph{additive}
stochastic heat equation (a.k.a.\ Edwards-Wilkinson equation):
\begin{align}\label{eq:EW}
	\partial_t v^{(\hat\beta)}(t,x)
	&= \frac{1}{2} \Delta v^{(\hat\beta)}(t,x) + \sqrt{\tfrac{1}{1-\hat\beta^2}} \, \xi(t,x)
	\qquad \text{with} \quad v^{(\hat\beta)}(0,x)= 0,
\end{align}
where the noise coefficient diverges as $\hat\beta\uparrow 1$.
More precisely, if we define
\begin{align}\label{scaled-field}
\mathfrak{u}^{(\hat\beta)}_\epsilon(t,x):=\tfrac{\sqrt{\log \epsilon^{-2}}}{\sqrt{4\pi} \hat\beta}
\big( u^{(\hat\beta)}_\epsilon(t,x)-1\big),
\end{align}
then for every test function $\phi \in C_c(\R^2)$
we have $\langle \mathfrak{u}^{(\hat\beta)}_\epsilon, \phi \rangle
\xrightarrow[]{d} \langle v^{(\hat\beta)}, \phi \rangle$ as $\epsilon\to 0$.

A similar result has been established for the solution of
the mollified $2d$ KPZ equation, with
$u^{(\hat\beta)}_\epsilon(t,x)-1$ in \eqref{scaled-field}
replaced by $\log u^{(\hat\beta)}_\epsilon(t,x)
- \bbE[\log u^{(\hat\beta)}_\epsilon(t,x)]$,
see \cite{CSZ20, Gu20}. This may be viewed as an indication that,
in the subcritical regime $\hat\beta < 1$,
the solution of the mollified $2d$ Stochastic Heat Equation
is close to the exponential of a Gaussian field
(as we already discussed before Theorem~\ref{thm:main}
in the directed polymer setting). This breaks down
at criticality, as we show in Theorem~\ref{thm:main}.

\smallskip

We next review the results when $\beta = \beta_\epsilon$ is scaled in
a critical window around the critical point $\hat\beta_c = 1$,
which for the mollified Stochastic Heat Equation reads as follows:
\begin{align}\label{beta-eps}
	\beta^2_\epsilon = \frac{4\pi}{\log \epsilon^{-2}}
	\bigg( 1 + \frac{\tilde \rho +o(1)}{\log\epsilon^{-2}} \bigg)
	= \frac{2\pi}{\log \epsilon^{-1}}
	\bigg( 1 + \frac{\tilde \rho +o(1)}{\log\epsilon^{-2}} \bigg) \,.
\end{align}
Note that this is similar to \eqref{eq:betasimple} with $N = \epsilon^{-2}$
(the different factor $4\pi$ vs.\ $\pi$ is because the simple symmetric
random walk on $\Z^2$ has period $2$ and covariance matrix $\frac{1}{2}I$:
see Subsection~\ref{sec:SHF-SHE} and Appendix~\ref{sec:crit-wind}
for a more detailed comparison).

The study of the mollified Stochastic Heat Equation with $\beta = \beta_\eps$
chosen in the critical window \eqref{beta-eps} was initiated
in \cite{BC98}, where they identified the limit of the second moment of
the solution $u^\eps(t, \cdot)$, see \eqref{eq:FK1}.
Subsequently, \cite{CSZ19b} computed the limit of the third moment of $u^\eps(t, \cdot)$
and \cite{GQT21} identified the limit of all higher moments (see also the more recent work \cite{Che21}).
These results ensure that the mollified solutions $(u^\eps(t, \cdot))_{\epsilon > 0}$
are \emph{tight} as random measures on $\R^2$, hence they admit subsequential limits in distribution
as $\epsilon \downarrow 0$, and any such limit has the same moments as identified in
\cite{BC98,CSZ19b,GQT21,Che21}.
However, these moments grow too fast to uniquely determine the
limiting random measure.

Existence of a \emph{unique} limit, which was named the {\em critical $2d$ Stochastic Heat Flow},
was finally established in \cite{CSZ23}
in the directed polymer setting, i.e.\ for the solution of the Stochastic Heat Equation
regularised via discretisation.
It is expected that the same holds for the regularisation via mollification,
i.e.\ that $u^\eps(t, \cdot)$ in \eqref{eq:FK1} converges to
the critical $2d$ Stochastic Heat Flow as $\eps\downarrow 0$,
although the proof of \cite{CSZ23}
needs to be adapted.

\subsection{Future perspectives}
\label{sec:future}

We now discuss some related works and open questions.

We proved in Theorem~\ref{thm:main} that the
(one-time marginal of the) SHF,
as a random measure on $\R^2$, is not a GMC.
There is, however, a very different sense in which a GMC
structure emerges naturally.
In the Feynman-Kac formula \eqref{eq:FK1} for the solution $u^\eps(t, x)$
of the mollified Stochastic Heat Equation, the
exponent $\int_0^t \xi^\eps(s, B_s){\rm d}s$ may be viewed as
a Gaussian process (w.r.t.\ the randomness of the white noise $\xi^\eps$)
indexed by $(B_s)_{s\in [0,t]} \in C[0,t]$, the space of continuous functions defined on $[0,t]$.
As a consequence, on the path space $C[0,t]$,
we can consider the GMC measure $\cM^\eps_x(\dd B)$
defined by
\begin{equation}\label{eq:GMC-path}
	\cM^\eps_x(\dd B) :=
	\rme^{\beta \int_0^t \xi^\eps(s, B_s){\rm d}s  - \frac{1}{2}\beta^2 \Vert j_\eps\Vert_2^2 t}
	\, \cW_x(\dd B) \,,
\end{equation}
where $\cW_x(\cdot)$ denotes the Wiener measure on paths $B \in C[0,t]$ with $B_0 = x$.
Note that $u^\eps(t, x) = \cM^\eps_x(C[0,t])$ in \eqref{eq:FK1} is the total mass
of $\cM^\eps_x(\cdot)$.

This was the perspective taken in \cite{Cla19a, Cla19b},
where an analogue of the critical $2d$ directed polymer
on the diamond hierarchical lattice was studied
(see also \cite{BM20} for the Euclidean setting).
In \cite{Cla19a, Cla19b}, partition functions were shown to
have a non-trivial limit and then used to construct \emph{a family of critical continuum
polymer measures} indexed by
the analogue of $\tilde \rho$ in \eqref{beta-eps}.
Interestingly, these continuum polymer measures are
related to each other through a conditional
GMC structure, even though they cannot be defined as a GMC w.r.t.\
the analogue of the Wiener measure on the continuum hierarchical lattice.

This raises the natural question whether similar results hold for the analogue of
the critical $2d$ SHF in path space, namely, whether the measures $\cM_x^\eps$ on
$C[0,t]$ converge as $\eps\to 0$, at least when integrated over~$x$,
and whether the limits corresponding
to different $\tilde \rho$ in \eqref{beta-eps} are related to each other through a
conditional GMC structure. There is ongoing work in this direction in \cite{CM22},
where the authors study the second moment measure of subsequential limits of
$\cM_x^\eps\, {\rm d}x$ and found properties that are consistent with the conditional GMC structure.
\smallskip

Another interesting direction of research concerns
the asymptotic behavior of the moments of the critical $2d$ SHF.
Theorems~\ref{th:3mom} and~\ref{th:mmom} provide lower bounds and
it is natural to ask whether these can be improved.
The works \cite{CSZ19b, GQT21, Che21}
show that for each integer $h\geq 3$, there is a well-defined $h$-point
kernel $\scrK^{(h)}: (\R^2)^h \to \R \cup \{+\infty\}$ such that for any $\varphi\in C_c(\R^2)$,
$$
	\bbE[\scrZ_t^\theta(\varphi)^h] = \frac{1}{2^h}\,
	\idotsint_{(\R^2)^h}
	\bigg( \prod_{i=1}^h \, \varphi(x_i) \bigg) \, \scrK_t^{(h)}(x_1, \cdots, x_h) \, \dd\vec x \,,
$$
see Theorem~\ref{th:moments} below.
In light of Theorem~\ref{th:mmom} and \eqref{eq:rhot}, it is natural to conjecture that
\begin{equation}\label{eq:conj-mom}
	\scrK^{(h)}_t(x_1, \cdots, x_h) \sim \scrC_{t,\theta;h}
	\prod_{1\leq i<j\leq h} \log \frac{1}{|x_i-x_j|}
	\qquad \text{as} \quad \max_{1\leq i\leq j} |x_i-x_j| \to 0 \,,
\end{equation}
for some constant $\scrC_{t,\theta;h}> (C_{t,\theta})^{h\choose 2}$, where
$C_{t,\theta}$ is the constant which determines the asymptotic behavior of the
second moment kernel, see \eqref{eq:second}-\eqref{eq:rhot}.

\subsection{Organization of the paper}
The rest of the paper is structured as follows.

\begin{itemize}
\item In Section~\ref{SHF-recap}, we recall the moments formulas for the critical $2d$ SHF.

\item In Section~\ref{GMC-recap} we review the construction of GMC and recall its moments.

\item In Sections~\ref{via-integral} and~\ref{via-gci}
we prove our main results Theorems~\ref{th:3mom} and~\ref{th:mmom}.

\item In Appendix~\ref{sec:crit-wind} we compare the critical windows
\eqref{eq:betasimple} and \eqref{beta-eps}.
\end{itemize}

\section{Moments of the critical $2d$ SHF}
\label{SHF-recap}

In this section, we recall the moments formulas for the critical $2d$ Stochastic Heat Flow
from \cite{BC98, CSZ19a, CSZ19b, GQT21}.
We denote by $g_t(x)$ the heat kernel on $\R^2$:
\begin{equation}\label{eq:gt}
	g_t(x) := \frac{1}{2\pi t}  \rme^{-\frac{|x|^2}{2t}} \,.
\end{equation}
An important role is played by the following special function,
defined for any $\theta\in\R$:
\begin{equation} \label{g-theta}
	G_\theta(t)
	= \int_0^\infty \frac{\rme^{(\theta-\gamma) u} \, u \, t^{u-1}}{\Gamma(u+1)} \, \dd u \,,
\end{equation}
where $\gamma = -\int_0^\infty \log u \, \rme^{-u} \, \dd u \simeq 0.577$ is
the Euler-Mascheroni constant.

\begin{remark}
The function $G_\theta$ has a probabilistic interpretation.
Denote by $Y = (Y_u)_{u\ge 0}$ the \emph{Dickman subordinator},
defined as the pure jump process with L\'evy measure
$\ind_{(0,1)}(x) \, x^{-1} \, \dd x$, see \cite{CSZ19a}.
Then $G_\theta$ is  the \emph{exponentially weighted renewal density} of $Y$:
\begin{equation*}
	G_\theta(t) = \int_0^\infty \rme^{\theta u}\, \frac{\P(Y_u \in \dd t)}{\dd t} \, \dd u
	\qquad \text{for } t \in [0,1] \,.
\end{equation*}
\end{remark}

\subsection{First and second moments}

The first moment of the SHF is
\begin{equation}
	\bbE[\mathscr{Z}^\theta_{s,t}(\dd x, \dd y)]
	= \tfrac{1}{2} \, g_{\frac{1}{2}(t-s)}(y-x) \, \dd x \, \dd y \,,
\end{equation}
while its covariance is given by
\begin{equation} \label{eq:formula-cov}
\begin{aligned}
	\bbcov[\mathscr{Z}^\theta_{s,t}(\dd x, \dd y), \mathscr{Z}^\theta_{s,t}(\dd x', \dd y')]
	&= \tfrac{1}{2} \, K_{t-s}^\theta(x,x'; y, y') \, \dd x \, \dd y \, \dd x' \, \dd y' \,,
\end{aligned}
\end{equation}
where
\begin{equation}
\label{eq:m2-lim}
\begin{split}
	K_{t}^{\theta}(x,x'; y,y')
	&\,:=\, \pi \: g_{\frac{t}{4}}\big(\tfrac{y+y'}{2} - \tfrac{x+x'}{2}\big)
	\!\!\! \iint\limits_{0<s<u<t} \!\!\! g_s(x'-x) \,
	G_\theta(u-s) \, g_{t-u}(y'-y) \, \dd s \, \dd u \,.
\end{split}
\end{equation}
These formulas were derived from the asymptotic results in \cite{CSZ19a}
connected to the Dickman subordinator,
see \cite[Proposition~3.5]{CSZ23}.

We will focus on the one-time marginal $\scrZ_t(\dd x)$ of the SHF,
see \eqref{eq:SHF1}, which we also call
{\it the SHF with flat initial data}.
The first moment of the averaged field is then
\begin{equation} \label{eq:Zmom1}
	\bbE[\scrZ_{t}^\theta(\varphi)] =
	\frac{1}{2} \int\limits_{\R^2} \varphi(z) \,  \dd z \,,
\end{equation}
while its centered second moment can be derived from \eqref{eq:formula-cov}-\eqref{eq:m2-lim} and equals
\begin{equation} \label{eq:Zmom2}
\begin{split}
	\bbE\big[ \big( \scrZ_{t}^\theta(\varphi) - \bbE[\scrZ_{t}^\theta(\varphi)] \big)^2 \big]
	&= \frac{1}{4} \, \int\limits_{(\R^2)^2} \varphi(z_1) \, \varphi(z_2) \,
	K_t^{(2)}(z_1, z_2) \, \dd z_1 \, \dd z_2 \,, \\
	\text{with} \qquad K_t^{(2)}(z_1, z_2)
	&:= 	2\pi \, \iint\limits_{0 < s < u < t} g_{s}(z_1-z_2) \,  G_\theta(u-s) \,
	\dd s\, \dd u \,,
\end{split}
\end{equation}
a formula that was first derived in \cite{BC98}
in the context of the mollified Stochastic Heat Equation
(see Subsection~\ref{sec:SHF-SHE} below).

\subsection{Third moment}

The centered third moment of the SHF can be written as follows:
\begin{equation} \label{eq:Zmom3}
\begin{split}
	&\bbE\big[ \big( \scrZ^\theta_{t}(\varphi) - \bbE[\scrZ^\theta_{t}(\varphi)] \big)^3 \big]
	= \frac{1}{8} \, \int\limits_{(\R^2)^3} \varphi(z_1) \, \varphi(z_2) \, \varphi(z_3)\,
	K_t^{(3)}(z_1, z_2, z_3) \, \dd z_1 \, \dd z_2 \, \dd z_3 \,,
\end{split}
\end{equation}
where the kernel $K_t^{(3)}(z_1, z_2, z_3)$, first obtained in \cite[Theorem~1.4]{CSZ19b},
admits the following explicit but quite involved expression
(see Figure~\ref{fig:3rd} for a pictorial representation):\footnote{We remark that in \cite[eq. (1.25)]{CSZ19b}
we have $\pi^m$, whereas in \eqref{eq:K3} we have $(2\pi)^m$.
The main source of this discrepancy is a missing factor $2^{m-2}$
in \cite[eq. (1.25)]{CSZ19b}: indeed,
a factor $2 \ind_{\{(n,x) \in \Z^3\}}$ due to periodicity  was omitted
in \cite[eq. (5.40)]{CSZ19b}, which plugged in \cite[eq. (5.30)]{CSZ19b} yields
a factor $2$ for each $i=3,\ldots, m$, hence the claimed factor $2^{m-2}$ in \cite[eq. (1.25)]{CSZ19b}.
Since the third moment in \eqref{eq:Zmom3}
is \emph{half} the one in \cite[Theorem~1.4]{CSZ19b}, see
Remark~\ref{rem:normalization}, we have a global factor
$\frac{1}{2} \cdot 2^{m-2} = \frac{1}{8} 2^{m}$:
this turns $\pi^m$  from \cite[eq. (1.25)]{CSZ19b} into $(2\pi)^m$ in \eqref{eq:K3} and
accounts for the extra factor $\frac{1}{8}$ in \eqref{eq:Zmom3}.}
\begin{equation}\label{eq:K3}
	K_t^{(3)}(z_1, z_2, z_3) :=
	\sum_{m=2}^\infty 2^{m-1} \, (2\pi)^m \,
	\big\{ \cI_t^{(m)}(z_1, z_2, z_3) + \cI_t^{(m)}(z_2, z_3, z_1) + \cI_t^{(m)}(z_3, z_1, z_2) \big\} \,,
\end{equation}
where the kernel $\cI_t^{(m)}(z_1, z_2, z_3)$ is defined by
\begin{equation} \label{eq:Im}
\begin{split}
	\cI_t^{(m)}(z_1, z_2, z_3)
	:= \!\!\!\!\!\!\!\!
	\idotsint\limits_{0 < a_1 < b_1 < \ldots < a_m < b_m < t} \!\!\!\!
	& \textbf{g}^{(m)}_{a_1,b_1,\ldots,a_m,b_m}(z_1, z_2, z_3) \,
	\bigg\{ \prod_{\ell=1}^m G_\theta(b_\ell-a_\ell) \bigg\} \, \dd \vec{a} \, \dd \vec{b} \,,
\end{split}
\end{equation}
and $\textbf{g}_{a_1,b_1,\ldots,a_m,b_m}(z_1, z_2, z_3)$ denotes the following
convolution of heat kernels:
\begin{equation} \label{eq:gm}
\begin{split}
	\textbf{g}^{(m)}_{a_1,b_1,\ldots,a_m,b_m}(z_1, z_2, z_3)
	\, := \!\!\!\!\!\!\!\!\!\iint\limits_{(\R^2)^{m} \times (\R^2)^{m}} \!\!\!\!\!\!\!\!
	\dd \vec x \, \dd \vec y \ \,
	g_{\frac{a_1}{2}}(x_1-z_1) \, g_{\frac{a_1}{2}}(x_1-z_2) \cdot
	g_{\frac{b_1-a_1}{4}}(y_1-x_1) & \\
	\cdot \,  g_{\frac{a_2}{2}}(x_2-z_3)
	\, g_{\frac{a_2-b_1}{2}}(x_2-y_1) \cdot g_{\frac{b_2-a_2}{4}}(y_2-x_2) & \\
	\cdot \prod_{\ell=3}^m
	\Big\{ g_{\frac{a_\ell-b_{\ell-2}}{2}}(x_\ell-y_{\ell-2}) \, g_{\frac{a_\ell-b_{\ell-1}}{2}}(x_\ell-y_{\ell-1})
	\cdot g_{\frac{b_\ell-a_\ell}{4}}(y_\ell - x_\ell) \Big\} &
\end{split}
\end{equation}
(we agree that $\prod_{\ell=3}^m\{\ldots\} := 1$ for $m=2$).
We refer again to Figure~\ref{fig:3rd}.

\smallskip

We stress that formulas \eqref{eq:Zmom3}-\eqref{eq:gm}
are the key to our proof of Theorem~\ref{th:3mom}.

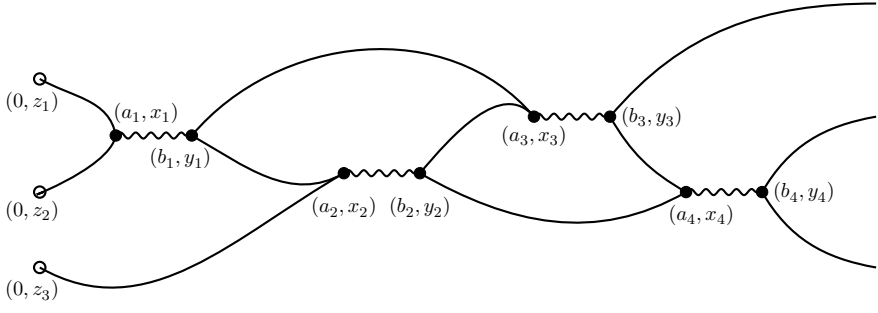
\begin{figure}
\begin{tikzpicture}[scale=0.5]
\draw  [fill] (0, -0.5)  circle [radius=0.15]; \node at (0,-1.4) {\scalebox{0.7}{$(a_2,x_2)$}};
\draw [-,thick, decorate, decoration={snake,amplitude=.4mm,segment length=2mm}] (0,-0.5) -- (2,-0.5);
\draw  [fill] (2, -0.5)  circle [radius=0.15]; \node at (2,-1.4) {\scalebox{0.7}{$(b_2,y_2)$}};
\draw[thick] (2,-0.5) to [out=50,in=130] (5,1);
\draw  [fill] (5, 1)  circle [radius=0.15]; \node at (5,0.5) {\scalebox{0.7}{$(a_3,x_3)$}};
 \draw [-,thick, decorate, decoration={snake,amplitude=.4mm,segment length=2mm}] (5,1) -- (7,1);
\draw  [fill] (7, 1)  circle [radius=0.15]; \node at (8.1,1.0) {\scalebox{0.7}{$(b_3,y_3)$}};
\draw[thick] (-4,0.5) to [out=50,in=130] (5,1);
\draw  [fill] (-4,0.5)  circle [radius=0.15]; \node at (-4.3,-0.1) {\scalebox{0.7}{$(b_1,y_1)$}};
 \draw [-,thick, decorate, decoration={snake,amplitude=.4mm,segment length=2mm}] (-6,0.5) -- (-4,0.5);
 \draw[thick] (-7.9,-1) to [out=200,in=260] (-6,0.5);
  \draw  [thick] (-8,-1)  circle [radius=0.15];  \node at (-8.2,1.4) {\scalebox{0.7}{$(0,z_1)$}};
   \draw[thick] (-8,2) to [out=-30,in=100] (-6,0.5);
   \draw  [thick] (-8,2)  circle [radius=0.15];  \node at (-8.2,-1.5) {\scalebox{0.7}{$(0,z_2)$}};
 \draw  [fill] (-6,0.5)  circle [radius=0.15]; \node at (-5.2,1.1) {\scalebox{0.7}{$(a_1,x_1)$}};
 \draw[thick] (-4,0.5) to [out=-30,in=210] (0,-0.5);
  \draw[thick] (-8,-3) to [out=-30,in=210] (0,-0.5); \node at (-8.2,-3.6) {\scalebox{0.7}{$(0,z_3)$}};
   \draw  [thick] (-8,-3)  circle [radius=0.15];
  \draw[thick] (2,-0.5) to [out=-30,in=210] (9,-1);
   \draw  [fill] (9,-1)  circle [radius=0.15];  \node at (9.4,-1.6) {\scalebox{0.7}{$(a_4,x_4)$}};
    \draw[thick] (7,1) to [out=-60,in=150] (9,-1);
    \draw [-,thick, decorate, decoration={snake,amplitude=.4mm,segment length=2mm}] (9,-1) -- (11,-1);
     \draw  [fill] (11,-1)  circle [radius=0.15];  \node at (12.1,-1.0) {\scalebox{0.7}{$(b_4,y_4)$}};
      \draw[thick] (11,-1) to [out=60,in=190] (14,1);
       \draw[thick] (11,-1) to [out=-60,in=-190] (14,-3);
       \draw[thick] (7,1) to [out=50,in=180] (14,4);
\end{tikzpicture}
\caption{Graphical representation of the kernel $K_t^{(3)}(z_1, z_2, z_3)$
for the centered third moment, see \eqref{eq:K3}-\eqref{eq:gm}.
Solid-curved lines from $(b,y)$ to $(a,x)$
are assigned weights $g_{\frac{a-b}{2}}(x-y)$ while wiggle lines
from $(a,x)$ to $(b,y)$ are assigned weights
$G_\theta(b-a) g_{\frac{b-a}{4}}(y-x)$.}
\label{fig:3rd}
\end{figure}

\begin{remark}\label{rem:normalization}
The normalisation chosen in \cite{CSZ23} to construct the critical $2d$ SHF
is slightly different from the one in \cite{CSZ19b}
due to the restriction to even parity sites, see \eqref{eq:rescZmeas}-\eqref{eq:even}.
As a consequence, the limiting field in \cite{CSZ19b} corresponds to
$\scrZ^{\theta,\mathrm{mix}}_{t}(\varphi) \overset{d}{=}
\scrZ^\theta_{t}(\varphi) + \scrZ^{\prime, \theta}_{t}(\varphi)$,
where $\scrZ^\theta_{t}(\varphi)$ and $\scrZ^{\prime, \theta}_{t}(\varphi)$
denote two independent copies of the SHF.
It follows that
\begin{equation*}
	\bbE[ ( \scrZ^{\theta}_{t}(\varphi) - \bbE[\scrZ^{\theta}_{t}(\varphi)] )^3 ]
	= \frac{1}{2} \, \bbE[ (\scrZ^{\theta, \mathrm{mix}}_{t}(\varphi)
	- \bbE[\scrZ^{\theta, \mathrm{mix}}_{t}(\varphi)] )^3 ]  \,,
\end{equation*}
that is, the third moment
in \eqref{eq:Zmom3}
is \emph{half} of that computed in \cite[Theorem~1.4]{CSZ19b}.
\end{remark}

\subsection{Higher moments}

A formula for higher moments of the SHF was first identified in \cite{GQT21}.
For completeness, we recall this formula in our framework.

Fix an integer $h\in\N$ with $h \ge 2$.
For $t > 0$ and a pair $\{i,j\} \subset \{1,\ldots, h\}$
of \emph{distinct} elements $i < j$,
we define two measure kernels mapping from $(\R^2)^h$ to measures supported on the subspace
\begin{equation} \label{eq:R2hij}
	(\R^2)^h_{\{i,j\}} \, := \, \big\{ \bx = (x_1, \ldots, x_h) \in (\R^2)^h: \
	x_i = x_j \big\} \,.
\end{equation}
\begin{itemize}
\item The first measure kernel (actually a probability kernel) is called {\it constrained evolution}:
\begin{align}\label{eq:sfQ}
	\sfQ_{t}^{\{i,j\}}(\by, \dd \bx)
	\,:=\, \Bigg( \prod_{\ell=1}^h \,
	g_{\frac{t}{2}} (x_\ell - y_\ell)  \Bigg) \cdot
	\Bigg( \prod_{\ell \in \{1,\ldots, h\} \setminus \{i,j\}} \dd x_\ell \Bigg)
	\cdot \dd x_i \cdot \delta_{x_i}(\dd x_{j}) \,,
\end{align}
where $\delta_{x_i}(\cdot)$ denotes the Dirac mass at $x_i \in \R^2$
and $g_t(\cdot)$ is the heat kernel, see \eqref{eq:gt}.

\item The second measure kernel is called {\it replica evolution}:
\begin{align} \label{eq:sfG}
	\sfG_{\theta,t}^{\{i,j\}}(\bx, \dd \by)
	\,:=\, \Bigg( \prod_{\ell \in \{1,\ldots, h\} \setminus \{i,j\}}
	\!\!\!\!\!\!\!\! g_{\frac{t}{2}} (y_\ell - x_\ell) \, \dd y_\ell \Bigg)
	G_\theta(t) \, g_{\frac{t}{4}} (y_i - x_i) \, \dd y_i
	\cdot \delta_{y_i}(\dd y_j) \,,
\end{align}
where $G_\theta(t)$ is the function in \eqref{g-theta}.
We will only need $\sfG_{\theta,t}^{\{i,j\}}(\bx, \dd \by)$
with $x_i=x_j$.

\end{itemize}

We now give the higher moments formula.

\begin{theorem}\label{th:moments}
Fix $h\in\N$ with $h \ge 2$.
 The $h$-th moment of the SHF $\scrZ_{t}^\theta$ with flat initial data,
averaged over a test function $\varphi\in C_c(\R^2)$, admits the expression
\begin{align} \label{eq:Zmomh}
	\bbE\big[ \scrZ^\theta_{t}(\varphi)^h \big]
	&= \frac{1}{2^h}
	\int\limits_{(\R^2)^h} \varphi(z_1) \cdots \varphi(z_h) \,
	\scrK_t^{(h)}(z_1,\ldots,z_h) \, \dd z_1 \cdots \dd z_h \,,
\end{align}
with
\begin{equation} \label{eq:Zmomh-kernel}
\begin{split}
       &  \scrK_t^{(h)}(z_1,\ldots, z_h) \\
       & \quad \ := \, 1 +
       \sum_{m=1}^\infty
        (2\pi)^m
       \!\!\!\!\!
        \sumtwo{\{i_1\ne j_1\},\ldots,\{i_m \ne j_m\} \subset \{1,\ldots,h\}}
        {\rule{0pt}{.8em}\text{with }
        \{i_\ell, j_\ell\} \ne \{i_{\ell-1}, j_{\ell-1}\} \, \forall \ell \ge 2}
       \ \ \, \idotsint\limits_{0 < a_1 < b_1 < \ldots < a_m < b_m < t} \!\!\!\!\!\!\!\!\!
        \dd \vec{a} \, \dd \vec{b} \
        \idotsint\limits_{(\vec{\bx}, \, \vec{\by}) \in ((\R^2)^{h})^{2m}}   \\
        & \qquad \quad \ \
	\sfQ_{a_1}^{\{i_1,j_1\}}(\bz, \dd \bx_1) \, \sfG_{\theta, b_1-a_1}^{\{i_1,j_1\}} (\bx_1, \dd \by_1)
	\prod_{\ell=2}^{m} \sfQ_{a_\ell-b_{\ell-1}}^{\{i_\ell,j_\ell\}}(\by_{\ell-1}, \dd \bx_\ell)
	\, \sfG_{\theta, b_\ell-a_\ell}^{\{i_\ell,j_\ell\}}(\bx_\ell, \dd \by_\ell) \, .
\end{split}
\end{equation}
\end{theorem}

This result can be proved by arguing as in \cite[Section 6]{CSZ23},
exploiting the local limit theory for the Dickman
subordinator as developed in \cite{CSZ19a}.
Formula~\eqref{eq:Zmomh-kernel} coincides with the one obtained
in \cite{GQT21} up to a simple scaling, see Proposition~\ref{th:SHF-SHE} below.

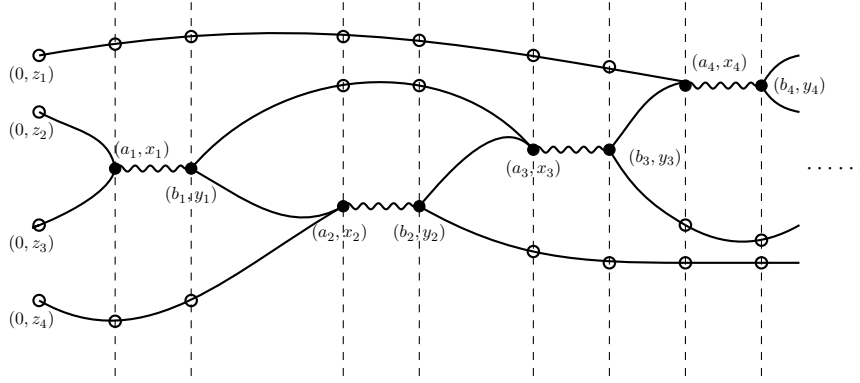
\begin{figure}
\hskip -0.2cm
\begin{tikzpicture}[scale=0.5]
\draw[dashed] (-6,-5)--(-6,5); \draw[dashed] (-4,-5)--(-4,5); \draw[dashed] (0,-5)--(0,5); \draw[dashed] (2,-5)--(2,5); \draw[dashed] (5,-5)--(5,5);
\draw[dashed] (7,-5)--(7,5); \draw[dashed] (9,-5)--(9,5);  \draw[dashed] (11,-5)--(11,5);
\draw  [thick] (-6, -3.55)  circle [radius=0.15]; \draw  [thick] (-4, -3)  circle [radius=0.15]; \draw  [thick] (0, 2.7)  circle [radius=0.15];
\draw  [thick] (2, 2.7)  circle [radius=0.15]; \draw  [thick] (5, -1.7)  circle [radius=0.15];
\draw  [thick] (7, -2)  circle [radius=0.15]; \draw  [fill] (9, 2.7)  circle [radius=0.15];
\draw  [thick] (9, -2)  circle [radius=0.15];
\draw  [thick] (-8,3.5) circle [radius=0.15];
\node at (-8.2,3.0) {\scalebox{0.6}{$(0,z_1)$}};
\draw  [thick] (-6, 3.8)  circle [radius=0.15];
\draw  [thick] (-4, 4.0)  circle [radius=0.15];
\draw  [thick] (2, 3.9)  circle [radius=0.15];
\draw  [thick] (0, 4.0)  circle [radius=0.15];
\draw  [thick] (5, 3.5)  circle [radius=0.15];
\draw  [thick] (7, 3.2)  circle [radius=0.15];
\draw  [fill] (0, -0.5)  circle [radius=0.15]; \node at (-0.1,-1.2) {\scalebox{0.6}{$(a_2,x_2)$}};
\draw [-,thick, decorate, decoration={snake,amplitude=.4mm,segment length=2mm}] (0,-0.5) -- (2,-0.5);
\draw  [fill] (2, -0.5)  circle [radius=0.15]; \node at (2,-1.2) {\scalebox{0.6}{$(b_2,y_2)$}};
\draw[thick] (2,-0.5) to [out=50,in=130] (5,1);
\draw  [fill] (5, 1)  circle [radius=0.15]; \node at (5,0.5) {\scalebox{0.6}{$(a_3,x_3)$}};
 \draw [-,thick, decorate, decoration={snake,amplitude=.4mm,segment length=2mm}] (5,1) -- (7,1);
\draw  [fill] (7, 1)  circle [radius=0.15]; \node at (8.2,0.8) {\scalebox{0.6}{$(b_3,y_3)$}};
\draw[thick] (-4,0.5) to [out=50,in=130] (5,1);
\draw  [fill] (-4,0.5)  circle [radius=0.15]; \node at (-4,-0.2) {\scalebox{0.6}{$(b_1,y_1)$}};
 \draw [-,thick, decorate, decoration={snake,amplitude=.4mm,segment length=2mm}] (-6,0.5) -- (-4,0.5);
 \draw[thick] (-8,-1) to [out=200,in=260] (-6,0.5);
  \draw  [thick] (-8,-1)  circle [radius=0.15];  \node at (-8.2,1.5) {\scalebox{0.6}{$(0,z_2)$}};
   \draw[thick] (-8,2) to [out=-30,in=100] (-6,0.5);
   \draw  [thick] (-8,2)  circle [radius=0.15];  \node at (-8.2,-1.5) {\scalebox{0.6}{$(0,z_3)$}};
 \draw  [fill] (-6,0.5)  circle [radius=0.15]; \node at (-5.3,1) {\scalebox{0.6}{$(a_1,x_1)$}};
 \draw[thick] (-4,0.5) to [out=-30,in=210] (0,-0.5);
  \draw[thick] (-8,-3) to [out=-30,in=210] (0,-0.5); \node at (-8.2,-3.5) {\scalebox{0.6}{$(0,z_4)$}};
   \draw  [thick] (-8,-3)  circle [radius=0.15];
  \draw[thick] (2,-0.5) to [out=-30,in=180] (9,-2);
   \draw  [thick] (9,-1)  circle [radius=0.15];  \node at (9.9,3.3) {\scalebox{0.6}{$(a_4,x_4)$}};
    \draw  [thick] (11,-1.4)  circle [radius=0.15];  \draw  [thick] (11,-2)  circle [radius=0.15];
    \draw[thick] (7,1) to [out=-60,in=150] (9,-1);

    \draw [-,thick, decorate, decoration={snake,amplitude=.4mm,segment length=2mm}] (9,2.7) -- (11,2.7);
    \draw  [fill] (11,2.7)  circle [radius=0.15];  \node at (12,2.7) {\scalebox{0.6}{$(b_4,y_4)$}};
    \draw[thick] (11,2.7) to [out=60,in=190] (12,3.5);
    \draw[thick] (11,2.7) to [out=-60,in=-190] (12, 2);
    \draw[thick] (7,1) to [out=50,in=190] (9,2.8);
    \draw[thick] (-8,3.5) to [out=10,in=170] (9,2.8);
     \draw[thick] (9,-1) to [out=-30,in=-150] (12,-1);  \draw[thick] (9,-2)--(12,-2);
     \node at (13,0.5) {\scalebox{0.8}{$\cdots\cdots$}};
\end{tikzpicture}
\caption{Graphical representation of the kernel $\scrK_t^{(4)}(z_1, z_2, z_3, z_4)$ for the fourth moment,
see \eqref{eq:Zmomh-kernel}. The solid-curved and wiggle lines are assigned the same weights as in Figure \ref{fig:3rd}.
The hollow circles on the vertical dashed lines are where we apply the Champman-Kolmogorov decomposition (see also Remark \ref{Remark-CHK}).} \label{figure-fourth-moment}
\end{figure}

\smallskip

\begin{remark}
The integral
over the space variables $\vec\bx, \vec\by$ in \eqref{eq:Zmomh-kernel}
can be restricted to the subspace
$\big((\R^2)^{h}_{\{i_1,j_1\}}\big)^2 \times \ldots \times  \big( (\R^2)^{h}_{\{i_m,j_m\}}
\big)^2 \subseteq ((\R^2)^{h})^{2m}$, see
\eqref{eq:R2hij}. This is because the kernels $\sfQ_{t}^{\{i,j\}}$ and $\sfG_{\theta,t}^{\{i,j\}}$
in \eqref{eq:sfQ}-\eqref{eq:sfG} are measures supported on $(\R^2)^{h}_{\{i,j\}}$.
\end{remark}

\begin{remark}
Centered moments
$\bbE\big[ \big( \scrZ^\theta_{t}(\varphi) - \bbE[\scrZ^\theta_{t}(\varphi)] \big)^h \big]$
admit formulas analogous to \eqref{eq:Zmomh}-\eqref{eq:Zmomh-kernel}, with a
correlation kernel $K^{(h)}_t(z_1,\ldots, z_h)$ which is obtained from
\eqref{eq:Zmomh-kernel} by removing the constant term ``$1+$'' and
imposing the constraint $\bigcup_{\ell=1}^m \{i_\ell,j_\ell\} = \{1,\ldots, h\}$
in the sum over $\{i_1 \ne j_1\},\ldots,\{i_m \ne j_m\} \subset \{1,\ldots,h\}$
(incidentally, this requires $m \ge \lceil \frac{h}{2}\rceil$).
\end{remark}

\begin{remark}\label{Remark-CHK}
In the special case $h=3$, formulas
\eqref{eq:Zmomh}-\eqref{eq:Zmomh-kernel} are consistent with formulas
\eqref{eq:Zmom3}-\eqref{eq:gm} for the centered third moment.
To check this, it suffices to
decompose the heat kernels $g_{\frac{a_\ell-b_{\ell-2}}{2}}(x_{\ell}-y_{\ell-2})$
in \eqref{eq:gm} at times $a_{\ell-1}, b_{\ell-1}$ by  Chapman-Kolmogorov:
\begin{equation*}
	g_{\frac{a_\ell-b_{\ell-2}}{2}}(x_{\ell}-y_{\ell-2})
	= \!\! \iint\limits_{(\R^2)^2} \!\! \dd x' \, \dd y' \,
	g_{\frac{a_{\ell-1}-b_{\ell-2}}{2}}(x'-y_{\ell-2})
	\, g_{\frac{b_{\ell-1} - a_{\ell-1}}{2}}(y'-x') \,
	g_{\frac{a_\ell-b_{\ell-1}}{2}}(x_{\ell}-y') \,,
\end{equation*}
which gives rise to the operators $\sfQ_{a_\ell-b_{\ell-1}}^{\{i_\ell,j_\ell\}}$,
$\sfG_{\theta,b_{\ell-1}-a_{\ell-1}}^{\{i_{\ell-1},j_{\ell-1}\}}$
and $\sfQ_{a_\ell-b_{\ell-1}}^{\{i_{\ell-1}, j_{\ell-1}\}}$,
see \eqref{eq:sfQ} and \eqref{eq:sfG}. See also Figure \ref{figure-fourth-moment} for
the application of Chapman-Kolmogorov (in the case $h=4$) .
\end{remark}

\section{GMC and its moments}
\label{GMC-recap}

As already mentioned in the introduction, a nice review of the Gaussian Multiplicative Chaos
(GMC) and
its various connections can be found in \cite{RV14}. Here we present its definition and the structure
of its moments, which is relevant towards our goals.

\subsection{Construction of GMC}
Let $k: \R^2 \times \R^2 \to \R \cup \{+\infty\}$ be a kernel which is
symmetric, locally integrable and positive definite, i.e.\
$\iint_{\R^2\times\R^2} \varphi(x) \, k(x,y) \, \varphi(y) \, \dd x \, \dd y \ge 0$
for all $\varphi \in C_c(\R^2)$.
Let $\scrX = (\scrX(\varphi))_{\varphi \in C_c(\R^2)}$ be
the centered Gaussian field with covariance
\begin{equation*}
	k(\varphi,\psi) := \iint\limits_{\R^2\times\R^2}
	\varphi(x) \, k(x,y) \, \psi(y) \, \dd x \, \dd y
	\qquad \text{ for $\varphi, \psi \in C_c(\R^2)$} \,.
\end{equation*}

Let us fix a locally finite measure $\mu$ on $\R^2$.
The Gaussian Multiplicative Chaos (GMC) associated to $\scrX$
with respect to the measure $\mu$,
denoted by  $\scrM(\dd x)$, is formally given by
\begin{equation*}
	\scrM(\dd x) = \, : \exp(\scrX(x)) \, \mu(\dd x) : \,.
\end{equation*}
For a precise definition, for $\epsilon > 0$ we take a continuous
regularization  $k_\epsilon(x,y)$ of $k(x,y)$, still positive definite,
such that $\lim_{\epsilon \downarrow 0} k_\epsilon(x,y) = k(x,y)$
locally uniformly in $x,y$.
We can then consider
the centered Gaussian process $\scrX_\epsilon = (\scrX_\epsilon(x))_{x\in\R^2}$
with covariance $k_\epsilon(x,y)$, which is well-defined pointwise,
and we define for $\epsilon > 0$
\begin{equation*}
	\scrM_\epsilon(\dd x) := \rme^{\scrX_\epsilon(x)
	- \frac{1}{2} \bbE[\scrX_\epsilon(x)^2]} \, \mu(\dd x)
	= \rme^{\scrX_\epsilon(x)
	- \frac{1}{2} k_\epsilon(x,x)} \, \mu(\dd x) \,.
\end{equation*}
The GMC $\scrM(\dd x)$ is then defined as the following limit in distribution:
\begin{equation*}
	\scrM(\dd x) := \lim_{\epsilon \downarrow 0} \scrM_\epsilon(\dd x) \,,
\end{equation*}
assuming that it exists in the vague sense: for $\varphi \in C_c(\R^2)$,
\begin{equation*}
	\scrM_\epsilon(\varphi) := \int\limits_{\R^2} \varphi(x) \, \scrM_\epsilon(\dd x)
	\ \xrightarrow[\epsilon\downarrow 0]{} \
	\scrM(\varphi) := \int\limits_{\R^2} \varphi(x) \, \scrM(\dd x) \,.
\end{equation*}

\subsection{Moments of GMC}
By construction, for $\epsilon > 0$ we have
\begin{equation} \label{eq:Mmom1}
	\bbE[\scrM_\epsilon(\varphi)] = \int\limits_{\R^2} \varphi(z) \, \mu(\dd z) \, .
\end{equation}
Since $\bbE[\rme^{\scrX_\epsilon(x) + \scrX_\epsilon(y)}] =
\bbE[\rme^{\frac{1}{2}\bbvar [\scrX_\epsilon(x) + \scrX_\epsilon(y)]}]
= \rme^{\frac{1}{2} \{k_\epsilon(x,x) + k_\epsilon(y,y) + 2k_\epsilon(x,y)\}}$,
we obtain
\begin{equation} \label{eq:Mmom2}
	\bbE\big[ \scrM_\epsilon(\varphi)^2\big]
	= \iint_{\R^2\times\R^2} \varphi(z_1) \, \varphi(z_2) \, \rme^{k_\epsilon(z_1,z_2)}
	\, \mu(\dd z_1) \, \mu(\dd z_2) \,.
\end{equation}
Similarly, since $\bbE[\rme^{\scrX_\epsilon(z_1) + \cdots + \scrX_\epsilon(z_m)}]
= \rme^{\frac{1}{2} \sum_{i,j=1}^m k_\epsilon(z_i,z_j)}$,
we have
\begin{equation} \label{eq:Mmom3}
\begin{split}
	& \bbE\big[ \scrM_\epsilon(\varphi)^m\big]
	= \int_{(\R^2)^m} \varphi(z_1) \cdots \varphi(z_m) \,
	\rme^{\sum_{1\leq i<j\leq m}k_\epsilon(z_i, z_j) }
	\, \mu(\dd z_1) \cdots  \mu(\dd z_m) \,.
\end{split}
\end{equation}
When we let $\epsilon \downarrow 0$, these formulas apply
to $\scrM(\varphi)$ once we replace $k_\epsilon(z_i,z_j)$
by $k(z_i,z_j)$.

Let us now record the centered second and third moments of GMC.
\begin{itemize}
\item \emph{Centered second moment}:
\begin{equation} \label{eq:Mmom2bis}
\begin{gathered}
	\bbE\big[ \big( \scrM(\varphi) - \bbE[\scrM(\varphi)] \big)^2 \big]
	= \int\limits_{(\R^2)^2} \varphi(z_1) \, \varphi(z_2) \,
	K^{(2)}_{\GMC}(z_1, z_2) \, \mu(\dd z_1) \, \mu(\dd z_2) \\
	\text{where} \qquad
	K^{(2)}_{\GMC}(z_1, z_2) := \rme^{k(z_1,z_2)}-1 \,.
\end{gathered}
\end{equation}
\item \emph{Centered third moment}:
\begin{equation} \label{eq:Mmom3bis}
\begin{gathered}
	\bbE\big[ \big( \scrM(\varphi) - \bbE[\scrM(\varphi)] \big)^3 \big]
	= \!\!\!\! \int\limits_{(\R^2)^3} \!\!\!\!
	\varphi(z_1)  \varphi(z_2)  \varphi(z_3)\,
	K^{(3)}_{\GMC}(z_1, z_2, z_3) \,
	\mu(\dd z_1) \mu(\dd z_2) \mu(\dd z_3)  \\
	\text{where} \qquad
	K^{(3)}_{\GMC}(z_1, z_2, z_3) \,:= \,
	\prod_{1 \le i < j \le 3}\rme^{k(z_i,z_{j})}
	- \sum_{1 \le i < j \le 3} \rme^{k(z_i,z_{j})} + 2 \,.
\end{gathered}
\end{equation}
\end{itemize}
Comparing \eqref{eq:Mmom3bis} with \eqref{eq:Mmom2bis},
we see that the following structural relation holds:
\begin{equation} \label{eq:GMC-3-2}
\begin{split}
	K_\GMC^{(3)}(z_1, z_2, z_3) \,=\,
	K_\GMC^{(2)}(z_1, z_2) \, K_\GMC^{(2)}(z_2, z_3) \, K_\GMC^{(2)}(z_1, z_3)
	\qquad\qquad & \\
	+ \ K_\GMC^{(2)}(z_1, z_2) \, K_\GMC^{(2)}(z_2, z_3) & \\
	\, + \ K_\GMC^{(2)}(z_1, z_2) \, K_\GMC^{(2)}(z_1, z_3)& \\
	\,+\ K_\GMC^{(2)}(z_1, z_3) \, K_\GMC^{(2)}(z_2, z_3) & \,.
\end{split}
\end{equation}

\subsection{A GMC Matching the First Two Moments of SHF}\label{ss:GMC}
Henceforth we denote by
$\mathscr{M}_{t}^\theta(\dd x)$ the GMC
with the same first and second moments
as the SHF $\scrZ_{t}^\theta(\dd x)$.
Comparing \eqref{eq:Mmom1} and \eqref{eq:Mmom2bis}
with \eqref{eq:Zmom1} and \eqref{eq:Zmom2},
we see that this can be obtained once we fix
\begin{equation} \label{eq:wefix}
	\mu(\dd x) := \frac{1}{2} \, \dd x  \,,
	\qquad
	k_t(z_1,z_2) = \log\big(1 + K^{(2)}_t(z_1, z_2) \big) \,,
\end{equation}
where $K^{(2)}_t$ is defined in \eqref{eq:Zmom2}.
This ensures that $K^{(2)}_\GMC(z_1,z_2) = K_t^{(2)}(z_1,z_2)$.\footnote{By
\eqref{eq:second}-\eqref{eq:rhot}
for the uncentered correlation function $\tfrac{1}{4}\, \scrK_{t,\theta}^{(2)}(z_1, z_2)= e^{k_t(z_1, z_2)}$,
the covariance kernel of the Gaussian field underlying the GMC satisfies
$k_t(z_1, z_2)\sim \log \log \frac{1}{|z_1-z_2|}$ as $|z_1-z_2|\to 0$.}
To show that  $\scrZ_{t}^\theta(\dd x)$ is not a GMC, it suffices to show that the higher moments
of  $\mathscr{M}_{t}^\theta(\dd x)$ and $\scrZ_{t}^\theta(\dd x)$ do not match.

\section{Proof of Theorem~\ref{th:3mom}: lower bounds via Gaussian integrals}
\label{via-integral}

In this section we prove Theorem~\ref{th:3mom}:
the third moment of the critical $2d$ SHF $\mathscr{Z}_{t}^\theta(\varphi)$
is strictly larger than that of a GMC $\mathscr{M}_{t}^\theta(\varphi)$ with matching
first and second moments,
when averaged over suitable integrable functions $\varphi: \R^2 \to [0,\infty)$.

\begin{remark}\label{rem:radially}
Most steps of our analysis cover any integrable function
$\varphi: \R^2 \to [0,\infty)$ which is
\emph{radially symmetric and non-increasing}, that is
$\varphi(x) = \rho(|x|)$ for some non-increasing function
$\rho: [0,\infty) \to [0,\infty]$, with $|\cdot|$ the Euclidean norm.
Only in the last step we need a basic inequality, see Proposition~\ref{th:bound-comp},
that we prove when $\varphi$ is the heat kernel or the indicator function of a ball,
as in Theorem~\ref{th:3mom}.
We believe that Proposition~\ref{th:bound-comp}
should hold in greater generality ---possibly, as soon as $\rho$ is log-concave---
but this remains open.
\end{remark}

\smallskip

Let us fix an integrable function $\varphi$, $t>0$ and $\theta \in \R$. Our goal is to prove that
\begin{equation*}
	\bbE\big[ \mathscr{Z}_{t}^\theta(\varphi)^3 \big] >
	\bbE\big[ \mathscr{M}_{t}^\theta(\varphi)^3 \big] \,.
\end{equation*}
Since first and second moments match, it is equivalent to work with
\emph{centered} third moments:
\begin{equation} \label{eq:lb-goal}
	\bbE\big[ \big(\mathscr{Z}_{t}^\theta(\varphi) - \bbE[\mathscr{Z}_{t}^\theta(\varphi)]\big)^3
	\big] >
	\bbE\big[ \big( \mathscr{M}_{t}^\theta(\varphi) - \bbE[\mathscr{M}_{t}^\theta(\varphi)]\big)^3
	\big] \,.
\end{equation}
In view of \eqref{eq:Zmom3} and \eqref{eq:Mmom3bis},
see also \eqref{eq:wefix}, we can
rewrite \eqref{eq:lb-goal} as
\begin{equation} \label{eq:K3bound}
	K_t^{(3)}(\varphi) > K^{(3)}_\GMC(\varphi) \,,
\end{equation}
where, given a kernel $H(z_1, z_2, z_3)$, we use the shorthand
\begin{equation}\label{eq:Kphi}
	H(\varphi) := \iiint\limits_{(\R^2)^3}
	\varphi(z_1) \, \varphi(z_2) \, \varphi(z_3) \,
	H(z_1, z_2, z_3) \,
	\dd z_1 \, \dd z_2 \, \dd z_3 \,.
\end{equation}

It remains to prove \eqref{eq:K3bound}.
The kernel $K^{(3)}_t$ is complicated,
but we can perform an almost exact computation of the
function $\textbf{g}^{(m)}_{a_1,b_1,\,\ldots,\, a_m,b_m}(z_1, z_2, z_3)$
in \eqref{eq:gm}, see Proposition~\ref{th:comutation-gm} below.
From this we obtain a lower bound on $K^{(3)}_t(\varphi)$
(Proposition~\ref{th:3lb}), that we complement with an upper bound
on $K^{(3)}_\GMC(\varphi)$ (Proposition~\ref{th:2ub}). At last,
we will show that these bounds are compatible (Proposition~\ref{th:bound-comp}),
which yields our goal \eqref{eq:K3bound}.

\smallskip

Let us introduce two key quantities $\scrG_{a_1, a_2}(\varphi)$
and $\widetilde\scrG_{a_1, a_2}(\varphi)$ that enter our bounds:
\begin{align}\label{eq:Gphi}
	\scrG_{a_1, a_2}(\varphi)
	& := (2\pi)^2
	\iiint\limits_{(\R^2)^3} \,
	\varphi(z_1) \, \varphi(z_2) \, \varphi(z_3)
	\ g_{a_1}(z_2-z_1) \, g_{a_2}(z_3 - \tfrac{z_1+z_2}{2})
	\ \dd \vec{z} \,, \\
	\label{eq:Gphi2}
	\widetilde\scrG_{a_1, a_2}(\varphi) &:= (2\pi)^2
	\iiint\limits_{(\R^2)^3} \,
	\varphi(z_1) \, \varphi(z_2) \, \varphi(z_3)
	\ g_{a_1}(z_2-z_1) \, g_{a_2}(z_3 - z_2)
	\ \dd \vec{z} \,,
\end{align}
where $g_t(z)$ denotes the heat kernel, see \eqref{eq:gt}.
We can now state our lower bound on $K_t^{(3)}(\varphi)$ which
involves the quantity $\scrG_{a_1, a_2}(\varphi)$.

\begin{proposition}[Third moment lower bound for the SHF]\label{th:3lb}
Fix $\theta \in \R$ and $t > 0$.
Let $K_t^{(3)}$ be the centered third moment kernel of the critical $2d$ SHF $\scrZ_t^\theta$,
see \eqref{eq:Zmom3}-\eqref{eq:K3}.
For any integrable function
$\varphi: \R^2 \to [0,\infty)$ which is  radially symmetric and non-increasing
(see Remark~\ref{rem:radially}), we have the strict lower bound
\begin{equation}\label{eq:3momlb}
	K_t^{(3)}(\varphi) \,>\, I_t^{(3)}(\varphi)  \,,
\end{equation}
where we define
\begin{equation}\label{eq:I3}
\begin{split}
	I_t^{(3)}(\varphi) :=\,
	3  \sum_{m=2}^\infty 2^{m-1} \,
	\!\!\!\!\!\!\!
	\idotsint\limits_{0 < a_1 < b_1 < \ldots < a_m < b_m < t}
	\!\!\!\!\!\!\!
	\scrG_{a_1, a_2}(\varphi) \, G_\theta(b_1-a_1) \,  G_\theta(b_2-a_2) & \\
	\times \, \prod_{i=3}^m \frac{G_\theta(b_i-a_i)}{a_i-b_{i-2}} &
	\, \dd \vec{a} \, \dd \vec{b} \,,
\end{split}
\end{equation}
with $\scrG_{a_1, a_2}(\varphi)$ as in \eqref{eq:Gphi} and $G_\theta$ as in \eqref{g-theta}.
\end{proposition}

We refer to Figure~\ref{fig:I3} for a graphical representation of $I_t^{(3)}(\varphi)$
when $\varphi = g_r$ is the heat kernel, in which case $\scrG_{a_1, a_2}(\varphi)$
can be computed explicitly (see Remark~\ref{rem:Gauss-comp}).

We next state an upper bound on $K^{(3)}_\GMC(\varphi)$
which involves the quantity $\widetilde\scrG_{a_1, a_2}(\varphi)$.
Interestingly, this bound applies to any positive integrable function~$\varphi$.

\begin{proposition}[Third moment upper bound for GMC]\label{th:2ub}
Fix $\theta \in \R$ and $t > 0$.
Let $K^{(3)}_\GMC$ be
the centered third moment kernel of the GMC $\scrM_t^\theta$, see \eqref{eq:Mmom3bis}
and \eqref{eq:wefix}.
For any integrable function $\varphi: \R^2 \to [0,\infty)$ we have the strict upper bound
\begin{equation}\label{eq:2momub}
	K^{(3)}_\GMC(\varphi)
	\,<\, \widetilde I_t^{(3)}(\varphi) \,,
\end{equation}
where we define
\begin{equation}\label{eq:tildeI3}
\begin{split}
	\widetilde I_t^{(3)}(\varphi) :=\,
	3  \sum_{m=2}^\infty 2^{m-1} \,
	\!\!\!\!\!\!\!
	\idotsint\limits_{0 < a_1 < b_1 < \ldots < a_m < b_m < t}
	\!\!\!\!\!\!\!
	\widetilde\scrG_{a_1, a_2}(\varphi) \, G_\theta(b_1-a_1) \,  G_\theta(b_2-a_2) & \\
	\times \, \prod_{i=3}^m \frac{G_\theta(b_i-a_i)}{a_i-b_{i-2}} &
	\, \dd \vec{a} \, \dd \vec{b} \,,
\end{split}
\end{equation}
with $\widetilde\scrG_{a_1, a_2}(\varphi)$ as in \eqref{eq:Gphi2} and $G_\theta$ as in \eqref{g-theta}.
\end{proposition}

Note that $\widetilde I_t^{(3)}(\varphi)$ in \eqref{eq:tildeI3}
is like $I_t^{(3)}(\varphi)$ in \eqref{eq:I3}, just
with $\widetilde\scrG_{a_1, a_2}(\varphi)$ in place of $\scrG_{a_1, a_2}(\varphi)$.
If $\scrG_{a_1, a_2}(\varphi) \,>\, \tilde\scrG_{a_1, a_2}(\varphi)$, then
we can combine the bounds  \eqref{eq:3momlb} and \eqref{eq:2momub} to yield
our goal \eqref{eq:K3bound}. We
finally show that this indeed holds when $\varphi$ is the indicator
function of a ball, or the heat kernel,
which completes the proof of Theorem~\ref{th:3mom}.

\begin{proposition}[Comparison of bounds]\label{th:bound-comp}
Recall $\scrG_{a_1, a_2}(\varphi)$ and $\widetilde \scrG_{a_1, a_2}(\varphi)$
from \eqref{eq:Gphi}-\eqref{eq:Gphi2}.
Let $\varphi$ be the indicator function of a ball or the heat kernel, see \eqref{eq:gt}:
\begin{equation}\label{eq:hypphi}
	\varphi = \ind_{\{x\in\R^2: \ |x| \le r\}} \qquad \text{or} \qquad
	\varphi = g_r \,, \qquad r > 0 \,.
\end{equation}
Then we have
\begin{equation}\label{eq:goalGG}
	\scrG_{a_1, a_2}(\varphi) \,>\, \tilde\scrG_{a_1, a_2}(\varphi) \
	\qquad \forall a_1, a_2 > 0 \,.
\end{equation}
Recalling \eqref{eq:I3} and \eqref{eq:tildeI3}, it follows that
for any $\theta \in \R$ and $t > 0$
\begin{equation}\label{eq:bound-comp}
	I^{(3)}_t(\varphi) > \widetilde I^{(3)}_t(\varphi) \,,
\end{equation}
therefore, in view of \eqref{eq:3momlb} and \eqref{eq:2momub},
one has $K_t^{(3)}(\varphi) > K^{(3)}_\GMC(\varphi)$.
\end{proposition}

\begin{remark}\label{rem:Gauss-comp}
When $\varphi = g_r$ is the heat kernel,
$\scrG_{a_1, a_2}(\varphi)$ and $\widetilde\scrG_{a_1, a_2}(\varphi)$ in
\eqref{eq:Gphi}-\eqref{eq:Gphi2} can be computed by an explicit Gaussian integration
(see Subsection~\ref{sec:bound-comp}):
\begin{equation} \label{eq:Ggr}
	\scrG_{a_1, a_2}(g_r)
	= \frac{1}{a_1+2r} \, \frac{1}{a_2 + \frac{3}{2}r} \,, \qquad
	\widetilde\scrG_{a_1, a_2}(g_r)
	= \frac{1}{a_1 a_2 + 2r(a_1+a_2) + 3r^2} \,,
\end{equation}
and in this case
one sees easily that $\scrG_{a_1, a_2}(g_r) > \widetilde\scrG_{a_1, a_2}(g_r)$,
in agreement with \eqref{eq:bound-comp}.

A graphical representation of $I_t^{(3)}(\varphi)$ for $\varphi = g_r$
is given in Figure~\ref{fig:I3}.
\end{remark}

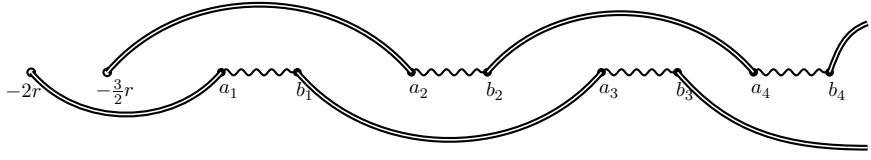
\begin{figure}
\begin{tikzpicture}[scale=0.5]
\draw  [thick] (-10, 0)  circle [radius=0.1]; \draw  [thick] (-8, 0)  circle [radius=0.1];
\draw  [fill] (0, 0)  circle [radius=0.1]; \draw  [fill] (2, 0)  circle [radius=0.1];
\draw  [fill] (5, 0)  circle [radius=0.1]; \draw  [fill] (7, 0)  circle [radius=0.1];
\draw  [fill] (-5,0)  circle [radius=0.1]; \draw  [fill] (-3,0)  circle [radius=0.1];
\draw  [fill] (9,0)  circle [radius=0.1];  \draw  [fill] (11,0)  circle [radius=0.1];
\node at (5.2,-0.5) {\scalebox{0.7}{$a_3$}};  \node at (7.2,-0.5) {\scalebox{0.7}{$b_3$}};
\node at (0.2,-0.5) {\scalebox{0.7}{$a_2$}};  \node at (2.2,-0.5) {\scalebox{0.7}{$b_2$}};
\node at (-4.8,-0.5) {\scalebox{0.7}{$a_1$}};  \node at (-2.8,-0.5) {\scalebox{0.7}{$b_1$}};
\node at (9.2,-0.5) {\scalebox{0.7}{$a_4$}};  \node at (11.2,-0.5) {\scalebox{0.7}{$b_4$}};
\node at (-10.2,-0.5) {\scalebox{0.7}{$-2r$}};  \node at (-7.8,-0.5) {\scalebox{0.7}{$-\frac{3}{2}r$}};
\draw [-,thick, decorate, decoration={snake,amplitude=.4mm,segment length=2mm}] (0,0) -- (2,0);
\draw [-,thick, decorate, decoration={snake,amplitude=.4mm,segment length=2mm}] (5,0) -- (7,0);
\draw [-,thick, decorate, decoration={snake,amplitude=.4mm,segment length=2mm}] (-5,0) -- (-3,0);
\draw [-,thick, decorate, decoration={snake,amplitude=.4mm,segment length=2mm}] (9,0) -- (11,0);
\draw[double, thick] (-10,0) to [out=-50,in=-130] (-5,0); \draw[double, thick] (-3,0) to [out=-50,in=-130] (5,0); \draw[double, thick] (7,0) to [out=-50,in=-180] (12,-2);
\draw[double, thick] (-8,0) to [out=50,in=130] (0,0); \draw[double, thick] (2,0) to [out=50,in=130] (9,0); \draw[double, thick] (11,0) to [out=70,in=200] (12,1.3);
\end{tikzpicture}
\caption{Graphical representation of $I_t^{(3)}(\varphi)$,
see \eqref{eq:I3}, when $\varphi =g_r$ is the heat kernel,
see \eqref{eq:Ggr}. More specifically, we represent the term $m=4$ in the series in \eqref{eq:I3}.
Double solid lines from $b_{i-2}$ to $a_i$ are assigned weights $(a_i - b_{i-2})^{-1}$,
while wiggle lines from $a_i$ to $b_i$ are assigned weights $G_\theta(b_i-a_i)$.}
\label{fig:I3}
\end{figure}

\smallskip

It only remains to prove Propositions~\ref{th:3lb}, \ref{th:2ub} and~\ref{th:bound-comp}, to which
Subsections~\ref{sec:3lb}, \ref{sec:2ub} and~\ref{sec:bound-comp} are devoted.

\subsection{Proof of Proposition~\ref{th:3lb}}
\label{sec:3lb}

The heart of the proof is the following
``computation'' of the function $\textbf{g}^{(m)}_{a_1,b_1,\ldots,a_m,b_m}(z_1, z_2, z_3)$
in \eqref{eq:gm}, which we will prove below.

\begin{proposition}\label{th:comutation-gm}
For $m \ge 2$, $0 < a_1 < b_1 < \ldots < a_m < b_m$,  $z_1, z_2, z_3 \in \R^2$
and ${\rm{ \bf g}}^{(m)}_{a_1,b_1,\ldots,a_m,b_m}(z_1, z_2, z_3)$
as in \eqref{eq:gm}, we have
\begin{equation}\label{eq:gmalt}
	{\rm {\bf g}}^{(m)}_{a_1,b_1,\ldots,a_m,b_m}(z_1, z_2, z_3)
	= g_{a_1}(z_1-z_2) \cdot g_{\overline{a_2}^{(m)}}\big( z_3 - \tfrac{z_1+z_2}{2}\big) \cdot
	\prod_{i=3}^m g_{\overline{a_i-b_{i-2}}^{(m)}}(0) \,,
\end{equation}
for suitable variables $\overline{a_2}^{(m)}$ and $\overline{a_i-b_{i-2}}^{(m)}$
(depending on $a_1,b_1,\ldots,a_m,b_m$) which satisfy
\begin{equation}\label{eq:variables}
\begin{gathered}
	\overline{a_2}^{(m)}\le a_2 - \tfrac{b_1}{4} < a_2 \,, \\
	\overline{a_i-b_{i-2}}^{(m)} \le a_i-b_{i-2} - \tfrac{b_{i-1}-a_{i-1}}{4} < a_i-b_{i-2} \,.
\end{gathered}
\end{equation}
\end{proposition}

We will also need a
basic monotonicity property for the function $\scrG_{a_1, a_2}(\varphi)$
in \eqref{eq:Gphi}.

\begin{lemma} \label{th:decre}
If $\varphi : \R^2 \to [0,\infty)$ is integrable, radially symmetric and non-increasing
(see Remark~\ref{rem:radially}),
then the function $\scrG_{a_1, a_2}(\varphi)$ in \eqref{eq:Gphi} is strictly decreasing in $a_2 > 0$.
\end{lemma}

\begin{proof}
By the change of variables
$x := z_1$, $y := z_3 - \frac{z_1+z_2}{2}$, $z := z_3$, we can write
\begin{equation} \label{eq:fg}
	\scrG_{a_1, a_2}(\varphi) := (2\pi)^2
	\int_{\R^2} f(y) \, g_{a_2}(y) \, \dd y \,,
\end{equation}
where we define
\begin{equation} \label{eq:fh}
\begin{split}
	f(y) := \int_{\R^2} h(z-y) \, \varphi(z) \, \dd z \,, \qquad
	h(w) := \int_{\R^2}
	\varphi(x) \, \varphi(2w-x) \, g_{a_1}(2w-2x) \, \dd x \,.
\end{split}
\end{equation}
By \eqref{eq:fg} we can write $\scrG_{a_1, a_2}(\varphi) = (2\pi)^2 \, \E[f(a_2 \, Z)]$, where
$Z$ is a standard Gaussian random variable on $\R^2$
(with density $g_1$). Then,
to prove that $a_2 \mapsto \scrG_{a_1, a_2}(\varphi)$ is strictly decreasing,
it is enough to prove that $f$ is radially non-increasing and integrable
(see Remark~\ref{rem:radially}). The integrability of $f$ is easily seen from \eqref{eq:fh},
which ensures that the volumes of the level sets of  $f(a_2 y)$ are finite and strictly decrease as $a_2$
increases. We then show that
\emph{both $f$ and $h$ are radially symmetric and non-increasing},
which completes the proof.

\smallskip

We recall the layer cake decomposition of a radially symmetric and non-increasing function:
\begin{equation}\label{eq:layer-cake}
	\text{for a.e.\ $x\in\R^2$}: \qquad
	\varphi(x) = \int_0^\infty \ind_{\{|x| < r\}} \, \mu^\varphi(\dd r) \,,
\end{equation}
where $\mu^\varphi$ is a positive measure on $(0,\infty)$,
defined by $\mu^\varphi((r,\infty)) := \varphi((r,0))$.
Using a similar decomposition for $g_{a_1}$,
we replace the three factors $\varphi$, $\varphi$ and $g_{a_1}$ in the definition of $h$
by $\ind_{\{|\cdot|<r_1\}}$, $\ind_{\{|\cdot|<r_2\}}$, and $\ind_{\{|\cdot|<s\}}$
and show that for any $r_1, r_2, s > 0$
the resulting function $\hat h$ is radially symmetric and non-increasing:
\begin{equation} \label{eq:hatfh}
\begin{split}
	\hat h(w) &:= \int_{\R^2}
	\ind_{\{|x| < r_1\}} \, \ind_{\{|x-2w|<r_2\}} \, \ind_{\{|x-w|<\frac{s}{2}\}} \, \dd x \\
	& = \mathrm{Leb}(B(0,r_1) \cap B(w,\tfrac{s}{2}) \cap B(2w, r_2)) \\
	& = \mathrm{Leb}(B(-w,r_1) \cap B(0,\tfrac{s}{2}) \cap B(w, r_2)) \,,
\end{split}
\end{equation}
where $B(z,r) := \{x\in\R^2 \colon |x| < r\}$ is the ball of radius~$r$ centered at~$z$.
It is clear that $\hat h$ is radially symmetric and non-increasing, and so is $h$ since it is
a mixture of $\hat h$ with different values of $r_1, r_2$ and $s$.

Note that we can write $f = \varphi * h$ as the convolution of two
radially symmetric and non-increasing functions.
If we replace $h$ and $\varphi$ by $\ind_{\{|\cdot|<t\}}$
and $\ind_{\{|\cdot|<r\}}$, by the layer cake decomposition,
we get the function
\begin{equation*}
	\hat f(y) = \int_{\R^2} \ind_{\{|z-y|<t\}} \, \ind_{\{|z|<r\}} \, \dd z
	= \mathrm{Leb}(B(0,r) \cap B(y,t)) \,,
\end{equation*}
which is clearly radially symmetric and non-increasing, hence the same holds for $f$.
\end{proof}

It is now easy to prove Proposition \ref{th:3lb}.
When we average ${\rm {\bf g}}^{(m)}_{a_1,b_1,\ldots,a_m,b_m}(z_1, z_2, z_3)$
with respect to the function $\varphi$ as in \eqref{eq:Kphi},
we can apply \eqref{eq:gmalt} to write, recalling \eqref{eq:Gphi},
\begin{equation*}
\begin{split}
	\textbf{g}^{(m)}_{a_1,b_1,\ldots,a_m,b_m}(\varphi)
	& \,=\, \frac{1}{(2\pi)^2} \, \scrG_{a_1, \overline{a_2}^{(m)}}(\varphi) \,
	\prod_{i=3}^m g_{\overline{a_i-b_{i-2}}^{(m)}}(0) \,.
\end{split}
\end{equation*}
Since $t \mapsto g_t(0)$ and $a_2 \mapsto \scrG_{a_1, a_2}(\varphi)$
are strictly decreasing functions, we obtain the bound
\begin{equation*}
\begin{split}
	\textbf{g}^{(m)}_{a_1,b_1,\ldots,a_m,b_m}(\varphi)
	& \,>\, \frac{1}{(2\pi)^2} \, \scrG_{a_1,  a_2}(\varphi) \,
	\prod_{i=3}^m g_{a_i-b_{i-2}}(0)
	\,=\, \frac{1}{(2\pi)^m} \, \scrG_{a_1,  a_2}(\varphi) \prod_{i=3}^m \frac{1}{a_i-b_{i-2}} \,.
\end{split}
\end{equation*}
In fact for $m\geq 3$, this strict inequality already follows from the fact that $t\to g_t(0)$ is strictly decreasing
and $a_2 \mapsto \scrG_{a_1, a_2}(\varphi)$ is non-decreasing.
Plugging this into \eqref{eq:K3}-\eqref{eq:Im}, we obtain $K_t^{(3)}(\varphi) > I_t^{(3)}(\varphi)$
with $I_t^{(3)}(\varphi)$ defined in \eqref{eq:I3}. This completes the proof of Proposition~\ref{th:3lb}.

\medskip

We are left with proving Proposition~\ref{th:comutation-gm}.
A key tool is the following elementary lemma.

\begin{lemma}[Triple Gaussian integral]\label{triple-gauss}
Let $g_t(x)$ be the two-dimensional heat kernel, see \eqref{eq:gt}.
For all $s,t > 0$ and $x,a,b\in\R^2$ we have
\begin{equation}\label{eq:identity}
	g_s(x-a) \, g_t(x-b) = g_{s+t}(a-b) \, g_{h(s,t)}(x-m_{t,s}(a,b)) \,,
\end{equation}
where we set
\begin{equation} \label{eq:hm}
	h(s,t) := \big(\tfrac{1}{s}+\tfrac{1}{t}\big)^{-1} = \frac{st}{s+t} \,, \qquad
	m_{t,s}(x,y) := \frac{t}{s+t} x + \frac{s}{s+t} y \,.
\end{equation}
It follows that for all $s,t,u > 0$ and $a,b,c \in \R^2$ we have
\begin{equation} \label{eq:3int}
	\int_{\R^2} g_s(x-a) \, g_t(x-b) \, g_u(x-c) \, \dd x =
	g_{s+t}(a-b) \, g_{h(s,t)+u}(c - m_{t,s}(a,b)) \,.
\end{equation}
\end{lemma}
\begin{proof}
\eqref{eq:identity} follows directly from the definition \eqref{eq:gt} of the heat
kernel and an easy algebraic manipulation.
Then \eqref{eq:3int} follows by \eqref{eq:identity} and a simple Gaussian convolution.
\end{proof}

\smallskip

\begin{proof}[Proof of Proposition~\ref{th:comutation-gm}]
We first prove \eqref{eq:gmalt}-\eqref{eq:variables} for $m=2$. Recall that, by \eqref{eq:gm},
\begin{equation} \label{eq:g2}
	\begin{split}
	& \textbf{g}^{(2)}_{a_1,b_1,a_2,b_2}(z_1, z_2, z_3)
	\ =
	\iiiint\limits_{(\R^2)^{4}}
	g_{\frac{a_1}{2}}(x_1-z_1) \, g_{\frac{a_1}{2}}(x_1-z_2) \cdot
	g_{\frac{b_1-a_1}{4}}(y_1-x_1) \\
	& \qquad\qquad\qquad\qquad \cdot \,  g_{\frac{a_2}{2}}(x_2-z_3)
	\, g_{\frac{a_2-b_1}{2}}(x_2-y_1) \cdot g_{\frac{b_2-a_2}{4}}(y_2-x_2) \,
	\dd  x_1 \, \dd y_1 \, \dd x_2 \, \dd y_2 \,.
\end{split}
\end{equation}
Since $\int_{\R^2} g_s(x-a) \, g_t(x-b) \, \dd x = g_{s+t}(a-b)$,
we can integrate $y_2$, then $x_2$, then $y_1$ to get
\begin{equation*}
	\begin{split}
	& \iint\limits_{(\R^2)^{2}}
	g_{\frac{a_1}{2}}(x_1-z_1) \, g_{\frac{a_1}{2}}(x_1-z_2) \cdot
	g_{\frac{b_1-a_1}{4}}(y_1-x_1)  \cdot  g_{\frac{a_2}{2} +\frac{a_2-b_1}{2}}(y_1 - z_3)
	\, \dd x_1 \, \dd  y_1 \\
	&  = \int\limits_{\R^2}
	g_{\frac{a_1}{2}}(x_1-z_1) \, g_{\frac{a_1}{2}}(x_1-z_2) \cdot
	g_{\frac{b_1-a_1}{4} + \frac{a_2}{2} +\frac{a_2-b_1}{2}}(x_1 - z_3) \, \dd x_1 \,.
\end{split}
\end{equation*}
Applying \eqref{eq:3int} to compute the last integral over $x_1$, we finally obtain
\begin{equation}\label{eq:m=2}
	\textbf{g}^{(2)}_{a_1,b_1,a_2,b_2}(z_1, z_2, z_3) =
	g_{a_1}(z_1-z_2) \, g_{\overline{a_2}^{(2)}}\big(z_3 - \tfrac{z_1+z_2}{2}\big) \,,
\end{equation}
where we set
\begin{equation}\label{eq:bara}
	\overline{a_2}^{(2)} := \tfrac{a_2}{2} +\tfrac{a_2-b_1}{2} + \tfrac{b_1}{4}
	= a_2 - \tfrac{b_1}{4} \,.
\end{equation}
This completes the proof of \eqref{eq:gmalt}-\eqref{eq:variables} for $m=2$.

We next move to $m \ge 3$.
In formula \eqref{eq:gm},
the terms depending on $x_{m}$ and $y_{m}$  are
\begin{equation} \label{eq:inv1}
\begin{split}
	& g_{\frac{a_{m}-b_{m-2}}{2}}(x_{m}-y_{m-2}) \, g_{\frac{a_{m}-b_{m-1}}{2}}(x_{m}-y_{m-1})
	\cdot g_{\frac{b_{m}-a_{m}}{4}}(y_{m} - x_{m}) \,,
\end{split}
\end{equation}
which after integration over $y_{m}$ and $x_{m}$ give
\begin{equation} \label{eq:inv2}
	g_{\frac{a_{m}-b_{m-2}}{2}+\frac{a_{m}-b_{m-1}}{2}}(y_{m-1} - y_{m-2})
	= g_{a_{m} - \frac{b_{m-1} + b_{m-2}}{2}}(y_{m-1} - y_{m-2}) \,.
\end{equation}
This shows that we can rewrite \eqref{eq:gm} for $m \ge 3$ as follows:
\begin{equation} \label{eq:gmhat}
\begin{split}
	\textbf{g}^{(m)}_{a_1,b_1,\ldots,a_m}(z_1, z_2, z_3)
	\, := \!\!\!\!\!\!\!\!\!\!\iint\limits_{(\R^2)^{m-1} \times (\R^2)^{m-1}} \!\!\!\!\!\!\!\!\!\!
	\dd \vec x \, \dd \vec y \
	g_{\frac{a_1}{2}}(x_1-z_1) \, g_{\frac{a_1}{2}}(x_1-z_2) \cdot
	g_{\frac{b_1-a_1}{4}}(y_1-x_1) & \\
	\cdot \,  g_{\frac{a_2}{2}}(x_2-z_3)
	\, g_{\frac{a_2-b_1}{2}}(x_2-y_1) \cdot g_{\frac{b_2-a_2}{4}}(y_2-x_2) & \\
	\cdot \prod_{i=3}^{m-1}
	\Big\{ g_{\frac{a_i-b_{i-2}}{2}}(x_i-y_{i-2}) \, g_{\frac{a_i-b_{i-1}}{2}}(x_i-y_{i-1})
	\cdot g_{\frac{b_i-a_i}{4}}(y_i - x_i) \Big\} & \\
	\cdot \, g_{a_m - \frac{b_{m-1} + b_{m-2}}{2}}(y_{m-1}-y_{m-2}) \,,
\end{split}
\end{equation}
where we agree that $\prod_{i=3}^{m-1}\{\ldots\} := 1$ for $m=3$.
We note that $b_m$ does not appear in the r.h.s.\ of \eqref{eq:gmhat}, hence we dropped it
from the notation $\textbf{g}^{(m)}_{a_1,b_1,\ldots,a_m}(z_1, z_2, z_3)$.

We are ready to prove \eqref{eq:gmalt}-\eqref{eq:variables} for $m\ge 3$ by induction.
For $m=3$, \eqref{eq:gmhat} becomes
\begin{equation*}
	\begin{split}
	\textbf{g}^{(3)}_{a_1,b_1,a_2,b_2, a_3}(z_1, z_2, z_3)
	= \iiiint\limits_{(\R^2)^{4}}  \
	g_{\frac{a_1}{2}}(x_1-z_1) \, g_{\frac{a_1}{2}}(x_1-z_2) \cdot
	g_{\frac{b_1-a_1}{4}}(y_1-x_1) & \\
	\cdot \,  g_{\frac{a_2}{2}}(x_2-z_3)
	\, g_{\frac{a_2-b_1}{2}}(x_2-y_1) \cdot g_{\frac{b_2-a_2}{4}}(y_2-x_2)  & \\
	\cdot \, g_{a_3-\frac{b_1+b_2}{2}}(y_2-y_1)  \
	\dd  x_1 \, \dd y_1 \, \dd x_2 \, \dd y_2 & \,,
\end{split}
\end{equation*}
and integrating over $y_2$ we obtain
\begin{equation} \label{eq:intermint}
	\begin{split}
	&\iiint\limits_{(\R^2)^{3}}  \
	g_{\frac{a_1}{2}}(x_1-z_1) \, g_{\frac{a_1}{2}}(x_1-z_2) \cdot
	g_{\frac{b_1-a_1}{4}}(y_1-x_1) \\
	& \qquad \cdot \,  g_{\frac{a_2}{2}}(x_2-z_3)
	\, g_{\frac{a_2-b_1}{2}}(x_2-y_1) \cdot
	g_{a_3-\frac{b_1}{2} - \frac{a_2+b_2}{4}}(x_2-y_1) \,
	\dd  x_1 \, \dd y_1 \, \dd x_2 \,.
\end{split}
\end{equation}
When we integrate the last line over $x_2$, by \eqref{eq:3int} we get
\begin{equation*}
\begin{split}
	& g_{a_3-\frac{b_1}{2} - \frac{a_2+b_2}{4}+\frac{a_2-b_1}{2}}(0) \,
	g_{\frac{a_2}{2}
	+ h(\frac{a_2-b_1}{2}, \, a_3-\frac{b_1}{2} - \frac{a_2+b_2}{4})}(y_1-z_3)
	 = g_{\overline{a_3-b_1}^{(3)}}(0) \cdot
	g_{\overline{a_2}^{(3)} - \frac{b_1}{4} }(y_1-z_3)
\end{split}
\end{equation*}
where we define
\begin{equation}\label{eq:tildea}
\begin{split}
	\overline{a_3-b_1}^{(3)} & := (a_3-b_1) - \tfrac{b_2-a_2}{4} \,, \\
	\overline{a_2}^{(3)} &:= \tfrac{a_2}{2} + \tfrac{b_1}{4} +
	h\big(\tfrac{a_2-b_1}{2}, \, a_3-\tfrac{b_1}{2} - \tfrac{a_2+b_2}{4}\big) \,.
\end{split}
\end{equation}
We can then perform the integral over $y_1$ in \eqref{eq:intermint} to get
\begin{equation*}
	\begin{split}
	g_{\overline{a_3-b_1}^{(3)}}(0)   \int\limits_{\R^2}  \
	g_{\frac{a_1}{2}}(x_1-z_1) \, g_{\frac{a_1}{2}}(x_1-z_2) \cdot
	g_{\overline{a_2}^{(3)} - \frac{a_1}{4}}(x_1-z_3) \, \dd  x_1 \,,
\end{split}
\end{equation*}
and a further application of \eqref{eq:3int} finally yields
\begin{equation}\label{eq:m=3}
	\textbf{g}^{(3)}_{a_1,b_1,a_2,b_2, a_3}(z_1, z_2, z_3) =
	g_{a_1}(z_1-z_2) \, g_{\overline{a}^{(3)}_2}\big(z_3 - \tfrac{z_1+z_2}{2}\big)
	\, g_{\overline{a_3-b_1}^{(3)}}(0) \,.
\end{equation}
This proves \eqref{eq:gmalt} for $m=3$. To prove \eqref{eq:variables},
we note that  $h(s,t) < s$, see \eqref{eq:hm}, hence
\begin{equation*}
	\begin{split}
	\overline{a_2}^{(3)} < \tfrac{a_2}{2} +
	\tfrac{b_1}{4} + \tfrac{a_2-b_1}{2}
	= a_2 - \tfrac{b_1}{4} \,.
\end{split}
\end{equation*}

We finally fix $m \ge 3$, we assume that formulas \eqref{eq:gmalt}-\eqref{eq:variables} hold
for $\textbf{g}^{(m)}$ and we prove that they hold for
$\textbf{g}^{(m+1)}$. To this purpose, it is enough to show that
\begin{equation}\label{eq:itisi}
\begin{gathered}
	\textbf{g}^{(m+1)}_{a_1,b_1,\ldots, a_m, b_m, a_{m+1}}(z_1, z_2, z_3)
	= g_{\overline{a_{m+1}-b_{m-1}}^{(m+1)}}(0) \cdot
	\textbf{g}^{(m)}_{a_1,b_1,\ldots, a_{m-1}, b_{m-1}, \widetilde{a_{m}}}(z_1, z_2, z_3) \\
	\text{for suitable} \quad \overline{a_{m+1}-b_{m-1}}^{(m+1)}
	\le a_{m+1}-b_{m-1} - \tfrac{b_m-a_m}{4}
	\quad \text{and} \quad \widetilde{a_{m}} < a_m \,.
\end{gathered}
\end{equation}
Indeed, by the induction step
we can apply \eqref{eq:gmalt}-\eqref{eq:variables} to $\textbf{g}^{(m)}$ in the r.h.s.,
and since $\widetilde{a_m} < a_m$
we obtain \eqref{eq:gmalt}-\eqref{eq:variables} for $\textbf{g}^{(m+1)}$.

It only remains to prove \eqref{eq:itisi}.
If we write formula \eqref{eq:gmhat} for
$\textbf{g}^{(m+1)}_{a_1,b_1,\ldots,a_m, b_m, a_{m+1}}(z_1, z_2, z_3)$, we see that
the terms which depend on $x_m$ and $y_m$ are
\begin{equation*}
\begin{split}
	g_{\frac{a_{m}-b_{m-2}}{2}}(x_{m}-y_{m-2}) \, g_{\frac{a_{m}-b_{m-1}}{2}}(x_{m}-y_{m-1})
	\cdot g_{\frac{b_{m}-a_{m}}{4}}(y_{m} - x_{m})  & \\
	\cdot \, g_{a_{m+1} - \frac{b_{m} + b_{m-1}}{2}}(y_{m}-y_{m-1}) & \,,
\end{split}
\end{equation*}
which after integration over $y_m$ yield
\begin{equation*}
\begin{split}
	& g_{\frac{a_{m}-b_{m-2}}{2}}(x_{m}-y_{m-2}) \,
	g_{\frac{a_{m}-b_{m-1}}{2}}(x_{m}-y_{m-1}) \cdot	
	g_{a_{m+1} - \frac{b_{m-1}}{2} - \frac{a_{m}+b_{m}}{4}}(x_{m} - y_{m-1}) \,.
\end{split}
\end{equation*}
A further integration over $x_m$ gives, by \eqref{eq:3int},
\begin{equation*}
\begin{split}
	& g_{a_{m+1} - \frac{b_{m-1}}{2} - \frac{a_{m}+b_{m}}{4} + \frac{a_{m}-b_{m-1}}{2}}(0)
	 \cdot g_{h(\frac{a_{m}-b_{m-1}}{2}, \,
	a_{m+1} - \frac{b_{m-1}}{2} - \frac{a_{m}+b_{m}}{4})
	+ \frac{a_{m}-b_{m-2}}{2}}(y_{m-1} - y_{m-2}) \\
	& = \, g_{\overline{a_{m+1}-b_{m-1}}^{(m+1)}}(0) \,\cdot\,
	g_{\widetilde{a_{m}} - \frac{b_{m-1}+b_{m-2}}{2}}(y_{m-1} - y_{m-2}) \,,
\end{split}
\end{equation*}
where we define
\begin{equation*}
\begin{split}
	\overline{a_{m+1}-b_{m-1}}^{(m+1)}
	& := (a_{m+1}-b_{m-1}) - \tfrac{b_m-a_m}{4} \,, \\
	\widetilde{a_{m}} &:= \tfrac{a_{m} + b_{m-1}}{2} +
	h\big(\tfrac{a_{m}-b_{m-1}}{2}, \,
	a_{m+1} - \tfrac{b_{m-1}}{2} - \tfrac{a_{m}+b_{m}}{4}\big) \,.
\end{split}
\end{equation*}
Recalling \eqref{eq:gmhat},
we see that \eqref{eq:itisi} holds (note that
$\widetilde{a_m} < a_m$ because $h(s,t) < s$).
\end{proof}

\subsection{Proof of Proposition~\ref{th:2ub}}
\label{sec:2ub}

We recall relation \eqref{eq:GMC-3-2} satisfied by any GMC.
Our choice \eqref{eq:wefix} ensures that
$K_\GMC^{(2)}(z_1, z_2) = K_t^{(2)}(z_1, z_2)$, see \eqref{eq:Zmom2},
hence \eqref{eq:GMC-3-2} becomes
\begin{equation} \label{eq:GMC?}
\begin{split}
	K^{(3)}_\GMC(z_1, z_2, z_3) = K_t^{(2)}(z_1, z_2) \, K_t^{(2)}(z_2, z_3) \, K_t^{(2)}(z_1, z_3)
	+  K_t^{(2)}(z_1, z_2) \, K_t^{(2)}(z_2, z_3) & \\
	\, + \, K_t^{(2)}(z_1, z_2) \, K_t^{(2)}(z_1, z_3) & \\
	\,+\, K_t^{(2)}(z_1, z_3) \, K_t^{(2)}(z_2, z_3) & \,.
\end{split}
\end{equation}

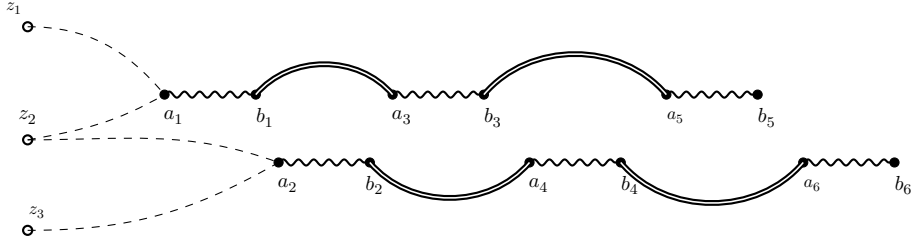
\begin{figure}
\begin{tikzpicture}[scale=0.6]
\draw  [thick] (-8, 1.5)  circle [radius=0.1]; \draw  [thick] (-8, -1)  circle [radius=0.1];
\draw  [fill] (0, 0)  circle [radius=0.1]; \draw  [fill] (2, 0)  circle [radius=0.1];
\draw  [fill] (-5,0)  circle [radius=0.1]; \draw  [fill] (-3,0)  circle [radius=0.1];
\draw  [fill] (6,0)  circle [radius=0.1];  \draw  [fill] (8,0)  circle [radius=0.1];
\node at (0.2,-0.5) {\scalebox{0.7}{$a_3$}};  \node at (2.2,-0.5) {\scalebox{0.7}{$b_3$}};
\node at (-4.8,-0.5) {\scalebox{0.7}{$a_1$}};  \node at (-2.8,-0.5) {\scalebox{0.7}{$b_1$}};
\node at (6.2,-0.5) {\scalebox{0.6}{$a_5$}};  \node at (8.2,-0.5) {\scalebox{0.7}{$b_5$}};
\node at (-8.3,1.9) {\scalebox{0.6}{$z_1$}};  \node at (-8,-0.5) {\scalebox{0.7}{$z_2$}};
\draw [-,thick, decorate, decoration={snake,amplitude=.4mm,segment length=2mm}] (0,0) -- (2,0);
\draw [-,thick, decorate, decoration={snake,amplitude=.4mm,segment length=2mm}] (-5,0) -- (-3,0);
\draw [-,thick, decorate, decoration={snake,amplitude=.4mm,segment length=2mm}] (6,0) -- (8,0);
\draw[dashed] (-8,1.5) to [out=0,in=130] (-5,0); \draw[double, thick] (-3,0) to [out=50,in=130] (0,0);
\draw[double, thick] (2,0) to [out=50,in=130] (6,0);
\draw[dashed] (-8,-1) to [out=10,in=-150] (-5,0);
\draw  [thick] (-8, -3)  circle [radius=0.1];
\draw  [fill] (3, -1.5)  circle [radius=0.1]; \draw  [fill] (5, -1.5)  circle [radius=0.1];
\draw  [fill] (-2.5,-1.5)  circle [radius=0.1]; \draw  [fill] (-0.5,-1.5)  circle [radius=0.1];
\draw  [fill] (9,-1.5)  circle [radius=0.1];  \draw  [fill] (11,-1.5)  circle [radius=0.1];
\node at (3.2,-2) {\scalebox{0.7}{$a_4$}};  \node at (5.2,-2) {\scalebox{0.7}{$b_4$}};
\node at (-2.3,-2) {\scalebox{0.7}{$a_2$}};  \node at (-0.4,-2) {\scalebox{0.7}{$b_2$}};
\node at (9.2,-2) {\scalebox{0.6}{$a_6$}};  \node at (11.2,-2) {\scalebox{0.7}{$b_6$}};
\node at (-7.8,-2.6) {\scalebox{0.6}{$z_3$}};
\draw [-,thick, decorate, decoration={snake,amplitude=.4mm,segment length=2mm}] (3,-1.5) -- (5,-1.5);
\draw [-,thick, decorate, decoration={snake,amplitude=.4mm,segment length=2mm}] (-2.5,-1.5) -- (-0.5,-1.5);
\draw [-,thick, decorate, decoration={snake,amplitude=.4mm,segment length=2mm}] (9,-1.5) -- (11,-1.5);
\draw[dashed] (-8,-3) to [out=0,in=-150] (-2.5,-1.5);
\draw[double, thick] (-0.5,-1.5) to [out=-50,in=-130] (3,-1.5); \draw[double, thick] (5,-1.5) to [out=-50,in=-130] (9,-1.5);
\draw[dashed] (-8,-1) to [out=0,in=160] (-2.5,-1.5);
\end{tikzpicture}
\caption{Graphical representation of
 the term $m=6$ in the series \eqref{eq:Kprod2alt}
 which represents $K_t^{(2)}(z_1, z_2) \, K_t^{(2)}(z_2, z_3)$.
The total weight of the dashed lines from $z_1$ and  $z_2$ to $a_1$ is assigned weight  $g_{a_1}(z_2 - z_1)$
and the total weight of the dashed lines from $z_2$ and $ z_3$ to $a_2$ is assigned weight  $g_{a_2}(z_3 - z_2)$ ;
a double solid line from $b_{i-2}$ to an $a_i$ is assigned weight $(a_i - b_{i-2})^{-1}$;
a wiggle line from an $a_i$ to $b_i$  is assigned weight  $G_\theta(b_i-a_i)$.}
\label{fig:GMC-2}
\end{figure}

We first give an alternative expression, that we prove below, for the product of two
covariance kernels which appear in the r.h.s.\ of \eqref{eq:GMC?}.

\begin{lemma}[Double correlation product] \label{th:prod2}
The following equality holds:
\begin{equation} \label{eq:Kprod2alt}
\begin{split}
	K_t^{(2)}(z_1, z_2) \, K_t^{(2)}(z_2, z_3)
	\, = \, & (2\pi)^2 \sum_{m=2}^\infty
	\quad \idotsint\limits_{0 < a_1 < b_1 < \ldots < a_m < b_m < t} \!\!\!\!\!\!\!\!
	\dd \vec{a} \ \dd\vec{b} \\
	& \ \Big\{ g_{a_1}(z_2-z_1) \, g_{a_2}(z_3-z_2)
	\,+ \, g_{a_1}(z_3-z_2) \, g_{a_2}(z_2-z_1) \Big\} \\
	& \ \cdot G_\theta(b_1-a_1) \, G_\theta(b_2-a_2)
	\prod_{i=3}^m \frac{G_\theta(b_i-a_i)}{a_i-b_{i-2}} \,,
\end{split}
\end{equation}
see Figure~\ref{fig:GMC-2} for a graphical representation.
\end{lemma}

When we average $K_t^{(2)}(z_1, z_2) \, K_t^{(2)}(z_2, z_3)$ with respect to a function
$\varphi$ as in \eqref{eq:Kphi},
recalling the quantity $\widetilde \scrG_{a_1, a_2}(\varphi)$ from \eqref{eq:Gphi2}, we obtain
the equality
\begin{equation} \label{eq:Kprod2bound}
\begin{split}
	& \int\limits_{(\R^2)^3} \!\!
	\varphi(z_1) \, \varphi(z_2) \, \varphi(z_3) \ K_t^{(2)}(z_1, z_2) \, K_t^{(2)}(z_2, z_3) \ \dd \vec{z}  \\
	& \; =  2 \sum_{m=2}^\infty
	\ \  \idotsint\limits_{0 < a_1 < b_1 < \ldots < a_m < b_m < t} \!\!\!\!\!\!\!\!\!\!\!\!\!\!
	\dd \vec{a} \ \dd\vec{b} \ \,
	\widetilde\scrG_{a_1, a_2}(\varphi) \
	G_\theta(b_1-a_1) \, G_\theta(b_2-a_2)
	\, \prod_{i=3}^m \frac{G_\theta(b_i-a_i)}{a_i-b_{i-2}} \,.
\end{split}
\end{equation}
Note that this expression resembles $\widetilde I_t^{(3)}(\varphi)$ in \eqref{eq:tildeI3},
except that $3 \cdot 2^{m-1}$
is replaced by~$2$.

\smallskip

We next consider the product of \emph{three} covariance kernels
as in \eqref{eq:GMC?}. The following result is also proved below.

\begin{lemma}[Triple correlation product] \label{th:prod3}
The following equality holds:
\begin{equation} \label{eq:Kprod3alt1}
	K_t^{(2)}(z_1, z_2) \, K_t^{(2)}(z_2, z_3)\, K_t^{(2)}(z_1, z_3)
	\,=\, \sum_{\substack{\alpha, \beta, \gamma \in \{12, \, 23, \,13\} \\
	\alpha \ne \beta, \, \beta \ne \gamma, \, \alpha \ne \gamma}} \cI(\alpha, \beta, \gamma)\,,
\end{equation}
where we set
\begin{equation} \label{eq:Kprod3alt2}
\begin{split}
	\cI(12,\, 23,\, 13)
	\, := \, (2\pi)^3 \sum_{m=3}^\infty \ \sum_{\ell = 3}^m
	\quad \idotsint\limits_{0 < a_1 < b_1 < \ldots < a_m < b_m < t} \!\!\!\!\!\!\!\!
	\dd \vec{a} \ \dd\vec{b} \qquad\qquad\qquad\qquad\qquad & \\
	g_{a_1}(z_1-z_2) \, g_{a_2}(z_2-z_3) \, G_\theta(b_1-a_1) \, G_\theta(b_2-a_2) \,
	\prod_{i=3}^{\ell-1} \frac{G_\theta(b_i-a_i)}{a_i-b_{i-2}}  & \\
	\cdot \, g_{a_\ell}(z_1 - z_3) \, G_\theta(b_\ell - a_\ell) \,
	\sum_{\substack{\sigma_{\ell+1}, \ldots, \sigma_m \in \{12, \, 23, \, 13\}\\
	\sigma_{\ell+1} \ne 13, \
	\sigma_i \ne \sigma_{i-1} \, \forall i }} \
	\prod_{i=\ell+1}^m \frac{G_\theta(b_i-a_i)}{a_i-b_{\mathrm{prev}(i)}}  & \,,
\end{split}
\end{equation}
see Figure~\ref{fig:GMC-3} for a graphical representation, where we define
\begin{equation} \label{eq:prev}
	\mathrm{prev}(i) := \max\{j \in \{1, \ldots, i-2\}: \, \sigma_j = \sigma_i\} \,,
\end{equation}
and we set
$\sigma_j = 12$ for odd $ j \le \ell-1$, $\sigma_j = 23$ for even $ j \le \ell-1$, and
$\sigma_\ell := 13$.
\end{lemma}

\begin{figure}
\begin{tikzpicture}[scale=0.6]
\draw  [thick] (-8, 1.5)  circle [radius=0.1]; \draw  [thick] (-8, -1)  circle [radius=0.1];
\draw  [fill] (-1, 0)  circle [radius=0.1]; \draw  [fill] (1, 0)  circle [radius=0.1];
\draw  [fill] (-6,0)  circle [radius=0.1]; \draw  [fill] (-4,0)  circle [radius=0.1];
\draw  [fill] (6.5,0)  circle [radius=0.1];  \draw  [fill] (8,0)  circle [radius=0.1];
\node at (-0.8,-0.5) {\scalebox{0.7}{$a_3$}};  \node at (1.2,-0.5) {\scalebox{0.7}{$b_3$}};
\node at (-5.8,-0.5) {\scalebox{0.7}{$a_1$}};  \node at (-3.8,-0.5) {\scalebox{0.7}{$b_1$}};
\node at (6.7,-0.5) {\scalebox{0.7}{$a_6$}};  \node at (8.2,-0.5) {\scalebox{0.7}{$b_6$}};
\node at (-8,1.9) {\scalebox{0.7}{$z_1$}};  \node at (-8,-0.6) {\scalebox{0.7}{$z_2$}};
\draw [-,thick, decorate, decoration={snake,amplitude=.4mm,segment length=2mm}] (-1,0) -- (1,0);
\draw [-,thick, decorate, decoration={snake,amplitude=.4mm,segment length=2mm}] (-6,0) -- (-4,0);
\draw [-,thick, decorate, decoration={snake,amplitude=.4mm,segment length=2mm}] (6.5,0) -- (8,0);
\draw[dashed] (-8,1.5) to [out=0,in=120] (-6,0);
\draw[double, thick] (-4,0) to [out=50,in=130] (-1,0); \draw[double, thick] (1,0) to [out=40,in=140] (6.5,0);
\draw[dashed] (-8,-1) to [out=10,in=-140] (-6,0);
\draw  [thick] (-8, -3)  circle [radius=0.1];
\draw  [fill] (2, -1.5)  circle [radius=0.1]; \draw  [fill] (3, -1.5)  circle [radius=0.1];
\draw  [fill] (-3.5,-1.5)  circle [radius=0.1]; \draw  [fill] (-1.5,-1.5)  circle [radius=0.1];
\draw  [fill] (9,-1.5)  circle [radius=0.1];  \draw  [fill] (11,-1.5)  circle [radius=0.1];
\node at (2.2,-2) {\scalebox{0.7}{$a_4$}};  \node at (3.2,-2) {\scalebox{0.7}{$b_4$}};
\node at (-3.3,-2) {\scalebox{0.7}{$a_2$}};  \node at (-1.4,-2) {\scalebox{0.7}{$b_2$}};
\node at (9.2,-2) {\scalebox{0.7}{$a_7$}};  \node at (11.2,-2) {\scalebox{0.7}{$b_7$}};
\node at (-8,-2.6) {\scalebox{0.7}{$z_3$}};
\draw [-,thick, decorate, decoration={snake,amplitude=.4mm,segment length=2mm}] (2,-1.5) -- (3,-1.5);
\draw [-,thick, decorate, decoration={snake,amplitude=.4mm,segment length=2mm}] (-3.5,-1.5) -- (-1.5,-1.5);
\draw [-,thick, decorate, decoration={snake,amplitude=.4mm,segment length=2mm}] (9,-1.5) -- (11,-1.5);
\draw[dashed] (-8,-3) to [out=0,in=-140] (-3.5,-1.5); \draw[double, thick] (-1.5,-1.5) to [out=-40,in=-140] (2,-1.5);
\draw[double, thick] (3,-1.5) to [out=-30,in=-150] (9,-1.5); \draw[dashed] (-8,-1) to [out=0,in=160] (-3.5,-1.5);
\draw  [fill] (4.25, -3.5)  circle [radius=0.1]; \draw  [fill] (5.75, -3.5)  circle [radius=0.1];
\draw [-,thick, decorate, decoration={snake,amplitude=.4mm,segment length=2mm}] (4.25,-3.5) -- (5.75,-3.5);
\node at (4.5,-4) {\scalebox{0.7}{$a_5$}};  \node at (6,-4) {\scalebox{0.7}{$b_5$}};
\draw[dashed] (-8,-3) to [out=-20,in=-160] (4.25,-3.5); \draw[dashed] (-8,1.5) to [out=20,in=90] (4.25,-3.5);
\end{tikzpicture}
\caption{Graphical representation of
the term $m=7$ in the series \eqref{eq:Kprod3alt2},
which describes $K_t^{(2)}(z_1, z_2) \, K_t^{(2)}(z_2, z_3)\, K_t^{(2)}(z_1, z_3)$,
see \eqref{eq:Kprod3alt1}.
Pairs of dashed lines from $z_i, z_j$ to an $a$ are assigned {\it total} weight $g_a(z_i - z_j)$;
double solid lines from $b_{i-2}$ to $a_i$ are assigned weight $(a_i - b_{i-2})^{-1}$;
wiggle lines from $a_i$ to $b_i$ are assigned weight $G_\theta(b_i-a_i)$.
Referring to \eqref{eq:Kprod3alt2}, we have $\ell = 5$ and $\mathrm{prev}(6) = 3$,
$\mathrm{prev}(7) = 4$.}
\label{fig:GMC-3}
\end{figure}
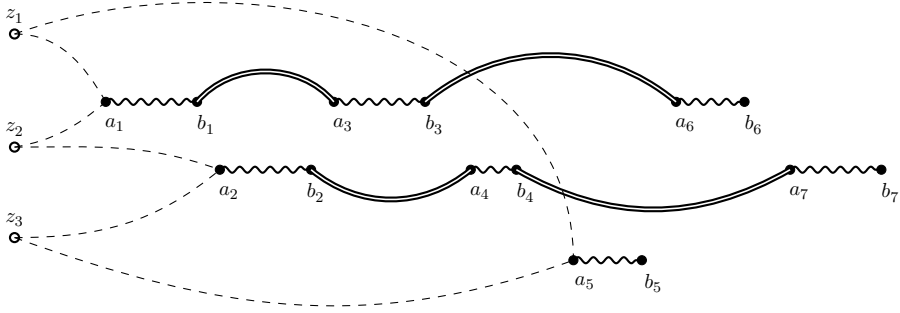

The definition of $\cI(12,\, 23,\, 13)$ in
\eqref{eq:Kprod3alt2} is complicated, but a much
simpler bound will be enough for us: if we shorten the gaps $a_i - b_{\mathrm{prev}(i)}
\ge a_i - b_{i-2}$, see \eqref{eq:prev}, and we bound
\begin{equation*}
	g_{a_\ell}(z_1-z_3) \le g_{a_\ell}(0)
	= \frac{1}{2\pi \, a_\ell} < \frac{1}{2\pi \, (a_\ell -b_{\ell-2})} \,,
\end{equation*}
then we can estimate
\begin{equation*}
\begin{split}
	\cI(12,\, 23,\, 13)
	\, < \, & (2\pi)^2 \sum_{m=3}^\infty \ \sum_{\ell = 3}^m
	\ 2^{m-\ell}
	\quad \idotsint\limits_{0 < a_1 < b_1 < \ldots < a_m < b_m < 1} \!\!\!\!\!\!\!\!
	\dd \vec{a} \ \dd\vec{b} \\
	& \qquad g_{a_1}(z_1-z_2) \, g_{a_2}(z_2-z_3) \, G_\theta(b_1-a_1) \, G_\theta(b_2-a_2) \,
	\prod_{i=3}^{m} \frac{G_\theta(b_i-a_i)}{a_i-b_{i-2}}  \,,
\end{split}
\end{equation*}
where $2^{m-\ell}$ is the number of choices of $\sigma_{\ell+1}, \ldots, \sigma_m$
in \eqref{eq:Kprod3alt2}.
Recalling \eqref{eq:Gphi2}, we obtain
\begin{equation} \label{eq:Kprod3bound}
\begin{split}
	& \int\limits_{(\R^2)^3} \!\!
	\varphi(z_1) \, \varphi(z_2) \, \varphi(z_3) \
	K_t^{(2)}(z_1, z_2) \, K_t^{(2)}(z_2, z_3) \, K_t^{(2)}(z_1, z_3) \ \dd \vec{z}  \\
	& \  < \, 6 \sum_{m=3}^\infty \ \sum_{\ell = 3}^m \, 2^{m-\ell}
	\!\!\!\!\! \idotsint\limits_{0 < a_1 < b_1 < \ldots < a_m < b_m < 1} \!\!\!\!\!\!\!\!\!\!\!\!\!
	\dd \vec{a} \ \dd\vec{b} \ \,
	\widetilde\scrG_{a_1, a_2}(\varphi) \
	G_\theta(b_1-a_1) \, G_\theta(b_2-a_2)
	\prod_{i=3}^m \frac{G_\theta(b_i-a_i)}{a_i-b_{i-2}} \\
	& \ = \, 6 \sum_{m=2}^\infty (2^{m-2}-1)
	\!\!\!\!\!\!\!\!\!\! \idotsint\limits_{0 < a_1 < b_1 < \ldots < a_m < b_m < 1} \!\!\!\!\!\!\!\!\!\!\!\!\!
	\dd \vec{a} \ \dd\vec{b} \ \,
	\widetilde\scrG_{a_1, a_2}(\varphi) \
	G_\theta(b_1-a_1) \, G_\theta(b_2-a_2)
	\prod_{i=3}^m \frac{G_\theta(b_i-a_i)}{a_i-b_{i-2}} \,,
\end{split}
\end{equation}
where in the last line we added the term $m=2$ because the factor $(2^{m-2}-1)$ vanishes.

\smallskip

We finally plug \eqref{eq:Kprod3bound} and (three times) \eqref{eq:Kprod2bound} into \eqref{eq:GMC?}.
Since $6(2^{m-2}-1)+3\cdot 2 = 3 \cdot 2^{m-1}$, we obtain
$K_\GMC^{(3)}(\varphi) < \widetilde I^{(3)}(\varphi)$, see \eqref{eq:tildeI3}.
This completes the proof of Proposition~\ref{th:2ub}.

\smallskip

\begin{proof}[Proof of Lemma~\ref{th:prod2}] Our basic strategy is to approximate $K^{(2)}$
by its lattice analogue.
Figure \ref{fig:GMC-2} provides a useful reference to the underlying structure that we will explain.
In \cite{CSZ19a}, Theorem 1.4, we arrived at the Dickman renewal density $G_\theta$ as the limit
\begin{align}\label{UN-lim}
U_N(n)=\frac{\log N}{N}G_\theta(\tfrac{n}{N}) (1+o(1)),\qquad \text{as $N\to\infty$},
\end{align}
where for $n\in \N$,
\begin{align}\label{def:UN}
  U_N(n)
  := \ind_{\{n=0\}} + \sum_{k\geq 1} \, (\sigma_N^2)^{k} \!\!\!\! \sum_{0=n_0<n_1<\cdots <n_{k}=n}  \,\,
   \prod_{i=1}^{k} q_{2(n_{i}-n_{i-1})}(0)
\end{align}
with $\sigma_N^2:=\tfrac{1}{R_N}(1+\tfrac{\theta+o(1)}{\log N})$ as in \eqref{intro:sigma}
and $q_n(0)$ denoting the $n$-step transition probability from $0$ to $0$ for a simple symmetric
random walk on $\Z^2$. Moreover, the following uniform bound was established in \cite[Theorem 1.4]{CSZ19a}:
\begin{align}\label{uniformUN}
 U_N(n) \leq C \frac{\log N}{N}G_\theta(\tfrac{n}{N})\qquad \forall \,\,0<n\leq N,
\end{align}
for $C\in(0,\infty)$.
It will also be useful to recall the following asymptotic estimates for $G_\theta$ from \cite[Proposition 1.6]{CSZ19a}:
\begin{equation}\label{Gtheta_asympt}
\begin{aligned}
G_\theta(t)&= \frac{1}{t(\log\frac{1}{t})^2}\Bigg\{ 1+\frac{2\theta+o(1)}{\log\frac{1}{t}} \Bigg\}, \qquad \text{as} \,\, t\to0 \,\,\, \text{and} \\
G_\theta(t)&\leq \frac{C}{t(\log\frac{1}{t})^2},   \hskip 3.9cm \text{for} \,\, t\in [0,1].
\end{aligned}
\end{equation}

 Using the local limit theorem for random walks, the asymptotic \eqref{UN-lim}, and the bound \eqref{uniformUN} which allows us to apply dominated convergence, we have that (recall $\ev{\cdot}$ from \eqref{eq:rescZmeas})
\begin{align*}
	&K^{(2)}(z_1,z_2) \\
	&=\lim_{N\to\infty} \, \sigma_N^2 \sum_{1\leq m_1 <m_2\leq N} q_{2m_1}(\ev{(z_1-z_2)\sqrt{N}}) \, U_N(m_2-m_1)\\
	&=\lim_{N\to\infty} \, \sum_{k\geq 1} \sum_{1\leq m_1<m_2\leq N} (\sigma_N^2)^{k+1}
	 \sum_{m_1=n_0<n_1<\cdots <n_{k}=m_2}
	 q_{2m_1}(\ev{(z_1-z_2)\sqrt{N}})  \prod_{i=1}^{k} q_{2(n_{i}-n_{i-1})}(0) \\
	 &=\lim_{N\to\infty} \, \sum_{k\geq 1} (\sigma_N^2)^{k+1}
	 \sum_{0<n_0<n_1<\cdots <n_{k}\leq N}
	 q_{2n_0}(\ev{(z_1-z_2)\sqrt{N}})  \prod_{i=1}^{k} q_{2(n_{i}-n_{i-1})}(0).
\end{align*}
To lighten the notation below, we will drop the brackets $\ev{\cdot}$, i.e., when we write $z\sqrt{N}$ we refer to $\ev{z\sqrt{N}}$.
Using this expression for the product $K^{(2)}(z_1,z_2) K^{(2)}(z_2,z_3)$ we obtain that
\begin{align}\label{K2prod-lim}
        &	K^{(2)}(z_1,z_2) K^{(2)}(z_2,z_3) \notag \\
	&=\lim_{N\to\infty}  \sum_{k, k'\geq 1} (\sigma_N^2)^{k+k'+2}  \!\!\!\!
	 \sumtwo{0<n_0<n_1<\cdots <n_{k}\leq N} {0<n'_0<n'_1<\cdots <n'_{k'}\leq N}  \!\!\!\!
	 q_{2n_0}\big((z_1-z_2) \sqrt{N}\big)\,  q_{2n'_0}\big((z_2-z_3)\sqrt{N}\big) \cdot \notag \\
	 &\hskip 6cm\cdot \prod_{i=1}^{k} q_{2(n_{i}-n_{i-1})}(0) \, \prod_{i=1}^{k'} q_{2(n'_{i}-n'_{i-1})}(0)
\end{align}
Let us start by assuming that
 the sequences $\{0<n_0<n_1<\cdots <n_{k}\leq N\}$
and $\{0<n'_0<n'_1<\cdots <n'_{k'}\leq N\}$ do not share common points and
let us look at all possible ways they interlace, i.e.
\begin{align}\label{blocks}
0<n_0<\cdots<n_{\tau_1}<n'_0<\cdots < n'_{\tau'_1}< n_{\tau_1+1}<
\cdots <n_{\tau_2} < n'_{\tau'_1+1}< \cdots  < n'_{\tau'_2} \cdots
\end{align}
 for integers $\tau_1,\tau_2,\ldots\in \{1,\ldots,k\}$ and
$\tau'_1,\tau'_2,\ldots\in \{1,\ldots,k'\}$. The case $n_0'<n_0$ is similar.
We can now group together the blocks of primed or un-primed integers and sum over the possible
cardinalities of the blocks as well as the values of their elements after fixing first the vector
$(a_1,b_1,a_2,b_2,\ldots)=(n_0, n_{\tau_1}, n'_0, n'_{\tau'_1},\ldots)$, which marks the boundaries of the blocks.
Afterwards, we sum over all possible values of $(a_1,b_1,a_2,b_2,\ldots)$.
Using this decomposition in expression \eqref{K2prod-lim} we can then see that
\begin{align}\label{K1K2}
        &	K^{(2)}(z_1,z_2) K^{(2)}(z_2,z_3)
	 =\lim_{N\to\infty}   \sum_{m=2}^\infty
	\quad \sum_{0 < a_1 < b_1 < \ldots < a_m < b_m < N} \!\!\!\!\!\!\!\!  \notag \\
	& \quad \Big\{ q_{2a_1}\big((z_1-z_2)\sqrt{N}\big) \, q_{2a_2}\big((z_2-z_3)\sqrt{N}\big)
	\,+ \, q_{2a_1}\big((z_2-z_3)\sqrt{N}\big) \, q_{2a_2}\big((z_1-z_2)\sqrt{N}\big) \Big\} \\
	& \quad \qquad \cdot \sigma_N^2\,  U_N(b_1-a_1) \, \cdot \sigma_N^2 U_N(b_2-a_2)  \cdot
	\prod_{i=3}^m \sigma_N^2 \,U_N(b_i-a_i) \,q_{2(a_i-b_{i-2})}(0) \,.\notag
\end{align}
After passing to the limit using the local limit theorem for random walks and the asymptotic \eqref{def:UN}, we arrive at
expression \eqref{eq:Kprod2alt}.

It only remains to check that the interlacing blocks \eqref{blocks} are well defined, i.e.
contribution to \eqref{K2prod-lim} from sequences $\{0<n_0<n_1<\cdots <n_{k}\leq N\}$
and $\{0<n'_0<n'_1<\cdots <n'_{k'}\leq N\}$ that share common points is negligible due to the loss of some degrees of freedom.
So let us look at \eqref{K2prod-lim} when the sum on the right hand side is over configurations such that
\begin{align*}
\{0<n_0<n_1<\cdots <n_{k}\leq N\}\, \bigcap \, \{0<n'_0<n'_1<\cdots <n'_{k'}\leq N\} \neq \emptyset.
\end{align*}
By summing over $1\leq \mathfrak{n}\leq N$ where a coincidence between some $n_\ell$ and $n'_{\ell'}$ can occur, the right hand side of \eqref{K2prod-lim} can be bounded by
\begin{align*}
	&\sum_{k, k'\geq 1} (\sigma_N^2)^{k+k'+2}  \sum_{1\leq \mathfrak{n}\leq N}
	 \sumtwo{0<n_0<n_1<\cdots <n_{k}\leq N} {0<n'_0<n'_1<\cdots <n'_{k'}\leq N}   \ind_{\mathfrak{n}\in \{n_1, \ldots, n_k\}\cap\{n'_1, \ldots, n'_{k'}\}}  \\
	 &\qquad \cdot q_{2n_0}\big((z_1-z_2) \sqrt{N}\big)\,  q_{2n'_0}\big((z_2-z_3)\sqrt{N}\big)
	  \prod_{i=1}^{k} q_{2(n_{i}-n_{i-1})}(0) \, \prod_{i=1}^{k'} q_{2(n'_{i}-n'_{i-1})}(0),
\end{align*}
Rearranging terms, this can be rewritten as
\begin{equation}\label{forrevision}
\begin{split}
&\sigma_N^4 \, \sum_{1\leq \mathfrak{n}\leq N} \sum_{1\leq n_0,n'_0 \leq \mathfrak{n}} q_{2n_0}\big((z_1-z_2) \sqrt{N}\big)\,  q_{2n'_0}\big((z_2-z_3)\sqrt{N}\big)  \\
&\qquad \cdot U_N(\mathfrak{n}-n_0) U_N(\mathfrak{n}-n'_0) 
\sum_{\mathfrak{n}\leq n,n'\leq N} U_N(n-\mathfrak{n}) U_N(n'-\mathfrak{n}),
\end{split}
\end{equation}
where recall from \eqref{def:UN} that $U_N(0)=1$. First restrict to the case $n_0, n_0', n, n'\neq \mathfrak{n}$. Using \eqref{uniformUN}, this can be bounded by
 \begin{align*}
&C\sigma_N^4 \,\sum_{1\leq n_0,n'_0 \leq N}q_{2n_0}\big((z_1-z_2) \sqrt{N}\big)\,  q_{2n'_0}\big((z_2-z_3)\sqrt{N}\big) \cdot \\
& \cdot\!\!\sum_{n_0\vee n'_0 < \mathfrak{n} \leq N} \frac{\log N}{N} G_\theta\big(\frac{\mathfrak{n}-n_0}{N}\big) \cdot
\frac{\log N}{N} G_\theta\big(\frac{\mathfrak{n}-n'_0}{N}\big)
\sum_{\mathfrak{n}<n,n'\leq N}  \frac{\log N}{N} G_\theta\big(\frac{n-\mathfrak{n}}{N}\big)
 \frac{\log N}{N} G_\theta\big(\frac{n'-\mathfrak{n}}{N}\big).
\end{align*}
We now show that this sum goes to $0$ as $N\to\infty$. Using the local limit theorem, we can approximate the above sum by
 \begin{align*}
&\frac{C\sigma_N^4}{N^2} \,\sum_{1\leq n_0,n'_0 \leq N}g_{\tfrac{2n_0}{N}}\big( z_1-z_2\big)\,
g_{\tfrac{2n'_0}{N}}\big( z_2-z_3 \big) \cdot \\
& \cdot\!\!\sum_{n_0\vee n'_0 < \mathfrak{n} \leq N} \frac{\log N}{N} G_\theta\big(\frac{\mathfrak{n}-n_0}{N}\big) \cdot
\frac{\log N}{N} G_\theta\big(\frac{\mathfrak{n}-n'_0}{N}\big)
\sum_{\mathfrak{n}<n,n'\leq N}  \frac{\log N}{N} G_\theta\big(\frac{n-\mathfrak{n}}{N}\big)
 \frac{\log N}{N} G_\theta\big(\frac{n'-\mathfrak{n}}{N}\big).
\end{align*}
Note that we have five independent summation variables $n_0,n'_0, \mathfrak{n}, n, n'$,
compared to six factors of $N^{-1}$.
Using a Riemann sum approximation and that $\sigma_N^4 =O((\log N)^{-2})$, we can further bound the above sum by
\begin{align*}
&\frac{C (\log N)^2}{N} \int_0^1 \dd t_0 \int_0^1  \dd t'_0 \,\, g_{2t_0}(z_1-z_2) \,  g_{2t'_0}(z_2-z_3) \cdot \\
& \qquad \cdot \int_{t_0\vee t'_0}^1 \dd \mathfrak{t} \,\, G_\theta(\mathfrak{t}-t_0) G_\theta(\mathfrak{t}-t'_0)
\int_{\mathfrak{t}}^1 \dd t \int_{\mathfrak{t}}^1 \dd t' \,\, G_\theta(t-\mathfrak{t}) G_\theta(t'-\mathfrak{t}).
\end{align*}
The asymptotics of $G_\theta$ from \eqref{Gtheta_asympt} show that all integrals involving $G_\theta$ are finite,
and so are the integrals involving the heat kernels for $z_1\neq z_2\neq z_3$.
Thus, the whole quantity vanishes at the speed of $O(\tfrac{(\log N)^2}{N})$ as $N$ tends to infinity.

Finally, we consider the sum in \eqref{forrevision} when $\mathfrak{n}$ coincides with at least one element in $\{n_0, n'_0, n, n'\}$,
in which case a corresponding sum of $U_N$ in \eqref{forrevision} is replaced by $1$, which yields a better bound. We illustrate this in the case
$n_0=n'_0=\mathfrak{n}$; the other cases are similar and will be omitted. The quantity in
\eqref{forrevision} now becomes
\begin{align*}
&\sigma_N^4 \,\sum_{1\leq \mathfrak{n} \leq N}
q_{2\mathfrak{n}}\big((z_1-z_2) \sqrt{N}\big)\,  g_{2\mathfrak{n}}\big((z_2-z_3)\sqrt{N}\big)
\sum_{\mathfrak{n}<n,n'\leq N} U_N(n-\mathfrak{n}) U_N(n'-\mathfrak{n}) \\
&\leq \frac{C\sigma_N^2}{N^2} \sum_{1\leq \mathfrak{n} \leq N}
g_{\tfrac{2\mathfrak{n}}{N}}\big( z_1-z_2 \big)\,  g_{\frac{2\mathfrak{n}}{N}}\big(z_2-z_3 \big)
\sum_{\mathfrak{n}<n,n'\leq N}\frac{\log N}{N} G_\theta\big(\frac{n-\mathfrak{n}}{N}\big)
 \frac{\log N}{N} G_\theta\big(\frac{n'-\mathfrak{n}}{N}\big) \\
&\leq \frac{C }{N} \int_0^1 \dd s  \,\, g_{2s}(z_1-z_2) \,  g_{2s}(z_2-z_3)
\int_s^1 \dd t \int_s^1 \dd t' \,\, G_\theta(t-s) G_\theta(t'-s),
\end{align*}
which is $O(N^{-1})$ as all integrals above are finite by \eqref{Gtheta_asympt} and by the
small time asymptotics of the heat kernels for $z_1\neq z_2 \neq z_3$.
\end{proof}

\smallskip

\begin{proof}[Proof of Lemma~\ref{th:prod3}]
The proof is similar to that of Lemma~\ref{th:prod2}, so we will just give a sketch.

For the product
$K^{(2)}(z_1, z_2) \, K^{(2)}(z_2, z_3)\, K^{(2)}(z_1, z_3)$
we can write a formula analogous to \eqref{K2prod-lim} and \eqref{K1K2},
where we now sum over three type of blocks:
un-primed, primed and double-primed, to each one of which
we assign a label $\sigma_i \in \{12, 23, 13\}$.
Due to the interlacing of the blocks, the assignment of labels will have the
constraint that $\sigma_i \ne \sigma_{i-1}$ for all $i$. Thus, the only difference with the analogous formula
for $K^{(2)}(z_1, z_2) \, K^{(2)}(z_2, z_3)$ would be that $q_{2(a_i-b_{i-2})}(0)$ would be replaced by
$q_{2(a_i-b_{\mathrm{prev}(i)} )}(0)$ where $\mathrm{prev}(i)$ corresponds to the previous block with the same label
$\sigma$.
\end{proof}

\subsection{Proof of Proposition~\ref{th:bound-comp}}
\label{sec:bound-comp}

If $\varphi = g_r$ is the heat kernel, see \eqref{eq:gt},
we can compute $\scrG_{a_1, a_2}(\varphi)$ and $\widetilde\scrG_{a_1, a_2}(\varphi)$,
as in Remark~\ref{rem:Gauss-comp}.
We start from the latter, see \eqref{eq:Gphi2}: integrating $z_3$ by Gaussian convolution, then
$z_2$ by Lemma~\ref{triple-gauss}, and finally $z_1$, we get
\begin{equation*}
\begin{split}
	\widetilde\scrG_{a_1, a_2}(g_r)
	& := (2\pi)^2 \iint \, g_r(z_1) \, g_r(z_2)
	\ g_{a_1}(z_2-z_1) \, g_{r+a_2}(z_2) \ \dd z_1 \, \dd z_2 \\
	& = (2\pi)^2 \, g_{2r+a_2}(0) \int \, g_r(z_1) \,
	g_{a_1 + h(r, r+a_2)}(z_1) \ \dd z_1  \\
	& = (2\pi)^2 \, g_{2r+a_2}(0) \, g_{r + a_1 + \frac{r(r+a_2)}{2r+a_2}}(0)
	= \frac{1}{3 r^2 + 2(a_1+a_2)r + a_1 a_2} \,,
\end{split}
\end{equation*}
which proves the second relation in \eqref{eq:Ggr}.
We can compute $\scrG_{a_1, a_2}(g_r)$ from \eqref{eq:Gphi}
with similar arguments, but it is easier to exploit the following basic fact:
when $z_1, z_2, z_3$ are independent Gaussian
random variables on $\R^2$ with density $g_r$,
then $x := z_1 - z_2$ and $y := z_3 - \frac{z_1+z_2}{2}$ are independent
with densities $g_{2r}$ and $g_{\frac{3}{2}r}$, therefore
\begin{equation*}
\begin{split}
	\scrG_{a_1, a_2}(g_r)
	&= (2\pi)^2 \iint g_{a_1}(x) \, g_{2r}(x) \, g_{a_2}(y) \, g_{\frac{3}{2}r}(y) \, \dd x \, \dd y \\
	&= (2\pi)^2 \, g_{a_1+2r}(0) \, g_{a_2+\frac{3}{2}r}(0)
	= \frac{1}{a_1+2r} \, \frac{1}{a_2 + \frac{3}{2}r} \,,
\end{split}
\end{equation*}
which proves the first relation in \eqref{eq:Ggr}.
The fact that $\scrG_{a_1, a_2}(g_r)  > \widetilde{\scrG}_{a_1, a_2}(g_r)$ then follows.

\smallskip

It remains to prove \eqref{eq:goalGG} when $\varphi(z)
= \ind_{\{|z| < r\}}$ is the indicator function of a ball.
If we define
\begin{equation*}
	\xi(z) := (\varphi * g_{a_2})(z) = \int \varphi(z') \, g_{a_2}(z - z') \, \dd z' \,,
\end{equation*}
then we can write, recalling \eqref{eq:Gphi} and performing a change of variables,
\begin{align*}
	\scrG_{a_1, a_2}(\varphi)
	& := (2\pi)^2
	\iint\limits_{(\R^2)^2} \,
	\varphi(z_1) \, \varphi(z_2) \ g_{a_1}(z_2-z_1) \, \xi(\tfrac{z_1+z_2}{2})
	\ \dd z_1 \, \dd z_2 \,, \\
	& = (2\pi)^2
	\iint\limits_{(\R^2)^2} \, \varphi(z-\tfrac{y}{2}) \, \varphi(z+\tfrac{y}{2})
	\, g_{a_1}(y) \, \xi(z) \, \dd y \, \dd z \,.
\end{align*}
Similarly, by \eqref{eq:Gphi2},
\begin{equation*}
	\widetilde\scrG_{a_1, a_2}(\varphi) := (2\pi)^2
	\iint\limits_{(\R^2)^2} \,
	\varphi(z-y) \, \varphi(z) \ g_{a_1}(y) \, \xi(z)
	\ \dd y \, \dd z \,.
\end{equation*}
Note that $\xi$ is a radially symmetric and strictly increasing function since
the convolution of two radially symmetric and non-increasing functions
(see the proof of Lemma~\ref{th:decre}).
We can apply a layer cake decomposition for $\xi$ as in \eqref{eq:layer-cake}, thus replacing
$\xi(z)$ by $\ind_{\{|z| < t\}}$ with $t$ integrated
w.r.t.\ the measure $\mu^\xi({\rm d}t)$, which has full support on $[0,\infty)$.
Plugging also $\varphi(x) = \ind_{\{|x|<r\}}$, we can write the contribution at each fixed $t>0$ by
\begin{equation*}
\begin{split}
	\scrG^{(t)}_{a_1, a_2}(\varphi) - \widetilde\scrG^{(t)}_{a_1, a_2}(\varphi)
	: = (2\pi)^2 \int\limits_{\R^2} \,
	\Big\{ &\, \mathrm{Leb}\big( B(\tfrac{y}{2},r) \cap B(-\tfrac{y}{2},r)
	\cap B(0,t \big) \\
	& \ - \mathrm{Leb}\big( B(y,r) \cap B(0,r)
	\cap B(0,t \big) \Big\}  \, g_{a_1}(y) \, \dd y \,,
\end{split}
\end{equation*}
where $B(z,r) := \{x\in\R^2 \colon |x| < r\}$ is the ball of radius~$r$ centered at~$z$.
Note that
\begin{equation*}
	A_r(y) := B(\tfrac{y}{2},r) \cap B(-\tfrac{y}{2},r) \,,
\end{equation*}
is a \emph{symmetric convex set} (possibly empty), which translated by $\frac{y}{2}$ gives
\begin{equation*}
	A_r(y) + \tfrac{y}{2} = B(y,r) \cap B(0,r) \,.
\end{equation*}
Then it follows from Anderson's inequality \cite[Theorem~1]{And55}  that we have the bound
\begin{equation*}
\begin{split}
	\mathrm{Leb}\big( A_r(y) \cap B(0,t )\big) \,\ge\,
	\mathrm{Leb}\big( ( A_r(y) + \tfrac{y}{2})  \cap B(0,t) \big) \,,
\end{split}
\end{equation*}
which can also be checked directly, and given $r$, the inequality is strict for a non-empty open set of $t$ and $y$.
Integrating $t$ w.r.t.\ $\mu^\xi$ and $y$ w.r.t.\ $g_{a_1}(y) {\rm d}y$ then gives $\scrG_{a_1, a_2}(\varphi)  > \widetilde\scrG_{a_1, a_2}(\varphi) $
when $\varphi$ is the indicator function of a ball.\qed


\section{Proof of Theorem~\ref{th:mmom}:
lower bounds via collision local times and the Gaussian correlation inequality}
\label{via-gci}

In this section we prove Theorem~\ref{th:mmom}.
The key point is the lower bound \eqref{eq:mombd0} on the moments of the SHF
$\scrZ_t^\theta$: for a suitable $\eta = \eta_{t,\theta} > 0$ we have, for any
$m \in \N$ with $m \ge 3$,
\begin{equation}\label{eq:mombd1}
	\bbE\big[ \big(2\,\scrZ_t^\theta(g_\delta) \big)^m \big]
	\geq (1+\eta) \, \bbE\big[\big(2\, \scrZ_t^\theta(g_\delta)\big)^2 \big]^{m\choose 2}
	\qquad \forall \delta \in (0,1) \,,
\end{equation}
where $g_\delta$ is the heat kernel on $\R^2$, see \eqref{eq:gt}.
Then, in order to obtain \eqref{eq:asystrict} and complete the proof,
it suffices to show that \eqref{eq:factorization} holds,
which follows from the next result.

\begin{proposition}[Higher moments of GMC]\label{P:momasy}
Let $\mathscr{M}_{t}^\theta(\dd x)$ be the GMC with
the same first and second moments as the SHF
$\scrZ_{t}^\theta(\dd x)$, see Section \ref{ss:GMC}. Then, as $\delta \downarrow 0$, we have
\begin{equation}\label{eq:momasy}
	\bbE\big[ \big(2\,\scrM_t^\theta(g_\delta) \big)^m \big]
	\sim \big( C_{t,\theta}\,
	\log \tfrac{1}{\sqrt{\delta}}\big)^{m\choose 2} \,,
\end{equation}
where $C_{t,\theta} = \frac{1}{\pi} \int_0^t G_\theta(v) \, \dd v$ is the
same constant which appears in \eqref{eq:rhot}.
\end{proposition}

The rest of this section is divided in three parts.
\begin{itemize}
\item First we show that the moments of the SHF $\scrZ_t^\theta(\dd x)$,
averaged over a test function $\varphi$, can be obtained as the limit
(as $\epsilon \downarrow 0$) of the moments
of the solution $u^\epsilon(t,x)$ of the mollified Stochastic Heat Equation
\eqref{eq:SHE0}, based on \cite{GQT21}.

\item Then we prove the bound \eqref{eq:mombd1} by exploiting
the \emph{Gaussian Correlation Inequality} \cite{R14, LM17},
adapting the approach in Feng's Ph.D. thesis \cite{Feng}.

\item Finally, we prove Proposition~\ref{P:momasy}, which completes the proof of
Theorem~\ref{th:mmom}.
\end{itemize}

\subsection{SHF and the mollified Stochastic Heat Equation}
\label{sec:SHF-SHE}
We consider the mollified Stochastic Heat Equation \eqref{eq:SHE0} with spatially mollified
space-time white noise
\begin{equation*}
	\xi^\eps(t,x) := (\xi(t,\cdot) * j_\eps)(x) = \int_{\R^2} j_\eps(z) \, \xi(t,x-z) \, \dd z \,,
\end{equation*}
where $j_\eps(x):=\epsilon^{-2} j(x/\epsilon)$
and $j(\cdot)$ is a probability density on $\R^2$
(usually taken compactly supported).
Assuming initial condition $u^\eps(0,\cdot)=1$,
by the Feynman-Kac formula~\cite[Section~3 and eq. (3.22)]{BC95}, the
It\^o solution $u^\eps(t,x) = u^\eps_\beta(t,x)$
of \eqref{eq:SHE0}, where we highlight the dependence on~$\beta$,
has the representation
\begin{equation}\label{eq:SHE2}
	u^\eps_\beta(t,x)
	= \bE_x\Big[\rme^{\beta \int_0^t \xi^\eps(t-u, B_u)\,\dd u  \,- \,\frac{1}{2}\beta^2
	\Vert j_\eps\Vert_2^2 t}  \Big]
	\stackrel{\rm dist}{=} \bE_x\Big[\rme^{\beta \int_0^t \xi^\eps(u, B_u)\,\dd u
	\, - \, \frac{1}{2}\beta^2
	\Vert j_\eps\Vert_2^2 t}  \Big],
\end{equation}
where $\bE_x$ denotes expectation
for a standard Brownian motion $B$ starting at $x$.
We will omit $x$ from $\bE_x$ if $x=0$.

We can directly compute the moments of $u^\eps_\beta(t,x)$, which do not depend on~$x$
by translation invariance,
thanks to the initial condition $u(0,\cdot)\equiv 1$. Given $m\in\N$, let
$(B^{(i)})_{1\leq i\leq m}$ denote $m$ independent Brownian motions,
and define $J_\epsilon := \epsilon^{-2} J(x/\epsilon)$ with
$J := j * j$. Note that
\begin{equation} \label{eq:L}
\begin{split}
	\bbvar \bigg[ \sum_{i=1}^m \int_0^t  \xi^\eps(u, B^{(i)}_u)\,\dd u \bigg]
	= \sum_{1 \le i,j \le m} L^{i,j}_{\epsilon,t} \,, \quad
	\ \ \text{where} \quad
	L^{i,j}_{\epsilon,t} :=
	\int_0^t   J_\epsilon(B^{(i)}_u-B^{(j)}_u) \,\dd u\, ,
\end{split}
\end{equation}
which can be viewed as a
\emph{collision local time} at scale~$\eps$ between $B^{(i)}$ and $B^{(j)}$.
Note that $L^{i,i}_{\eps,t} = J_\eps(0) \, t = \|j_\eps\|_2^2 \, t$,
where $\|\cdot\|_2$ denotes the $L^2$ norm.
Given $x_1, \ldots, x_m \in \R^2$, if we denote by $\bP_{\vec{x}}$ the law under which
$B^{(i)}$ starts at $B^{(i)}_0 = x_i$, a Gaussian computation yields
\begin{equation}\label{eq:uepsk}
	\bbE\Bigg[\prod_{i=1}^m u^{\eps}_\beta(t,x_i) \Bigg]
	= \bbE\bE_{\vec{x}}\Big[ \rme^{\beta \sum_{i=1}^m \int_0^t
	\xi^\eps(u, B^{(i)}_u)\,\dd u \, - \,
	\frac{m}{2}\beta^2 \Vert j_{\eps}\Vert_2^2 t} \Big]
	= \bE_{\vec{x}}\Bigg[\prod_{1\leq i<j\leq m}
	\rme^{\beta^2 \int_0^t L^{i,j}_{\epsilon,t}} \Bigg] .
\end{equation}

\begin{remark}
In the critical window \eqref{beta-eps}
we have $\beta_\eps^2 \sim 2\pi/\log\epsilon^{-1}$, hence
$\beta_\eps^2 L^{i,j}_{\eps, t}$ for $i\ne j$
converges in law as $\eps \downarrow 0$
to an exponential random variable $Y$
of mean~$1$, by a classical result \cite{KR53}. This explains why $\beta_\eps$ is critical,
since $\E[\rme^{\lambda Y}]$ diverges precisely at $\lambda=1$.
\end{remark}

We now describe the link between the solution $u^\epsilon_\beta(t,x)$
of the mollified Stochastic Heat
Equation and the SHF $\scrZ_t^\theta(\dd x)$.
We recall that the latter was obtained in \cite{CSZ23}
from the directed
polymer random measure $\cZ^{\beta}_{N; t}(\dd x)
= \cZ^{\beta}_{N;\, 0,t}(\dd x, \R^2)$,
see \eqref{eq:rescZmeas},
based on the \emph{simple random walk} $(S_n)$ on $\Z^2$,
which has covariance matrix
$\sfs I$ with $\sfs=\frac{1}{2}$ and is periodic
(note that $S_{2n}$ takes values in $\Z^2_\even$, see \eqref{eq:even}).
On the other hand, the solution $u^\epsilon_\beta(t,x)$
of the mollified Stochastic Heat Equation, see \eqref{eq:SHE2},
is based on a standard Brownian motion on $\R^2$
with covariance matrix $I$ and, of course, with no periodicity issues.

For these reasons, to obtain the SHF $\scrZ_t^\theta(\dd x)$
from the solution $u^\epsilon_\beta(t,x)$ of the mollified
Stochastic Heat Equation, we need an \emph{appropriate rescaling}:
given $\theta \in \R$, if we scale $\beta_\eps = \beta_\eps(\theta)$ in the
critical window \eqref{eq:betaepssimple}-\eqref{eq:cReps}
(see also \eqref{eq:cReps-as}-\eqref{beta-eps2}),
we expect that
\begin{equation} \label{eq:conj}
	\frac{1}{2} \, u^\epsilon_{\beta_\eps}
	\big( t, x\sqrt{2} \big) \, \dd x \
	\xrightarrow[]{\ \ d \ \ } \ \scrZ_t^\theta(\dd x) \,,
\end{equation}
see Appendix~\ref{sec:comparison} for a heuristic derivation.
We refrain from proving such a convergence, which we expect to follow
from the same techniques as in the paper \cite{CSZ23}.
As a matter of fact, for our goals, it is enough to
show that the two sides of \eqref{eq:conj} have asymptotically
\emph{the same moments},
and this follows by \cite{GQT21} and \cite{CSZ23}, as we now describe.

\begin{proposition}[Moments of SHF and Stochastic Heat Equation]\label{th:SHF-SHE}
Fix $\theta \in \R$
and set $\beta = \beta_\eps$ as in \eqref{beta-eps2}.
Fix a mollification density $j(\cdot)$ which is
radially symmetric and non-increasing.
For any integrable $\varphi: \R^2 \to \R$, and for any $h\in\N$, we have
\begin{equation} \label{eq:moments-SHF-SHE}
\begin{split}
	\bbE\big[ \scrZ_t^\theta(\varphi)^h \big]
	& \,=\, \frac{1}{2^h} \, \lim_{\eps\downarrow 0} \,
	\bbE\bigg[ \bigg(\int_{\R^2} u^\eps_{\beta_\eps}\big(t,x\sqrt{2}\big)
	\, \varphi(x) \, \dd x\bigg)^h \bigg] \,.
\end{split}
\end{equation}
\end{proposition}

\begin{proof}
It is enough to compare formulas \eqref{eq:Zmomh}-\eqref{eq:Zmomh-kernel}
with Theorem~1.1 and eq.~(2.5) in \cite{GQT21}.
\end{proof}

\begin{remark}
Recalling \eqref{eq:Zmom2},
we see that relation \eqref{eq:moments-SHF-SHE} for $h=2$ reduces to
\begin{equation}\label{eq:variance-check}
	\iint\limits_{(\R^2)^2} \varphi(x) \, \varphi(x')
	\, K_t^{(2)}(x,x') \, \dd x \, \dd x'
	\,=\,
	\lim_{\eps \downarrow 0} \,
	\bbvar\bigg[ \int_{\R^2} u^\eps_{\beta_\eps}(t,x\sqrt{2})
	\, \varphi(x) \, \dd x \bigg] \,.
\end{equation}
The validity of such a relation was proved in \cite[Theorem~1.7]{CSZ19b}
(note that the choice of $\theta$ in \eqref{eq:defCReps}-\eqref{beta-eps2},
which enters $K_t^{(2)} = K_{t,\theta}^{(2)}$ in \eqref{eq:Zmom2},
matches \cite[eq.~(1.38)]{CSZ19b}).
\end{remark}

\subsection{Proof of the lower bound \eqref{eq:mombd1}}

Henceforth we fix $\beta = \beta_\eps$ as in
\eqref{eq:betaepssimple}-\eqref{eq:cReps}
and omit it from notation, i.e.\ we set
$u^\eps(t,x) := u^\eps_{\beta_\eps}(t,x)$.
It follows by \eqref{eq:L}-\eqref{eq:uepsk} that
\begin{equation}\label{eq:mombd1app}
	\bbE\bigg[ \bigg(\int u^\eps(t,x\sqrt{2}) \, \varphi(x)
	\, \dd x\bigg)^m \bigg]
	= \!\!\! \int\limits_{(\R^2)^m} \, \prod_{i=1}^m \varphi (x_i)
	\, \bE_{\vec x\sqrt{2}}
	\bigg[\prod_{1\leq i<j\leq m} \rme^{\beta_\eps^2 \int_0^t
	J_\eps(B^{(i)}_s-B^{(j)}_s) \,\dd s}\bigg] \dd  \vec x \,,
\end{equation}
where we recall that $\bE_{\vec y}$ denotes expectation w.r.t.\
$m$ independent Brownian motions
with $B^{(i)}_0=y_i$.
We now take $\varphi = g_\delta$ to be the heat kernel, see \eqref{eq:gt},
and note that by diffusive scaling
we can write $g_\delta(x) = 2 \, g_{2\delta}(x\sqrt{2})$.
Then, in view of \eqref{eq:moments-SHF-SHE} and by a change of variables,
to prove \eqref{eq:mombd1} it suffices to find $\eta = \eta_{t,\theta} >0$ such
that, uniformly in $m\geq 3$ and $\delta\in (0,1)$,
\begin{equation}\label{eq:mombd2}
\begin{aligned}
	& \lim_{\eps\downarrow 0} \
	\int\limits_{(\R^2)^m} \, \prod_{i=1}^m g_{2\delta} (x_i) \,
	\bE_{\vec x}\bigg[\prod_{1\leq i<j\leq m} \rme^{\beta_\eps^2
	\int_0^t J_\eps (B^{(i)}_s-B^{(j)}_s) \,\dd s}\bigg] \dd  \vec x \\
	& \quad \geq \ (1+\eta) \,
	\lim_{\eps\downarrow 0} \Bigg( \, \int\limits_{(\R^2)^2} g_{2\delta}(x_1) \,
	g_{2\delta}(x_2)  \, \bE_{\vec x}\bigg[\rme^{\beta_\eps^2
	\int_0^t J_\eps (B^{(1)}_s-B^{(2)}_s) \,\dd s}\bigg] \,
	\dd  x_1 \, \dd x_2 \Bigg)^{m\choose 2}.
\end{aligned}
\end{equation}

We will adapt the approach in Feng's thesis \cite{Feng}, which used the Gaussian
correlation inequality \cite{R14, LM17} to prove an analogue of \eqref{eq:mombd2}
for $m=3$ with $g_\delta(\cdot)$ replaced by $\delta_0(\cdot)$.
Unfortunately, not much could be concluded in that case,
because  \emph{all moments $\bbE[u^\eps(t,0)^m]$ of order $m>1$
diverge as $\eps \downarrow 0$}: this is due to the fact that
$u^\eps(t, 0) \to 0$ in distribution
for $\beta=\beta_\eps$ in the critical window \eqref{beta-eps},
see \cite[Theorem~2.15]{CSZ17b},
while $\bbE[u^\eps(t,0)] \equiv 1$ stays constant.
We will show that the Gaussian correlation inequality can still be applied
when we average $u^\eps(t,x)$ w.r.t.\ $g_\delta$, which will lead
to the interesting bound \eqref{eq:mombd1}.

Let $Z_{2\delta}^{(1)}, \ldots, Z_{2\delta}^{(m)}$
be i.i.d.\ normal random variables on $\R^2$
with probability density $g_{2\delta}$, independent of the Brownian motions
$B^{(1)}, \ldots, B^{(m)}$ all starting from~$0$.
Denoting by $\bE$ expectation w.r.t.\ their joint law,
we can rewrite \eqref{eq:mombd2} as
\begin{equation}\label{eq:mombd3}
\begin{aligned}
	& \lim_{\eps\downarrow 0} \, \bE\bigg[\prod_{1\leq i<j\leq m}
	\rme^{\beta_\eps^2 \int_0^t J_\eps
	(Z_{2\delta}^{(i)}+B^{(i)}_s- Z_{2\delta}^{(j)} -B^{(j)}_s) \,\dd s}\bigg]  \\
	& \quad
	\ge (1+\eta) \, \lim_{\eps\downarrow 0} \, \bE\Big[\rme^{\beta_\eps^2
	\int_0^t J_\eps (Z_{2\delta}^{(1)} + B^{(1)}_s- Z_{2\delta}^{(2)}- B^{(2)}_s) \,\dd s}\Big]^{m\choose 2} \,.
\end{aligned}
\end{equation}
Next we Taylor expand the exponential in the l.h.s.: for each $i<j$, we write
\begin{align*}
	& \rme^{\beta_\eps^2 \int_0^t J_\eps (Z_{2\delta}^{(i)}+B^{(i)}_s-Z_{2\delta}^{(j)}-B^{(j)}_s) \,\dd s}
	= 1 + \sum_{n=1}^\infty \beta_\eps^{2n} \idotsint\limits_{0<s_1<\cdots <s_n<t}
	\prod_{l=1}^n J_\eps (Z_{2\delta}^{(i)} + B^{(i)}_{s_l} - Z_{2\delta}^{(j)}- B^{(j)}_{s_l}) \, \dd  \vec s \\
	& \qquad
	= 1 + \sum_{n=1}^\infty \beta_\eps^{2n} \idotsint\limits_{0<s_1<\cdots <s_n<t
	\atop y_1, \ldots, y_n>0}
	\prod_{l=1}^n \ind_{A_\eps(y_l)}(Z_{2\delta}^{(i)} + B^{(i)}_{s_l} - Z_{2\delta}^{(j)}- B^{(j)}_{s_l}) \,
	\dd \vec s \, \dd \vec y \,,
\end{align*}
where we used the decomposition $J_\eps(x) = \int_0^\infty \ind_{A_\eps(y)}(x)
\, \dd y$, with
\begin{equation}
	A_\eps(y) := \{x \in \R^2: \, J_\eps(x) \geq y\} \,.
\end{equation}
Note that $J := j * j$ is a radially symmetric and non-increasing function,
as the convolution of two radially symmetric and
non-increasing functions, as we showed in
the proof of Lemma~\ref{th:decre}. It follows that
the set $A_\eps(y)$ is a ball centered at the origin, for any $y > 0$.

We can substitute this Taylor expansion into the l.h.s.\ of \eqref{eq:mombd3} to obtain
\begin{equation}\label{eq:expansion}
	\bE\Bigg[ \prod_{1\leq i<j\leq m} \bigg(1 + \sum_{n=1}^\infty \beta_\eps^{2n}
	\idotsint\limits_{0<s_1<\cdots <s_n<t \atop y_1, \ldots, y_n>0}
	\prod_{l=1}^n \ind_{A_\eps(y_l)}(Z_{2\delta}^{(i)} + B^{(i)}_{s_l} - Z_{2\delta}^{(j)}- B^{(j)}_{s_l}) \,
	\dd \vec s \, \dd \vec y\bigg) \Bigg] \,,
\end{equation}
which, upon expansion, leads to a positive mixture of terms of the form
\begin{equation}\label{eq:mombd4}
	\bE\Bigg[ \prod_{(i,j) \in \cI} \, \prod_{l=1}^{n^{(i,j)}}
	\ind_{A_\eps(y^{(i,j)}_{l})}
	(Z_{2\delta}^{(i)} + B^{(i)}_{s^{(i,j)}_l} - Z_{2\delta}^{(j)}- B^{(j)}_{s^{(i,j)}_{l}})  \Bigg],
\end{equation}
where $\cI \subset \{(i,j): 1\leq i<j\leq m\}$ and, for each $(i,j)\in \cI$,
we have $n^{(i,j)} \in \N$ as well as $0<s^{(i,j)}_1 <\cdots <s^{(i,j)}_{n^{(i,j)}}<t$
and $y^{(i,j)}_1, \ldots, y^{(i,j)}_{n^{(i,j)}}>0$. Note that
$$
	\cW_{\cI, \vec s, \vec n}
	:=\Big((Z_{2\delta}^{(i)})_{1\leq i\leq m} \,, \, \big(B^{(i)}_{s^{(i,j)}_l}
	\,,\,
	B^{(j)}_{s^{(i,j)}_{l}}\big)_{(i,j)\in \cI, \, 1 \le l \le n^{(i,j)}} \Big)
$$
is a centered multi-dimensional Gaussian random vector.
Since $A_\eps(y)$ is a convex set symmetric about the origin (in fact, a ball),
we can apply the celebrated \emph{Gaussian correlation inequality} \cite{R14, LM17}
to lower bound
\eqref{eq:mombd4} by
\begin{equation}\label{eq:mombd5}
\begin{aligned}
	& \bE\Bigg[ \prod_{(i,j) \in \cI\cap \{(1,2), (1,3)\}} \, \prod_{l=1}^{n^{(i,j)}}
	\ind_{A_\eps(y^{(i,j)}_{l})}(Z_{2\delta}^{(i)} + B^{(i)}_{s^{(i,j)}_l} -
	Z_{2\delta}^{(j)}- B^{(j)}_{s^{(i,j)}_{l}})  \Bigg] \\
	& \qquad \times \prod_{(i,j) \in \cI \atop (i,j)\neq (1,2), (1,3)} \,
	\bE\Bigg[\prod_{l=1}^{n^{(i,j)}}
	\ind_{A_\eps(y^{(i,j)}_{l})}(Z_{2\delta}^{(i)} + B^{(i)}_{s^{(i,j)}_l} -
	Z_{2\delta}^{(j)}- B^{(j)}_{s^{(i,j)}_{l}})  \Bigg] \,,
\end{aligned}
\end{equation}
where we have kept the factors from $(i,j)=(1,2)$ and $(1,3)$ inside the same expectation,
while separating all other factors involving different $(i,j)\in \cI$.

Substituting the bound \eqref{eq:mombd5} back into the expansion of \eqref{eq:expansion}
gives a lower bound on the l.h.s.\ of \eqref{eq:mombd3}, namely
\begin{equation}\label{eq:mombd6}
\begin{aligned}
	& \bE\bigg[\prod_{1\leq i<j\leq m} \rme^{\beta_\eps^2 \int_0^t J_\eps
	(Z_{2\delta}^{(i)}+B^{(i)}_s- Z_{2\delta}^{(j)} -B^{(j)}_s) \,\dd s}\bigg]  \\
	& \ \, \geq \bE\bigg[\prod_{j=2,3} \rme^{\beta_\eps^2 \int_0^t
	J_\eps (Z_{2\delta}^{(1)}+B^{(1)}_s- Z_{2\delta}^{(j)} -B^{(j)}_s) \,\dd s}\bigg]  \,
	\bE\bigg[\rme^{\beta_\eps^2 \int_0^t
	J_\eps (Z_{2\delta}^{(1)} + B^{(1)}_s- Z_{2\delta}^{(2)}- B^{(2)}_s) \,\dd s}\bigg]^{{m\choose 2}-2} \,.
\end{aligned}
\end{equation}
Then the proof of \eqref{eq:mombd2}, and hence \eqref{eq:mombd1},
is complete once we prove the following Lemma.

\begin{lemma} \label{L:lowerbd}
There exits $\eta = \eta_{t,\theta} >0$ such that, uniformly in $\delta \in (0,1)$, we have
\begin{equation}\label{eq:mombd7}
\begin{aligned}
	\liminf_{\eps\downarrow 0} \frac{\bE\Big[\prod_{j=2,3}
	\rme^{\beta_\eps^2 \int_0^t J_\eps (Z_{2\delta}^{(1)}+B^{(1)}_s
	- Z_{2\delta}^{(j)} -B^{(j)}_s) \,\dd s}\Big]}{\bE\Big[
	\rme^{\beta_\eps^2 \int_0^t J_\eps (Z_{2\delta}^{(1)} +
	B^{(1)}_s- Z_{2\delta}^{(2)}- B^{(2)}_s) \,\dd s}\Big]^2} \geq\ 1+\eta \,.
\end{aligned}
\end{equation}
\end{lemma}

\begin{proof}
Let us define $W^{(i)}_s= Z_{2\delta}^{(i)}+B^{(i)}_s$ and
$W^{(i)} = (W^{(i)}_s)_{0 \le s \le t}$. We  introduce the shortcuts
\begin{align*}
	\Phi_{\eps, \delta}(W^{(1)})
	&:= \bE\Big[\rme^{\beta_\eps^2 \int_0^t J_\eps (W^{(1)}_s - W^{(2)}_s) \,\dd s}\,
	\Big|\, W^{(1)}\Big] \,, \\
	\widebar\Phi_{\eps, \delta}(W^{(1)}_0, W^{(1)}_t)
	& := \bE\big[ \Phi_{\eps, \delta}(W^{(1)}) \big| W^{(1)}_0, W^{(1)}_t \big] \,,
\end{align*}
so that the ratio in the l.h.s.\ of \eqref{eq:mombd7} can be written as
\begin{equation*}
	\frac{\bE[\Phi_{\eps, \delta}(W^{(1)})^2]}{\bE[\Phi_{\eps, \delta}(W^{(1)})]^2}
	= \frac{\bE[\bE[\Phi_{\eps, \delta}(W^{(1)})^2 | W^{(1)}_0,
	W^{(1)}_t]]}{\bE[\bE[\Phi_{\eps, \delta}(W^{(1)}) | W^{(1)}_0,
	W^{(1)}_t]]^2}
	\geq \frac{\bE[\widebar \Phi_{\eps, \delta}
	(W^{(1)}_0, W^{(1)}_t)^2]}{\bE[\widebar\Phi_{\eps, \delta}(W^{(1)}_0, W^{(1)}_t)]^2} \,,
\end{equation*}
by Jensen's inequality.
Therefore it suffices to show that, uniformly for $\delta \in (0,1)$,
\begin{equation}\label{eq:mombd8}
	\liminf_{\eps\downarrow 0} \,
	\bE\Bigg[\Bigg(\frac{\widebar\Phi_{\eps, \delta}(W^{(1)}_0, W^{(1)}_t)}
	{\bE[\widebar\Phi_{\eps, \delta}(W^{(1)}_0, W^{(1)}_t)]}\Bigg)^2\Bigg] \geq 1+\eta \,.
\end{equation}

Let us show that the fraction in the l.h.s.\
has a limit as $\eps\downarrow 0$. We treat separately numerator
and denominator, starting from the latter:
by \eqref{eq:mombd1app} with $m=2$ and $\varphi = g_\delta$,
\begin{equation*}
	\bE[\widebar\Phi_{\eps, \delta}(W^{(1)}_0, W^{(1)}_t)]
	= \bbE\bigg[ \bigg( \! \int \! u^\eps(t,x\sqrt{2}) \, g_{\delta}(x) \dd x
	\bigg)^2 \bigg]
\end{equation*}
hence by \eqref{eq:moments-SHF-SHE} with $h=2$, recalling \eqref{eq:Zmom2}, we get
\begin{equation}\label{eq:Phi1}
\begin{aligned}
	\widebar \Phi_\delta
	& \,:=\,
	\lim_{\eps\downarrow 0} \,
	\bE[\widebar\Phi_{\eps, \delta}(W^{(1)}_0, W^{(1)}_t)]
	\,= \, 4 \, \bbE\big[ \scrZ^\theta_t(g_\delta)^2 \big] \\
	& \,=\, 1 + \iint\limits_{(\R^2)^2}
	g_\delta(x_1) \, g_\delta(x_2) \, K^{(2)}_{t} (x_1,x_2) \,
	\dd x_1 \, \dd x_2 \\
	& \,=\, 1 + 2\pi \, \iint\limits_{0 < s < u < t}
	g_{2\delta+s}(0) \,  G_\theta(u-s) \,
	\dd s\, \dd u
	\ \underset{\delta\downarrow 0}{\sim} \
	\bigg(\int_0^t G_\theta(u) \, \dd u \bigg) \log \frac{1}{\delta} \,.
\end{aligned}
\end{equation}

Next we focus on the numerator:
in analogy with \eqref{eq:uepsk}, we can write
\begin{equation}\label{eq:barPhi}
\begin{aligned}
	\widebar\Phi_{\eps, \delta} (x_1, y)
	& := \bE\Big[\rme^{\beta_\eps^2 \int_0^t J_\eps (W^{(1)}_s - W^{(2)}_s) \,\dd s}
	\,\Big|\, W^{(1)}_0=x_1, W^{(1)}_t=y\Big] \\
	& = \int_{\R^2} g_{2\delta}(x_2)\,
	\bbE[ u^{\eps}(t,x_1 | y) \, u^\eps(t,x_2)] \, \dd x_2,
\end{aligned}
\end{equation}
where we define $u^{\eps}(t,x_1 | y)$ as a modification of
the Feynman-Kac formula \eqref{eq:SHE2}:
\begin{equation}\label{eq:SHE2bis}
	u^{\eps}(t,x_1 | y)
	:= \bE_{x_1} \Big[\rme^{\beta \int_0^t \xi^\eps(u, B_u)\,\dd u
	\, - \, \frac{1}{2}\beta^2
	\Vert j_\eps\Vert_2^2 t} \,\Big|\, B_t = y  \Big]
\end{equation}
(we recall that
$\bE_{x_1}$ is the expectation for a brownian motion $B$ started at $B_0= x_1$,
so that conditioning on $B_t = y$ yields a Brownian bridge). In \cite[Theorem 1.7 \& Section~8]{CSZ19b},
a formula for $\lim_{\eps\downarrow 0} \E[(\int \phi(x) u^\eps(t,x))^2]$ was derived using chaos expansion and renewal
type arguments. The same arguments can be adapted to show that
\begin{equation*}
\begin{aligned}
	&  \lim_{\eps\downarrow 0} \bbE[u^{\eps}(t,x_1|y) \, u^{\eps}(t,x_2)] \\
	& \quad = 1+  4\pi
	\idotsint\limits_{\substack{z, w \in \R^2 \\ 0<s<u<t}}  g_{s}(z-x_1) \,
	g_{s}(z-x_2) \, G_\theta(u-s) \,
	g_{\frac{u-s}{2}}(w-z) \,
	\frac{g_{t-u}(y-w)}{g_t(y-x_1)} \, \dd z\, \dd w \, \dd s \, \dd u \,,
\end{aligned}
\end{equation*}
where the integral is equal to (modulo some different constants as explained in Appendix~\ref{sec:comparison}) the covariance kernel $K^{(2)}(x_1, x_2)$ defined in \eqref{eq:Zmom2} and
illustrated in Figure \ref{fig:GMC-2},  if the factor $g_{t-u}(y-w)/g_t(y-x_1)$ was not present.\footnote{For consistency:
if we remove that factor, the r.h.s.\ becomes $1 + 4\pi \int_{0<s<u<t} g_{2s}(x_1-x_2)
\, G_\theta(u-s) \, \dd s \, \dd u$, which is consistent
with formula \eqref{eq:moments-SHF-SHE}
once we plug in $x_1\sqrt{2}$ and $x_2\sqrt{2}$;
see also \eqref{eq:variance-check}.} This factor is the conditional transition kernel from $(u, w)$ to $(t, y)$, originating from
the conditioning on $B_t=y$ in the definition of $u^{\eps}(t,x_1|y)$, while $(u,w)$ is the last time-space point of matching disorder between the chaos expansions of $u^{\eps}(t,x_1|y)$ and $u^{\eps}(t,x_2)$. This factor disappears if we average over the law of $y=B_t$.
Therefore
\begin{align} \label{eq:Phi2}
	& \widebar\Phi_{\delta}(x_1, y)
	\,:=\, \lim_{\eps\downarrow 0} \widebar\Phi_{\eps, \delta} (x_1, y) \\
	& \quad = 1+ 4\pi
	\idotsint\limits_{\substack{x_2, z \in \R^2 \\ 0<s<u<t}}\,
	\frac{g_{2\delta}(x_2) \, g_{s}(z-x_1) \, g_{s}(z-x_2)
	\, g_{t-\frac{u+s}{2}}(y-z)}{g_t(y-x_1)} \, G_\theta(u-s)
	\, \dd z\, \dd x_2 \, \dd s \, \dd u \,. \nonumber
\end{align}

We can now combine \eqref{eq:Phi1} and \eqref{eq:Phi2},
where $\widebar\Phi_{\delta}(x_1, y)$ and $\widebar \Phi_\delta$
are defined: if we define
\begin{equation}
	\Psi_\delta(x_1, y) := \frac{\widebar\Phi_{\delta}(x_1, y)}
	{\widebar \Phi_\delta} \,,
\end{equation}
then by Fatou's lemma we can bound
$$
	\liminf_{\eps\downarrow 0}  \bE\Bigg[\Bigg(
	\frac{\widebar\Phi_{\eps, \delta}
	(W^{(1)}_0, W^{(1)}_t)}{\bE[\widebar\Phi_{\eps, \delta}
	(W^{(1)}_0, W^{(1)}_t)]}\Bigg)^2\Bigg] \geq \bE\big[ \Psi_\delta
	(W^{(1)}_0, W^{(1)}_t)^2 \big] \,.
$$
It is easy to check that $\bE[\Psi_\delta(W^{(1)}_0, W^{(1)}_t)]=1$
(see \eqref{eq:checkEPsi} below).
Since $\Psi_\delta(W^{(1)}_0, W^{(1)}_t)$ is clearly not a constant,
it follows by Jensen's inequality that for any $\delta \in (0,1)$
\begin{equation*}
	\bE\big[\Psi_\delta(W^{(1)}_0, W^{(1)}_t)^2\big] > 1 \,.
\end{equation*}
Since $\delta \mapsto \bE[\Psi_\delta(W^{(1)}_0, W^{(1)}_t)^2]$ is continuous, to
prove \eqref{eq:mombd8} it only remains to show that
\begin{equation}\label{eq:Phi3}
	\lim_{\delta\downarrow 0} \, \bE\big[
	\Psi_\delta(W^{(1)}_0, W^{(1)}_t)^2 \big] \,>\, 1 \,.
\end{equation}

Denote $\widebar \Psi_\delta(W^{(1)}_t):=
\bE[\Psi_\delta(W^{(1)}_0, W^{(1)}_t) | W^{(1)}_t]$.
By $W^{(1)}_s = Z^{(1)}_{2\delta}+B^{(1)}_s$, we have
\begin{align*}
	\widebar \Psi_\delta(y) & = \frac{1}{\widebar \Phi_\delta}
	\int \widebar\Phi_{\delta}(x_1, y)  \,
	\frac{g_{2\delta}(x_1) g_t(y-x_1)}{g_{t+2\delta}(y)} \dd x_1 \\
	& = \frac{1}{\widebar \Phi_\delta}
	\bigg(1 + \frac{4\pi}{g_{t+2\delta}(y)}
	\idotsint\limits_{\substack{x_1, x_2, z \in \R^2 \\ 0<s<u<t}}
	g_{2\delta}(x_1) \, g_{2\delta}(x_2)
	\, g_{s}(z-x_1) \, g_{s}(z-x_2) \nonumber \\
	& \qquad\qquad\qquad\qquad\qquad\qquad\quad
	\times g_{t-\frac{u+s}{2}}(y-z) \, G_\theta(u-s)
	\, \dd z \, \dd x_1 \, \dd x_2 \, \dd s \, \dd u \bigg) \\
	& = \frac{1}{\widebar \Phi_\delta}
	\bigg(1 + \frac{4\pi}{g_{t+2\delta}(y)}
	\iiint\limits_{\substack{z \in \R^2 \\ 0<s<u<t}}
	g_{2\delta+s}(z)^2 \, g_{t-\frac{u+s}{2}}(y-z) \, G_\theta(u-s)
	\, \dd z \, \dd s \, \dd u \bigg) \,,
\end{align*}
and since $g_{2\delta+s}(z)^2 = g_{2(2\delta+s)}(0) \, g_{\delta+\frac{s}{2}}(z)$
by \eqref{eq:identity}, we obtain
\begin{align}
	\widebar \Psi_\delta(y)
	& = \frac{1}{\widebar \Phi_\delta}
	\bigg(1 + \frac{4\pi}{g_{t+2\delta}(y)} \, \iint\limits_{0<s<u<t}
	g_{2(2\delta+s)}(0) \, g_{t+\delta-\frac{u}{2}}(y) \, G_\theta(u-s)
	\, \dd s \, \dd u \bigg) \nonumber \\
	& = \frac{1}{\widebar \Phi_\delta}
	\bigg(1 + \iint\limits_{0<s<u<t}
	\frac{1}{2\delta+s} \,
	\frac{g_{t + \delta -\frac{u}{2}}(y)}{g_{t+2\delta}(y)}  \, G_\theta(u-s)
	\, \dd s \, \dd u \bigg) \,. \label{eq:Phi4}
\end{align}
Incidentally, this relation together with \eqref{eq:Phi1} shows that
\begin{equation}\label{eq:checkEPsi}
	\bE[\Psi_\delta(W^{(1)}_0, W^{(1)}_t)]
	= \bE[\widebar\Psi_\delta(W^{(1)}_t)]
	= \int\limits_{\R^2} \widebar\Psi_\delta(y) \,
	g_{t+2\delta}(y) \, \dd y = 1 \,.
\end{equation}

Note that as $\delta\downarrow 0$, the dominant contribution to the integral
in \eqref{eq:Phi1} for $\widebar \Phi_\delta$ comes from $s \ll 1$,
since we can restrict the integral to $s < (\log \frac{1}{\delta})^{-1}$ (say)
without changing the asymptotic behavior.
The same is true for the integral in \eqref{eq:Phi4},
hence we obtain
$$
	\lim_{\delta\downarrow 0} \, \widebar \Psi_\delta(y)
	= \widebar \Psi_0 (y) \,:=\, \frac{\int_0^t
	g_{t-\frac{u}{2}}(y) \,
	 G_\theta(u) \, \dd u}{g_t(y)
	 \, \int_0^t G_\theta(u) \, \dd u}
	 \,,
$$
which implies that $\widebar \Psi_\delta(W^{(1)}_t)
= \widebar \Psi_\delta(Z^{(1)}_{2\delta} + B^{(1)}_t)$ converges in law
to $\widebar \Psi_0(B^{(1)}_t)$ as $\delta\downarrow 0$.
Therefore, by Jensen's inequality and Fatou's lemma,
$$
	\lim_{\delta\downarrow 0} \bE\big[\Psi_\delta(W^{(1)}_0, W^{(1)}_t)^2\big]
	\,\geq\, \lim_{\delta\downarrow 0}
	\bE\big[ \widebar \Psi_\delta(W^{(1)}_t)^2 \big]
	\,\geq\, \bE\big[ \widebar \Psi_0(B^{(1)}_t)^2 \big] \,>\, 1 \,,
$$
where the last inequality holds because $\bE[\widebar \Psi(B^{(1)}_t)] =1$ and
$\widebar \Psi(B^{(1)}_t)$ is not a.s.\ equal to~$1$.
This concludes the proof of \eqref{eq:Phi3},
hence of Lemma~\ref{L:lowerbd}.
\end{proof}

\subsection{Proof of Proposition~\ref{P:momasy}}

The log-divergence of the second moment kernel $\scrK_t^{(2)}(x,y)$ of the SHF,
see \eqref{eq:rhot}, plays a crucial role.
Recall from  \eqref{eq:wefix} and \eqref{eq:Zmom2} that
$$
	\scrK_t^{(2)}(x,y) = \rme^{k_t(x,y)}= 1+ K^{(2)}_t(x,y)
	= 1+ 2\pi \, \iint\limits_{0 < s < u < t} g_{s}(x-y) \,
	G_\theta(u-s) \, \dd s \, \dd u \,,
$$
which is a monotonically decreasing function of $|x-y|$. By a change of variable,
$$
	2\pi \, \iint\limits_{0 < s < u < t} g_{s}(x-y) \,
	G_\theta(u-s) \, \dd s \, \dd u
	= \int_0^{t |x-y|^{-2}} \frac{\rme^{-\frac{1}{2\tilde s}}}{\tilde s} \,
	\bigg( \int_0^{t-|x-y|^2 \tilde s} G_\theta(v) \, \dd v \bigg)
	\, \dd \tilde s \,,
$$
and note that, as $|x-y|\downarrow 0$, the dominant
contribution to the integral comes
from the range of values $1\ll \tilde s \ll |x-y|^{-2}$. Therefore, as $|x-y|\downarrow 0$,
\begin{equation}\label{eq:ktasy}
	\scrK_t^{(2)}(x,y) = \rme^{k_t(x,y)}
	\sim \bigg(\int_0^t G_\theta(v) \dd v \bigg) \,
	\log\frac{t}{|x-y|^2}
	\sim C_{t,\theta} \,  \log\frac{1}{|x-y|} \,,
\end{equation}
where we set $C_{t,\theta} := 2 \int_0^t G_\theta(v) \, \dd v$.

Applying the moment formula \eqref{eq:Mmom3} and \eqref{eq:ktasy} to the l.h.s.\ of
\eqref{eq:momasy}, we find that as $\delta\downarrow 0$,
\begin{align*}
	\bbE\big[ \big( 2\, \scrM_t^\theta(g_\delta) \big)^m \big]
	& = \int\limits_{(\R^2)^m} \prod_{i=1}^m \, g_\delta (x_i) \
	\rme^{\sum_{1\leq i<j\leq m} k_t(x_i, x_j)} \, \dd \vec x \\
	& =  (1+o(1)) \, (C_{t,\theta})^{m\choose 2} \int\limits_{(\R^2)^m}
	\prod_{i=1}^m \, g_\delta (x_i) \, \prod_{1\leq i<j\leq m}
	\, \log \frac{1}{|x_i-x_j|} \, \dd \vec x.
\end{align*}
Via the change of variable $y_i = x_i / \sqrt{\delta}$,
the integral in the r.h.s.\ can be written as
$$
	\int\limits_{(\R^2)^m} \prod_{i=1}^m \, g_1 (y_i)
	\prod_{1\leq i<j\leq m}  \bigg(
	\log \frac{1}{\sqrt\delta} + \log \frac{1}{|y_i-y_j|} \bigg) \,
	\dd \vec y \,\sim\, \Big(\log \frac{1}{\sqrt{\delta}}\Big)^{m\choose 2},
$$
where the asymptotic equivalence as $\delta\downarrow 0$ follows by expanding the product and
noting the finiteness of the integrals. This shows that, as $\delta \downarrow 0$,
\begin{equation*}
	\bbE\big[ (2\, \scrM_t^\theta(g_\delta))^m \big]
	\sim \bigg( C_{t,\theta} \, \log \frac{1}{\sqrt{\delta}}\bigg)^{m\choose 2} ,
\end{equation*}
which proves \eqref{eq:momasy} and completes the proof of Proposition~\ref{P:momasy}.\qed

\appendix

\section{On the critical windows}
\label{sec:crit-wind}

In this secion,
we compare the critical windows for directed polymers and for the mollified
Stochastic Heat Equation.

\subsection{Directed polymer setting}

The critical scaling of
$\beta = \beta_N$ for the directed polymer partition functions \eqref{eq:paf}
is defined by the following asymptotic relation:
 \begin{equation} \label{intro:sigma}
	\sigma_N^2 := \rme^{\lambda(2\beta_N)-2\lambda(\beta_N)}-1
	= \frac{1}{R_N} \bigg(1 + \frac{\theta + o(1)}{\log N}\bigg) \,,\quad\quad
	\text{for some fixed } \theta\in \R \,,
\end{equation}
where $\lambda(\cdot)$ is the log-moment generating function of the
disorder, see \eqref{eq:lambda}, while
$R_N$ is the expected replica overlap of two
independent simple symmetric random walks $S, S'$ on $\Z^2$:
\begin{equation}\label{eq:RN0}
\begin{split}
	R_N &:= \E \bigg[ \sum_{n=1}^N \ind_{\{S_n = S'_n\}}\bigg]
	= \sum_{n=1}^N \sum_{z\in\Z^2} \P(S_n = z)^2
	= \sum_{n=1}^N \P(S_{2n}=0)  \\
	& = \sum_{n=1}^N \bigg\{ \frac{1}{2^{2n}} \binom{2n}{n}\bigg\}^2
	= \frac{\log N}{\pi} +\frac{\alpha}{\pi}+o(1) \qquad \text{as $N\to \infty$} \,,
\end{split}
\end{equation}
with $\alpha=\gamma+\log 16-\pi$ and 
$\gamma = - \int_0^\infty \log u \, \rme^{-u} \, \dd u \approx 0.577$ the Euler-Mascheroni constant.

Since $\lambda(\beta)\sim \tfrac{1}{2} \beta^2$
as $\beta \to 0$, it follows  from \eqref{intro:sigma} that
$\beta_N^2 \sim \pi/\log N$ as $N\to\infty$. The parameter $\theta \in \R$
tunes the higher order asymptotic behavior of $\beta_N$,
which also depends on the third and fourth cumulants
$\kappa_3, \kappa_4$ of the disorder:
see \cite[eq.~(1.17)]{CSZ19b} for the exact formula,
which simplifies
when $\kappa_3 = 0$ (e.g.\ for symmetric disorder distribution) and yields
\begin{equation}\label{eq:betasimple2}
	\begin{split}
	\beta_N^2 = \frac{\pi}{\log N}
	\bigg( 1 + \frac{\theta - c + o(1)}{\log N}\bigg)
	\qquad \text{where} \quad
	c &:= \alpha + \tfrac{1}{2}\pi + \tfrac{7}{12}\pi \kappa_4 \\
	&= \gamma + \log 16 - \tfrac{1}{2}\pi + \tfrac{7}{12}\pi \kappa_4 \,,
	\end{split}
\end{equation}
that is \eqref{eq:betasimple} holds with $\rho = \theta - c$.

\subsection{Stochastic Heat Equation setting}
\label{sec:crit-wind-SHE}

We next consider the Stochastic Heat Equation \eqref{eq:SHE0} with
mollified noise $\xi^\epsilon (t,x) = (\xi(t,\cdot) * j_\epsilon)(x)$,
where $j_\epsilon(x) := \epsilon^{-2} j(\epsilon^{-1}x)$.
The critical scaling $\beta = \beta_\epsilon$ is
(see \cite[eq.~(8.28)]{CSZ19b}):
\begin{equation}\label{eq:betaepssimple}
	\beta_\epsilon^2 = \frac{1}{\cR_\epsilon} \bigg(1+ \frac{\theta + o(1)}{\log \epsilon^{-2}}\bigg)
\end{equation}
where $\cR_\epsilon$ is defined as follows
(see \cite[Section~8.2]{CSZ19b}):
\begin{equation}\label{eq:cReps}
\begin{split}
	\cR_\epsilon &\,=\, \int_0^{\epsilon^{-2}} \bigg(
	\int_{(\R^2)^2} J(x) \, J(y)
	\, g_{2t}(x-y)  \, \dd x
	\, \dd y \bigg) \dd t \,.
\end{split}
\end{equation}
Note that we can view $\cR_\epsilon$ as the expected replica overlap
of two independent
Brownian motions $B, B'$ on $\R^2$ enlarged via $J := j * j$
into Wiener sausages, described by the functions $J_{B_t}(z) := J(z-B_t)$
and $J_{B'_t}(z) := J(z-B'_t)$:
\begin{equation}\label{eq:Reps}
\begin{split}
	\cR_\epsilon
	&= \int_0^{\epsilon^{-2}} \int_{(\R^2)^3}
	J(x) \, J(y) \, g_t(z-x) \, g_t(z-y)  \,  \dd x \, \dd y \, \dd z \, \dd t \\
	& = \int_0^{\epsilon^{-2}} \int_{\R^2}
	\bbE[J(z-B_t) \, J(z-B'_t)] \, \dd z \, \dd t
	= \bbE \bigg[ \int_0^{\epsilon^{-2}}
	\langle J_{B_t}, J_{B'_t} \rangle_{L^2(\R^2)}  \,  \dd t \bigg] \,.
\end{split}
\end{equation}

It was shown in \cite[end of Section~8.2]{CSZ19b} that
\begin{equation}\label{eq:cReps-as}
	\cR_\epsilon \,=\, \frac{\log\epsilon^{-2}}{4\pi} + \frac{C}{4\pi} + o(1)
	\qquad \text{as } \epsilon \downarrow 0 \,,
\end{equation}
where
\begin{equation} \label{eq:defCReps}
	C = 2 \int_{(\R^2)^2} J(x) \, \log\frac{1}{|x-y|} \, J(y) \, \dd x \, \dd y
	+ \log 4 - \gamma \,.
\end{equation}
Plugging this into \eqref{eq:betaepssimple} yields
\begin{align}\label{beta-eps2}
	\beta^2_\epsilon =
	\frac{4\pi}{\log \epsilon^{-2}}
	\bigg( 1 + \frac{\theta - C +o(1)}{\log\epsilon^{-2}} \bigg) \,.
\end{align}
that is \eqref{beta-eps} holds with $\rho = \theta - C$.

\subsection{Matching directed polymers with the Stochastic Heat Equation}
\label{sec:comparison}

In this appendix we explain heuristically relation \eqref{eq:conj}.

The Stochastic Heat Flow $\scrZ_t^\theta(\dd x)$ is the limit
of the directed polymer random measure
\begin{equation} \label{eq:point-to-plane}
	\cZ^{\beta_N}_{N;\, 0,t}(\dd x, \R^2) = \frac{1}{2} \,
	\sum_{z\in\Z^2_\even} Z_{0, \ev{Nt}}^{\beta,\,\omega}(\ev{\sqrt{N} x}, z) \,,
\end{equation}
see \eqref{eq:rescZmeas} and Theorem~\ref{thm:2dSHF}. We can then rewrite \eqref{eq:conj} as
\begin{equation} \label{eq:conj2special}
	\cZ^{\beta_N}_{N;\, 0,t}(\dd x, \R^2) 
	\! \underset{\ N = \epsilon^{-2}}{\overset{\rule[-.28em]{0pt}{.28em}d}{\approx}}\,
	\frac{1}{2} \, u^\epsilon_{\beta_\epsilon}\big( t,x\sqrt{2} \big) \, \dd x
	\qquad \text{as } \epsilon \downarrow 0 \,,
\end{equation}
where the disorder strengths in the two sides are
tuned in the respective critical windows,
see \eqref{intro:sigma} or \eqref{eq:betasimple2} for $\beta_N$
and \eqref{beta-eps2} for $\beta_\epsilon$, \emph{for the same value of~$\theta\in\R$}.

Relation \eqref{eq:conj2special} is expected to hold by comparing both sides to the same
coarse-grained model in \cite{CSZ23}. However, we
can simply explain the scaling factors in \eqref{eq:conj2special} by comparing the mean and covariance
of both sides:
\begin{itemize}
\item the multiplicative factor $\frac{1}{2}$ is due to the \emph{periodicity} 
of the simple random walk: 
we have indeed $\bbE[\cZ^{\beta_N}_{N;\, 0,t}(\dd x, \R^2)] = \frac{1}{2} \, \dd x$,
see \eqref{eq:point-to-plane}, while $\bbE[u^\epsilon_{\beta_\epsilon}(\cdot,\cdot)] \equiv 1$;
\item the factor $\sqrt{2}$ is because each random walk component
has \emph{variance}~$\frac{1}{2}$:
we have $\bbcov[\cZ^{\beta_N}_{N;\, 0,t}(\dd x, \R^2), \cZ^{\beta_N}_{N;\, 0,t}(\dd y, \R^2)]
\sim \frac{1}{4} \, K_t^{(2)}(x-y) \, \dd x \, \dd y$,
see \cite[Rem.~3.7]{CSZ23}, while
$\bbcov[u^\epsilon_{\beta_\epsilon}(t,x), u^\epsilon_{\beta_\epsilon}(t,y)] \sim 
K_t^{(2)}(\frac{x-y}{\sqrt{2}})$,
see \cite[eq.~(3.14)]{BC98}, \cite[Thm.~1.9]{CSZ19b}.
\end{itemize}

\smallskip

We now give a heuristic derivation of relation \eqref{eq:conj2special}.
Let $(S_n)$ be a $\sfT$-periodic random walk on $\Z^2$
(i.e.\ $S_{n}$
takes values in a sub-lattice $\bbT_n \subset \Z^2$
whose cells have area~$\sfT$) with covariance matrix $\sfs I$.
For the simple random walk we have $\sfs = \frac{1}{2}$ and $\sfT = 2$ with
\begin{equation}
	\bbT_n := \begin{cases}
	\Z^2_\even & \text{ for $n$ even} \,, \\
	\Z^2_\odd := \Z^2 \setminus\Z^2_\even
	& \text{ for $n$ odd} \,,
	\end{cases}
\end{equation}
see \eqref{eq:even}. 
The parameters $\sfs$ and $\sfT$ enter in the local limit theorem: recalling that
$g_t(x)$ denotes the heat kernel, see \eqref{eq:gt}, we have as $n\to\infty$
\begin{equation}\label{eq:llt}
	\P(S_{n}=z) = \big( g_{\sfs n}(z) + o(n^{-1}) \big)
	\sfT \, \ind_{\bbT_n}(z) \,.
\end{equation}
We insist on the use of general parameters $\sfs$ and $\sfT$, instead of the particular
values $\frac{1}{2}$ and $2$, because the following arguments become more transparent.

\smallskip

The solution $u^\epsilon_\beta(t,x)$ of the mollified Stochastic Heat Equation \eqref{eq:SHE0}
can be viewed as the partition function for a \emph{Brownian directed polymer $B$
in a mollified white-noise random environment~$\xi^\epsilon$},
comparing \eqref{eq:paf} with the Feynman-Kac representation formula \eqref{eq:SHE2}.
To account for the random walk variance~$\sfs$ and periodicity~$\sfT$, we can
modify \eqref{eq:SHE2} as follows:
\begin{itemize}
\item we replace $(B,x)$ by $(\sqrt{\sfs}\, B,\,  x/\sqrt{\sfs})$
to get a Brownian motion with variance~$\sfs$ started at~$x$
and, accordingly, we replace the mollified white noise $\xi^\epsilon$
by $\xi^{\sqrt{\sfs} \,\epsilon}$;
\item we replace $\beta$ by $\sqrt{\sfT} \beta$; this ensures that 
computing the variance $u^\epsilon_\beta(t,x)$ as a power series in $\beta^2$,
arising from the polynomial chaos expansion  (see e.g.\ \cite[eq.~(8.12)]{CSZ19b}),
each heat kernel is multiplied by~$\sfT$, matching the local limit theorem~\eqref{eq:llt}.
\end{itemize}
Overall, since $\sqrt{T} \beta \, \xi^{\sqrt{\sfs}\,\epsilon}(t-u, \sqrt{\sfs}\, B_u)$
has the same distribution as $\sqrt{T/\sfs}\, \beta \, \xi^{\epsilon}(t-u, B_u)$,
we can simply \emph{modify the Feynman-Kac formula \eqref{eq:SHE2}
replacing $x$ by $x/\sqrt{\sfs}$ and $\beta$ by $\sqrt{T/\sfs}\, \beta$}.

\smallskip

Summarizing, for the directed polymer random measure $\cZ^{\beta_N}_{N;\, 0,t}(\dd x, \R^2)$
defined in analogy with \eqref{eq:point-to-plane}, with $\frac{1}{2}$ replaced by $\frac{1}{\sfT}$
and $\Z^2_\even$ replaced by $\bbT_0$, we expect that
\begin{equation} \label{eq:conj2}
	\cZ^{\beta_N}_{N;\, 0,t}(\dd x, \R^2)
	\! \underset{\ N = \epsilon^{-2}}{\overset{\rule[-.28em]{0pt}{.28em}d}{\approx}}\,
	\tilde u(t, x) \, \dd x \qquad
	\text{with} \qquad \tilde u(t,x) := 
	\frac{1}{\sfT} \,
	u^\epsilon_{\sqrt{\frac{\sfT}{\sfs}}\beta_N}
	\big(t,\tfrac{x}{\sqrt{\sfs}}\big) \,.
\end{equation}
For $\sfs=\tfrac{1}{2}$ and $\sfT = 2$,
this equation is ``close'' to \eqref{eq:conj2special}
since $\sqrt{\sfT/\sfs}\, \beta_N = 2 \, \beta_N \sim \beta_\eps$, 
cf.\ \eqref{eq:betasimple2} and \eqref{beta-eps2}.
For an accurate comparison, we should replace $\sqrt{\sfT/\sfs} \, \beta_N$
by $\beta_\epsilon$ in the definition of $\tilde u(t,x)$ in \eqref{eq:conj2},
which leads to \eqref{eq:conj2special}.

\smallskip

Finally, we note that $\tilde{u}(t,x)$ from \eqref{eq:conj2}
solves a mollified Stochastic Heat Equation
with adjusted coefficients, to account for the
random walk variance~$\sfs$ and periodicity~$\sfT$:
\begin{equation}\label{eq:SHE-mod}
	\begin{cases}
	\rule[-1.2em]{0pt}{2.5em}\displaystyle
	\partial_t \tilde u(t,x)
	= \frac{\sfs}{2} \, \Delta \tilde u(t,x) +
	\sqrt{\sfs} \, \sqrt{\tfrac{\sfT}{\sfs}} \, \beta_N \, \tilde u(t,x) \,
	\xi^{\eps \sqrt{\sfs}}(t,x) \\
	\rule[-.8em]{0pt}{2.2em}\displaystyle
	\tilde u(0,\cdot) \equiv \frac{1}{\sfT}
	\end{cases} \,,
\end{equation}
where $\tilde \xi^{a}(t,x) := \frac{1}{\sqrt{\sfs}} \, \xi^{a/\sqrt{\sfs}}(t, \frac{x}{\sqrt{\sfs}})$
has the same distribution as~$\xi^a(t,x)$. Again, for an accurate comparison with directed polymers,
we should replace $\sqrt{\sfT/\sfs} \, \beta_N$ in \eqref{eq:SHE-mod}
by $\beta_\epsilon$ from \eqref{beta-eps2}.

\section*{Acknowledgements}
We wish to thank Christophe Garban for asking us the question that led to the present paper, and
for interesting discussions during our visit to Lyon. We also thank Jeremy Clark for informing us of
his ongoing work \cite{CM22}. R.S.~is supported by NUS grant R-146-000-288-114. N.Z.~is supported
by EPRSC through grant EP/R024456/1, F.C.~is supported by INdAM/GNAMPA.

The completion of this work coincided with the passing of Francis Comets. We wish to express our deepest
gratitude to Francis and dedicate this work to him for the inspiration he provided us over the years.

\end{document}